\documentclass[11pt]{article}

\usepackage{enumerate}
\usepackage{mathrsfs}
\usepackage{amsmath, color}
\usepackage{latexsym}
\usepackage{amssymb}
\usepackage{amsthm}
\usepackage{amsfonts}
\usepackage{dsfont}
\usepackage{hyperref}
\usepackage{inputenc}

\newtheorem{dfn}{Definition}[section]
\newtheorem{thm}[dfn]{Theorem}
\newtheorem{rmk}[dfn]{Remark}

\newtheorem{prop}[dfn]{Proposition}

\newtheorem{lem}[dfn]{Lemma}
\theoremstyle{definition}

\numberwithin{equation}{section}
\renewcommand\theenumi\arabicenumi \renewcommand\labelenumi\rm

\newcommand{\IE}{{\mathbb{E}}}
\newcommand{\IP}{{\mathbb{P}}}
\newcommand{\IR}{{\mathbb{R}}}
\newcommand{\FF}{{\mathcal{F}}}
\newcommand{\EE}{{\mathcal{E}}}
\newcommand{\IQ}{{\mathbb{Q}}}

\def\eps{\varepsilon}
\def\wh{\widehat}
\def\wt{\widetilde}
\def\<{\langle}
\def\>{\rangle}

\title{\bf Brownian Motion with Drift on Spaces with Varying Dimension}
\date{\today}

\author{Shuwen Lou}

\begin{document}

\maketitle
\begin{abstract}
\noindent Many properties of Brownian motion on spaces with varying dimension (BMVD in abbreviation) have been explored in \cite{CL}.  In this paper, we study Brownian motion with drift on spaces with varying dimension (BMVD with drift in abbreviation).  Such a process can be conveniently defined by a regular Dirichlet form that is not necessarily symmetric. Through the method of Duhamel's principle, it is established in this paper that  the transition density of BMVD with drift has the same type of  two-sided Gaussian  bounds as that for BMVD (without drift). As a corollary,  we derive Green function estimate for BMVD with drift.
\end{abstract}

\medskip
\noindent
{\bf AMS 2010 Mathematics Subject Classification}: Primary 60J60, 60J35; Secondary 31C25, 60H30, 60J45

\smallskip\noindent
{\bf Keywords and phrases}: Space of varying dimension, Brownian motion,  Laplacian,  singular drift,  transition density function, 
heat kernel estimates,  Green function 

\section{Introduction}\label{S:1}
Brownian motion on spaces with varying dimension has been introduced and studied in details with an emphasis on its transition density estimate in \cite{CL}.
Such a process is an interesting example of Brownian motion on non-smooth spaces  and can be characterized nicely via Dirichlet form. The state space of BMVD, with or without drift, looks like a plane with a vertical half line installed on it. Roughly speaking, it is ``embedded" in the following space:
\begin{equation*}
\IR^2 \cup \IR_+ =\{(x_1, x_2, x_3)\in \IR^3: x_3=0 \textrm{ or } x_1=x_2=0 \hbox{ and } x_3>0\}.
\end{equation*}
As has been pointed out in \cite{CL}, Brownian motion cannot be defined on such a state space in the usual sense because a two-dimensional Brownian motion does not hit a singleton. Therefore in order to define the desired process,  in \cite{CL}  we ``short" a closed disc on $\IR^2$ to a singleton. In other words, we let the media offer infinite conductance on this closed disc, so that  the process travels across it at infinite velocity. The resulting Brownian motion hits the shorted disc in finite time with probability one. Then we install an infinite pole at this ``shorted" disc.  

To be more precise, the state space of BMVD on $E$ is defined as follows.
Fix $\eps>0$ and $p>0$.
Denote by  $B_\eps $  the closed disk on $\IR^2$ centered at $(0,0)$ with radius $\eps $. Let 
${D_\eps}:=\IR^2\setminus  B_\eps $. By identifying $B_\eps $ with a singleton denoted by $a^*$, we can introduce a topological space $E:={D_\eps}\cup \{a^*\}\cup \IR_+$, with the origin of $\IR_+$ identified with $a^*$ and
a neighborhood of $a^*$ defined as $\{a^*\}\cup \left(V_1\cap \IR_+ \right)\cup \left(V_2\cap {D_\eps}\right)$ for some neighborhood $V_1$ of $0$ in $\IR^1$ and $V_2$ of $B_\eps $ in $\IR^2$. Let $m_p$ be the measure on $E$ whose restriction on $\IR_+$ and ${D_\eps}$ is the Lebesgue measure multiplied by $p$ and $1$, respectively.
In particular, we have $m_p (\{a^*\} )=0$. Note that the measure $m_p$ also depends on $\eps$, the radius of the ``hole" $B_\eps$. 

In particular, we set $m_p\left(\{a^*\}\right)=0$.
The following definition for BMVD can be found in \cite{CL}.
\begin{dfn}[Brownian motion with varying dimension]\label{def-bmvd}An $m_p$-symmetric diffusion process satisfying the following properties is called Brownian motion with varying dimension.

\begin{description}
\item{\rm (i)} its part process in $\IR_+$ or $D_\eps$ has the same law as standard Brownian motion in $\IR_+$ or $D_\eps$;
\item{\rm (ii)} it admits no killings on $a^*$;
\end{description}
\end{dfn}
We denote BMVD (without drift) by $X^0$. It follows from the definition that the process spends zero amount of time under Lebesgue measure (i.e. zero sojourn time) at $a^*$. The following theorem is a restatement of   \cite[Theorem 2.2]{CL}, which  asserts that given  $\eps>0, p>0$,  BMVD exists and is unique in law. It also gives the Dirichlet form characterization for $X^0$.

\begin{thm}\label{BMVD-non-drift}
For every $\eps >0$ and $p>0$,
BMVD $X^0$ on $E$ with parameter $(\eps, p)$ exists and is unique.
Its associated Dirichlet form $(\EE^0, \mathcal{D}(\EE^0))$ on $L^2(E; m_p)$ is given by
\begin{eqnarray*}
\mathcal{D}(\EE^0) &= &  \left\{f: f|_{D_\eps}\in W^{1,2}(D_\eps),  \, f|_{\IR_+}\in W^{1,2}(\IR_+),
\hbox{ and }
f (x) =f (0) \hbox{ q.e. on } {\partial D_\eps}\right\},  
\\
\EE^0(f,g) &=& \frac12 \int_{D_\eps}\nabla f(x) \cdot \nabla g(x) dx+\frac{p}2\int_{\IR_+}f'(x)g'(x)dx . 
\end{eqnarray*}
\end{thm}

In this paper, we study Brownian motion with singular drift on spaces with varying dimension which is a natural generalization  of BMVD. We seek to establish the stability of the heat kernel asymptotic behavior under gradient perturbation.  To give a brief overview on  the classic  results for   Brownian motion with drift on Euclidean spaces, we restate the following definition for Kato class functions (see, for instance, \cite{ChenZ}):

\begin{dfn}[Kato class functions]\label{def-kato-class}
On $\IR^n$, for $d\in \mathbb{N}$, we say a function $h: \IR^n\rightarrow \IR$ is in Kato class $\mathbf{K}_{d}$ if
\begin{equation*}
\lim_{r\downarrow 0}\sup_{x\in \IR^n} \int_{|x-y|<r} \frac{\left|h(y)\right|}{|x-y|^{d-2}}dy=0, \quad \text{for }d\ge 3,
\end{equation*}
\begin{equation*}
\lim_{r\downarrow 0}\sup_{x\in \IR^n}\int_{|x-y|<r}\log\left(|x-y|^{-1}\right)\left|h(y)\right|dy=0, \quad \text{for }d=2,
\end{equation*}
and
\begin{equation*}
\sup_{x\in \IR^n}\int_{|x-y|\le 1}\left|h(y)\right|dy<\infty, \quad \text{for }d=1.
\end{equation*}
\end{dfn}
\begin{rmk}
We point out that in Definition \ref{def-kato-class}, it is \emph{not}  necessary that $n=d$. Also, the definition of $\mathbf{K}_d$ depends on $n$, the dimension of the domain of   the function. 
\end{rmk}
On $\IR^d$, Brownian motion with drift can be characterized by its associated generator $\mathcal{L}^b=\frac{1}{2}\Delta+b\cdot \nabla$ where $b\in \{h:\IR^d\rightarrow \IR^d, \, |h|\in \mathbf{K}_{d+1} \; \text{or } |h|^2\in \mathbf{K}_{d}\}$. It was established by Chen and Zhao in \cite{ChenZ} that when $|b|^2\in \mathbf{K}_{d}$, the bilinear form associated with $\mathcal{L}^b=\frac{1}{2}\Delta+b\cdot \nabla$ on $C^\infty_c(\IR^d)$ is lower semibounded, closable, Markovian and satisfies Silverstein's sector condition. Therefore there is a minimal diffusion process associated with this bilinear form.  Bass and Chen claimed in \cite{BC} that when $|b|\in \mathbf{K}_{d+1}$, there is a unique weak solution to the SDE:
\begin{equation*}
dY_t=dB_t+b(Y_t)dt, \qquad Y_0=y_0.
\end{equation*}
Furthermore, such a solution is a strong Markov process associated to the  generator $\mathcal{L}^b=\frac{1}{2}\Delta+b\cdot\nabla$. Indeed, on $\IR^d$, $L^{p}\subset \left(\mathbf{K}_{d+1}\cap\left\{b: |b|^2\in \mathbf{K}_{d}\right\}\right)$ for all $p>d$.

The sharp two-sided bounds for the transition densities of  Brownian motion with drift on Euclidean spaces have been extensively studied  over   the past few decades. It was first  shown by Aronson that the transition density $p(t,x,y)$ of Brownian motion with drift  on $\IR^d$ has the following two-sided Gaussian-type bounds, provided that $b\in L^{p}(B(0, R))$ for some $p>d$ and $R>0$, and $b$ is bounded outside $B(0, R)$:
\begin{equation*}
\frac{C_1}{t^{d/2}}\exp\left(-\frac{C_2|x-y|^2}{t}\right) \le p(t,x,y) \le \frac{C_3}{t^{d/2}}\exp\left(-\frac{C_4|x-y|^2}{t}\right), \quad 0<t\le 1.
\end{equation*}
Later it was developed  by Zhang  in \cite{Z1} that Aronson-type heat kernel two-sided bounds hold provided that $b$ satisfies some integral conditions. It was also claimed in \cite{Z1} that Kato class $\mathbf{K}_{d+1}$ implies the same integral conditions. In \cite[Proposition 2.3]{KS}, it was justified  that these integral conditions are indeed equivalent to Kato class $\mathbf{K}_{d+1}$. See also Riahi \cite{Riahi}.

To give definition to BMVD with drift, we first give definition to the class of functions $L^{p_1, p_2}(E)$. 
\begin{dfn}
Given $p_1\in (1, \infty]$ and $p_2\in (2,\infty]$, we  say a measurable function $b: E\rightarrow \IR$ is in $ L^{p_1, p_2}(E)$  if  
\begin{description}
\item{\rm (i)} $b|_{\IR_+}\in L^{p_1}(\IR_+)$;
\item{\rm (ii)} $b|_{D_\eps\cup \{a^*\}}\in L^{p_2}(D_\eps\cup \{a^*\})$.
 \end{description}
\end{dfn}
Given that $b$ takes value in $\IR$, for notation convenience, in this paper we define
\begin{equation*}
b\cdot \nabla f:= b(x)\frac{\partial f}{\partial x_1}(x)+b(x)\frac{\partial f}{\partial x_2}(x), \quad \text{for }x\in D_\eps, \, x=(x_1, x_2). 
\end{equation*}
For every $b\in L^{p_1, p_2}(E)$, we set $b_1:=b|_{\IR_+}$ and $b_2:=b|_{D_\eps\cup \{a^*\}}$.  It follows that for every pair of positive constants $(\eps, p)$ and every $b\in L^{p_1, p_2}(E)$, 
\begin{equation}\label{Dirichlet-drifted-BMVD}
\EE^b\left(f,\;g\right)=\EE^0\left(f,\;g\right)-\left(b\cdot \nabla f,\; g\right),\quad \text{for }f, g\in  \mathcal{D}(\EE^b)=\mathcal{D}(\EE^0),
\end{equation}
is a strongly local regular Dirichlet space (which in general is non-symmetric), therefore there is a unique diffusion process associated with it.  
\begin{dfn}[Brownian motion with drift  on space  with varying dimension]\label{Def-driftedBMVD}  Let $\eps >0$ and $p>0$ be fixed, and let $b\in L^{p_1, p_2}(E)$ with $p_1\in (1, +\infty]$ and $p_2\in (2, +\infty]$.   The diffusion process   associated with $(\EE^b, \mathcal{D}(\EE^b))$  is called   Brownian motion with drift  on space   with varying dimension   and  is denoted by $X$.
\end{dfn}
In this paper, we restrict ourselves to  the cases where the   drift function $b$  is  scalar-valued, this is   because the proofs rely substantially on H\"{o}lder's inequality involving $b$. 

Unlike  $X^0$, BMVD  with drift  $X$ in general is non-symmetric.  
The major goal of this paper is to establish  sharp two-sided bounds for the transition density function of $X$.  Given that in general $X$ is not rotationally-invariant on $D_\eps$,  we can not employ the method of radial process which has been used in \cite{CL}  to study $X^0$. Instead, in this paper, via  Duhamel's principle,  we identify both  the transition density and the resolvent operators of $X$ as the transformations of those for   $X^0$. Then we derive the upper bound estimates for the  heat kernel  of $X$. The lower bound estimate, on the other hand, is derived via the combination of short-time  near-diagonal lower bound estimates and a chain argument. 

To state the results of this paper, we need to introduce a few more notations.
In this paper, we denote the geodesic metric on the underlying space $E$ by $\rho$.
Namely, for $ x,y\in E$, $\rho (x, y)$ is the shortest path distance (induced
from the Euclidean space) in $E$ between $x$ and $y$.
For notation simplicity, we write $|x|_\rho$ for $\rho (x, a^*)$ when $x\in D_\eps$.
We use $| \cdot |$ to denote the usual Euclidean norm. For example, for $x,y\in D_\eps$,
$|x-y| $ is  the Euclidean distance between $x$ and $y$ in $\IR^2$.
Note that for $x\in D_\eps$, $|x|_\rho =|x|-\eps$.
 Apparently,
\begin{equation}\label{e:1.1}
\rho (x, y)=|x-y|\wedge \left( |x|_\rho + |y|_\rho \right)
\quad \hbox{for } x,y\in D_\eps,
\end{equation}
and $\rho (x, y)= |x|+|y|-\eps$ when $x\in \IR_+$ and $y\in D_\eps$ or vice versa.
Here and in the rest of this paper, for $a,b\in \IR$, $a\wedge b:=\min\{a,b\}$.

For any compact set $K\subset E$, we define $\sigma_K:=\inf\{t>0, X_t\in K\}$. For any open domain $D\subset E$, we define $\tau_{D}:=\inf\{t>0:  X_t \notin D\}$.   Similar notations will be used for other stochastic processes. Also we set $\delta_D(x):=\inf\{d(x,y): \, y\in D^c\}$, where $d(\cdot, \cdot)$ stands for either Euclidean or geodesic distance according to the context.

For two positive functions $f$ and $g$, $f\asymp g$ means that $f/g$ is bounded between two positive constants.
In the following, we will also use  notation $f\lesssim g$ (respectively, $f\gtrsim g$) to mean that there is some constant $c>0$ so that $f\leq c g$
(respectively, $f\geq c g$).

The following theorem is a   restatement of   \cite[Theorem 1.3]{CL}, which establishes  the  two-sided short-time heat kernel bounds for BMVD without drift $X^0$.  By $p(t,x,y)$ we   denote the transition density of $X^0$.
\begin{thm}\label{smalltime}
Let $T>0$ be fixed. There exist positive constants $C_i$, $1\le i\le 14$ so  that the 
   transition density $p(t,x,y)$ of BMVD $X^0$  satisfies the following estimates when $t\in (0, T]$:
\begin{description}
\item{\rm (i)} For $x \in \IR_+$ and $y\in E$, 
\begin{equation*}
\frac{C_1}{\sqrt{t}}e^{-C_2\rho (x, y)^2/t }  \le p(t,x,y)\le\frac{C_3}{\sqrt{t}}e^{- C_4\rho (x, y)^2/t}.
\end{equation*}

\item{\rm (ii)} For $x,y\in {D_\eps}\cup \{a^*\}$, when $|x|_\rho+|y|_\rho<1$,
\begin{eqnarray*}
&& \frac{C_{5}}{\sqrt{t}}e^{- C_{6}\rho(x,y)^2/t}+\frac{C_{5}}{t}\left(1\wedge
\frac{|x|_\rho}{\sqrt{t}}\right)\left(1\wedge
\frac{|y|_\rho}{\sqrt{t}}\right)e^{- C_{7}|x-y|^2/t} \nonumber
\\
 &\le  & p(t,x,y)
\le \frac{C_{8}}{\sqrt{t}}e^{-C_{9}\rho(x,y)^2/t}
+\frac{C_{8}}{t}\left(1\wedge
\frac{|x|_\rho}{\sqrt{t}}\right)\left(1\wedge
\frac{|y|_\rho}{\sqrt{t}}\right)e^{-C_{10}|x-y|^2/t}; \nonumber
\end{eqnarray*}
and when $|x|_\rho+|y|_\rho\geq 1$,
\begin{equation*}
\frac{C_{11}}{t}e^{- C_{12}\rho(x,y)^2/t} \le p(t,x,y) \le \frac{C_{13}}{t}e^{- C_{14}\rho(x,y)^2/t}.
\end{equation*}
 \end{description}
\end{thm}

The above theorem has covered all the cases for $x, y\in E$,  owing to  the fact that  $p(t,x,y)$ is symmetric in $(x, y)$, since $X^0$ is a symmetric Markov process.

To state the main result of this paper, for each  $\alpha>0$, we  first  define the canonical form $p^0_{0, \alpha}(t,x,y)$ of the transition density of $X$  as follows:
\begin{equation}\label{e:1.3}
p^0_{0, \alpha}(t,x,y):=\left\{
  \begin{aligned}
   & \frac{1}{\sqrt{t}}e^{-\alpha |x-y|^2/t},\quad 
 x\in \IR_+ ,y\in \IR_+;&
\\ 
& \frac{1}{\sqrt{t}}e^{-\alpha (|x|^2+|y|_\rho^2)/t},\quad 
x\in \IR_+,y\in D_\eps \cup\{a^*\} ;& 
\\ 
& \frac{1}{\sqrt{t}}e^{-\alpha (|x|_\rho^2+|y|^2)/t},\quad
x\in D_\eps \cup\{a^*\}, y\in \IR_+;& 
\\
 & \frac{1}{t}e^{-\alpha |x-y|^2/t}, \quad  x,y\in D_\eps \cup \{a^*\} \text{ with } \max\{|x|_\rho,|y|_\rho\}> 1;& 
\\ 
&   \frac{1}{\sqrt{t}}e^{-\alpha(|x|^2_\rho+|y|_\rho^2)/t}+\frac{1}{t}\left(1\wedge \frac{|x|_\rho}{\sqrt{t}}\right)\left(1\wedge \frac{|y|_\rho}{\sqrt{t}}\right)e^{-\alpha |x-y|^2/t},\quad \text{otherwise}. & 
  \end{aligned}
\right.
\end{equation}
\begin{rmk}
By comparing $p(t,x,y)$ in  Theorem \ref{smalltime} with \eqref{e:1.3}, one can see that for every $T>0$, there exist $\alpha_1, \alpha_2>0$ such that
\begin{equation}\label{RMK:1.6}
p_{0,\alpha_1}^0(t,x,y)\lesssim p(t,x,y)\lesssim p^0_{0,\alpha_2}(t,x,y) \quad \text{for all } 0<t\le T, x,y \in E.
\end{equation}
For example, it is clear that $\rho(x,y)^2=(|x|+|y|_\rho)^2\asymp |x|^2+|y|_\rho^2$  for $x\in \IR_+$ and $y\in D_\eps$; and it follows from elementary geometry that  $\rho(x,y)\asymp  |x-y|$  for  $x,y\in D_\eps$ when   $\max\{|x|_\rho, |y|_\rho \}>1$. 
We rewrite the two-sided bounds for the transition density of BMVD only for  computation convenience in this article. 
\end{rmk}
The main result of this paper is the following:
\begin{thm}\label{T:smalltime}
Let $T>0$ be fixed. Let $b$ be a  measurable function  in $L^{p_1, p_2}(E)$ with  $p_1\in (1, \infty]$ and $p_2\in (2,\infty]$.  The transition density of BMVD with drift  $X$ associated with \eqref{Dirichlet-drifted-BMVD} exists.  We denote   the transition density of $X$  by $p^b(t,x,y)$.  There exist constants $C_{15},  C_{16}>0$,  and  $0<\alpha_1<\alpha_2$ such that 
\begin{equation*}
C_{15}p^0_{0,\alpha_2}(t,x,y)\le p^b(t,x,y)\le C_{16}p^0_{0,\alpha_1}(t,x,y), \qquad (t,x,y)\in (0,T]\times E\times E.
\end{equation*}
\end{thm}

Recall that an open set $D\subset \IR^d$ is called $C^{1,1}$ if there exist a localization radius $R_0>0$ and a constant $\Lambda_0>0$ such that for every $z\in \partial D$, there exist a $C^{1,1}-$function $\phi=\phi_z: \IR^{d-1}\rightarrow \IR$ satisfying $\phi(0)=0$, $\nabla \phi(0)=(0,\dots, 0)$, $\|\nabla \phi\|_\infty \le \Lambda_0$, $|\nabla \phi (x)-\nabla \phi (z)|\le \Lambda |x-z|$ and an orthonormal coordinate system $CS_z:\, y=(y_1, \dots , y_d):=(\tilde{y}, y_d)$ with its origin at $z$ such that
\begin{equation*}
B(z, R_0)\cap D=\{y\in B(0, R_0)\textrm{ in }CS_z: y_d>\phi (\tilde{y})\}.
\end{equation*}
For the state space $E$ in this paper, an open set $D\subset E$ will be called $C^{1,1}$ in $E$, if $D\cap \IR_+$ is a $C^{1,1}$ open set in $\IR_+$, and $D\cap \left(D_\eps\cup \{a^*\}\right)$ is also a $C^{1,1}$ open set in $D_\eps\cup \{a^*\}$.

Let $D$ be a bounded $C^{1,1}$ domain of $E$ containing $a^*$.   We denote by $X^{D}$ the part process of $X$ killed upon exiting  domain $D$ of $E$, and  let   $p^b_D(t,x,y)$ be the transition density of $X^{D}$. Define the Green function of $X^D$  as
\begin{equation*}
G^b_{D}(x,y):=\int_{0}^\infty p^b_{D}(t,x,y)dt.
\end{equation*}
 The following theorem   provides   two-sided bounds for the Green function $G_D^b$.   Recall for any bounded open set $U\subset E$,   $\delta_U(\cdot):=\rho(\cdot, \partial U)$  denotes the $\rho$-distance to
 the boundary $\partial U$.   For notation convenience, we set ${U_1}:=(D\cap \IR_+)\setminus \{a^*\}$ and ${U_2}:=D\cap {D_\eps}$.
 Note that  $a^*\in \partial {U_1} \cap \partial {U_2}$ and $U_1=(0, l)$ for some $l>0$.

\begin{thm} \label{T:1.10} 
Let $b$ be a measurable function in    $L^{p_1, p_2}(E)$ with $p_1\in (1, \infty]$ and $p_2\in (2,\infty]$.  Suppose  $D$ is a connected bounded $C^{1,1}$ domain of $E$ containing   $a^*$.   Let $G^b_{D}(x,y)$ be the Green function of $X$ killed upon exiting $D$, then we have for $x\not= y$ in $D$ that 
\begin{equation*}
G^b_{D}(x,y)\asymp\left\{
  \begin{array}{ll}
    \delta_D(x)\wedge \delta_D(y),
&
\hbox{$x\in U_1 \cup \{a^*\}$, $y\in U_1 \cup \{a^*\}$;}
\\ \\
  \delta_D(x)  \delta_D(y)+\ln\left(1+\frac{\delta_{U_2}(x)\delta_{U_2}(y)}{|x-y|^2}\right),
& \hbox{$x\in U_2$, $y\in U_2$;}
\\ \\
  \delta_D(x) \delta_D(y),
& \hbox{$x\in U_1\cup \{a^*\}$, $y\in U_2$ or vice versa.}
  \end{array}
\right.
\end{equation*}
\end{thm}
For notation convenience, given a domain $D$  of $E$,   in this paper we set
\begin{equation}\label{e:1.6}
\bar{p}_D(t,x,y):=p(t,x,y)-p_D(t,x,y),
\end{equation}
where $p_D(t,x,y)$ is the transition density of the part process  $(X^0)^D$  killed upon exiting $D$.
In other words, for any non-negative function $f\geq 0$ on $E$,
\begin{equation}\label{e:1.7}
 \int_E  \bar{p}_D (t, x, y)  f(y) m_p(dy) = \IE_x \left[ f(X^0_t); t\geq \tau_D \right].
\end{equation}
The intuition for   $p_D(t,x,y)$ and $\bar{p}_D(t,x,y)$ is that for $p_D(t,x,y)$, the trajectory starting from $x$ hits $y$ at time $t$ without exiting $D$, while for $\bar{p}_D(t,x,y)$, the trajectory has to exit $D$ before reaching  $y$ at time $t$. In the same manner, for BMVD with drift $X$,
\begin{equation*}
\bar{p}^b_D(t,x,y):=p^b(t,x,y)-p^b_D(t,x,y).
\end{equation*}

We use $C^\infty_c(E)$ to denote the space of continuous functions with compact support in $E$ so that
their restrictions to ${D_\eps}$ and $\IR_+$ are smooth on $\overline{D}_\eps$ and on $\IR_+$, respectively.
 We also follow the convention that in the statements of the theorems or propositions $C, \alpha, \beta, \alpha_1, \beta_1, C_1, \cdots$ denote positive constants, whereas  the exact values of the positive constants appearing in their proofs are  unimportant and may change
 from line to line.

The remaining of this paper is organized as follows: Let  $\IP$ be the distribution law of $X^0$.  A  family of measures $\IQ$  can be {\it  formally} defined   through the following Girsanov transform:
\begin{equation}\label{Girsanov}
\frac{d\IQ}{d\IP}\bigg|_{\FF_t}=M_t:=\exp\left(\int_0^t b(X^0_s)dX^0_s-\int_0^t |b(X^0_s)|^2 ds \right),
\end{equation}
where  $\FF_t=\sigma\{X^0_s, s\le t\}$.  In Section \ref{S:2}, we present the rigorous definition for $\IQ$ and verify that $(X, \IQ)$ and $(X^0, \IP)$ are different representations of the same process. We also show in \ref{S:2}  that the  infinitesimal  generator of $X$  is indeed the infinitesimal generator of $X^0$ under gradient perturbation.  However, any function in the domain of the   generator of $X$ has to satisfy the ``zero flux" condition at the darning point $a^*$. The resolvents of $X$  can also be represented in terms of the resolvents of $X^0$ via Duhamel's principle. In Section \ref{S:3}, we establish  two-sided short time heat kernel estimates for $X$. The upper bound estimates  are  established combining the smallness of the gradients of the transition densities of $X^0$ and  Duhamel's principle. We show that the transition density function of $X$ has the same type  of upper bound estimates as those for $X^0$. The lower bound estimates, on the other hand,  are  obtained via the conjunction of short-time  near-diagonal lower bound estimates and a chain argument. The two-sided Green function estimates are  provided in Section \ref{S:4}.


\section{Preliminaries: Girsanov Transform of BMVD and Its Resolvent Kernel}\label{S:2}

Let $b:E\to \IR$ be a measurable function in  $L^{p_1, p_2}(E)$ with $p_1\in (1, \infty]$ and $p_2\in (2,\infty]$.  We have defined  BMVD  with drift   $X$ in Definition \ref{Def-driftedBMVD} in terms of Dirichlet form. In this section, we characterize $X$ through Girsanov transform and identify its resolvent kernel.

To begin with, we give a  rigorous definition  for the   family of probability measures  $\IQ$  characterized by \eqref{Girsanov}.
Towards  this purpose,  we first set
$$
M^{(1)}_t:=\int_0^t \mathbf{1}_{\IR_+}(X^0_s)dX^0_s \quad  \text{and} \quad  M^{(2)}_t:=\int_0^t \mathbf{1}_{D_\eps}(X^0_s)dX^0_s.
$$ 
To see how the above two  stochastic integrals  are rigorously defined, we set  $A_\delta:=\{x\in \IR_+: \, |x|>\delta\}$, $B_\delta:=\{x\in D_\eps: \, |x|_\rho>\delta\}$. We inductively define a sequence of stopping times as follows:
\begin{align*}
S^\delta_0&:=\inf\{t>0: X^0_t\in A_\delta \};\\
T^\delta_1&:=\sigma_{\{a^*\}}; 
\end{align*}
then for $k=1, 2\dots$, 
\begin{align*}
S^\delta_{k}&:=S_{k-1}^\delta \circ \theta_{T^\delta_{k}}+T^\delta_{k}; \\
T^\delta_{k+1}&:=\sigma_{\{a^*\}}\circ\theta_{S^\delta_{k}}+S^\delta_{k}; \\
&\cdots
\end{align*}
For each $t>0$, we define
\begin{equation*}
M^{(1), \delta}_t:=\sum_{n\ge 1}\int_{S^\delta_{n-1}\wedge t}^{T^\delta_n\wedge t}1\, dX^0_s=\int_0^t \sum_{n\ge 1}\mathbf{1}_{\left[S_{n-1}^\delta\wedge t, \;T_n^\delta\wedge t\right]}(s)\;dX^0_s.
\end{equation*}
The above stochastic integral is well-defined because the restriction of $X^0$ on $\IR_+$ is an   1-dimensional Brownian motion, and the summation is finite. Since for each fixed $t>0$,
\begin{equation*}
\sum_{n\ge 1}\mathbf{1}_{\left[S_{n-1}^\delta\wedge t, \;T_n^\delta \wedge t\right]}(s) \uparrow \mathbf{1}_{\{0\le s\le t:\,X^0_s\in \IR_+\}} ,  \quad \text{as }\delta\rightarrow 0,
\end{equation*}
it follows that there is a unique square-integrable martingale
\begin{equation*}
M_t^{(1)}=\int_0^t \mathbf{1}_{\{X^0_s\in \IR_+\}}dX^0_s.
\end{equation*}
Similarly, in order  to  rigorously  define  $M_t^{(2)}$, we set
\begin{align*}
\widehat{S}^\delta_0&:=\inf\{t>0: X^0_t\in B_\delta\};
\\
\widehat{T}^\delta_1&:=\sigma_{\{a^*\}}; 
\end{align*}
then  for $k=1,2,\dots,$
\begin{align*}
\widehat{S}^\delta_{k}&:=\widehat{S}^\delta_{k-1}\circ \theta_{\widehat{T}^\delta_{k}}+\widehat{T}^\delta_{k}; \\
\widehat{T}^\delta_{k+1}&:=\sigma_{\{a^*\}}\circ\theta_{\widehat{S}^\delta_{k}}+\widehat{S}^\delta_{k}; \\
&\cdots
\end{align*}
Then for each $t>0$, we set
\begin{equation*}
M^{(2), \delta}_t:=\sum_{n\ge 1}\int_{\widehat{S}^\delta_{n-1}\wedge t}^{\widehat{T}^\delta_n\wedge t}1\, dX^0_s=\int_0^t \sum_{n\ge 1}\mathbf{1}_{\left[\widehat{S}_{n-1}^\delta\wedge t,\; \widehat{T}_n^\delta\wedge t\right]}(s)\;dX^0_s.
\end{equation*}
It follows that
\begin{equation*}
\sum_{n\ge 1}\mathbf{1}_{\left[\widehat{S}_{n-1}^\delta\wedge t,\; \widehat{T}_n^\delta \wedge t\right]}(s) \uparrow \mathbf{1}_{\{0\le s\le t:\,X^0_s\in D_\eps\}}, \quad \text{as }\delta\rightarrow 0.
\end{equation*}
Therefore  there is a unique square-integrable martingale
\begin{equation*}
M_t^{(2)}=\int_0^t \mathbf{1}_{\{X^0_s\in D_\eps\}}dX^0_s.
\end{equation*}
Now we define the  probability measure $\IQ$ as follows:
\begin{align*}
\frac{d\IQ}{d\IP}\bigg|_{\FF_t}=M_t:=\exp\bigg(\int_0^t b_1(M^{(1)}_s)dM^{(1)}_s+\int_0^t b_{2}\Big(M^{(2)}_s\Big)dM_s^{(2)}
-\frac{1}{2} \int_0^t |b(X^0_s)|^2 ds \bigg),\quad t\ge 0,
\end{align*}
where  $\FF_t=\sigma\{X^0_s, s\le t\}$,    and $\IP$ is the distribution law of $X^0$.   Let $G_\alpha^0$ be the resolvents of $X^0$. The next  theorem identifies the process determined by $\IQ$ is indeed BMVD with drift $X$. 

\begin{thm}\label{T:2.1}  Let  $X$ be the process associated with the  Dirichlet space \eqref{Dirichlet-drifted-BMVD}.  $(X, \IQ)$ and $(X^0, \IP)$ are different descriptions of the same process. Denote the resolvent operators of $X^0$ and $X$ by $(G^0_\alpha)_{\alpha>0}$ and  $(G_\alpha^b)_{\alpha>0}$ on $L^2(E, m_p)$, respectively. Then there exists some $\alpha_0>0$ such that it holds for all $\alpha>\alpha_0$ and all  $f\in L^2(E, m_p)$ that 
\begin{description}
\item{(i)} $G_\alpha^0(b\cdot \nabla G_\alpha ^0)^nf$ is in $\mathcal{D}(\EE^0)$, for every $n=0, 1,2,\cdots$.
\item{(ii)} $G^b_\alpha f\in  \mathcal{D}(\EE^0)$, and  $\displaystyle{G_\alpha^bf=\sum_{n=0}^\infty G_\alpha^0(b\cdot \nabla G_\alpha ^0)^n}f$, where the infinite sum converges in norm $\|\cdot \|_{1,2}$.
\end{description}
The norm $\|\cdot \|_{1,2}$ above   is understood  as the sum of the $\|\cdot \|_{1,2}$ on $D_\eps$ and $\IR_{+}$, i.e., $\|f\|_{1,2}:=\|f\cdot \mathbf{1}_{D_\eps}\|_{W^{1,2}(D_\eps)}+\|f\cdot \mathbf{1}_{\IR_+}\|_{W^{1,2}(\IR_+)}$.
\end{thm}


\begin{proof}
The proof follows an argument similar to  \cite[Theorem 5.5]{ChenZ} with few minor changes. The details are omitted.

\end{proof}

We now describe the infinitesimal generator associated with $X$. For notation simplicity we let
$$
\FF^0:=\left\{f:\, f|_{D_\eps}\in W^{1,2}_0(D_\eps),\, f|_{\IR_+}\in W_0^{1,2}(\IR_+), f(x) = f(0) \, \text{ q.e. } x \in  \partial D_\eps\right\}.
$$
Let $u_0(x):=\IE_x[e^{-\sigma_{\IR_+}}]$ for  $x\in D_\eps$ and $u_0(x):=\IE_x[e^{-\sigma_{D_\eps}}]$  for   $x\in \IR_+$. It is known that $u_0|_{D_\eps}\in W^{1,2}(D_\eps)$, $u_0|_{\IR_+}\in W^{1,2}(\IR_+)$. For any $u\in \mathcal{D}(\EE^b)=\mathcal{D}(\EE^0)$, we define the its {\it flux} at $a^*$ by
$$
\mathcal{N}_p(u)(a^*):=\int_E\nabla u(x)\cdot \nabla u_0(x)m_p(dx)+\int_E \Delta u(x)u_0(x)m_p(dx).
$$
The followting theorem characterizes the infinitesimal generator for  $X$. 
\begin{thm}\label{T:2.2}
$X$ has an infinitesimal generator $\mathcal{L}^b :=\frac{1}{2}\Delta +b\cdot \nabla $.  Its domain $\mathcal{D}(\mathcal{L}^b)$ is a subspace of $\mathcal{D}(\EE^b)$ such that a function $u\in\mathcal{D}(\mathcal{\EE}^b)$ is in  $\mathcal{D}(\mathcal{L}^b)$ if and only if the distributional Laplacian $\Delta u$  exists as an $L^2$-integrable function on $E\backslash \{a^*\}$ and $u$ satisfies zero flux property  at $a^*$. 
\end{thm}
\begin{proof}
  $\mathcal{D}(\EE^b)= \mathcal{D}(\EE^0)$   is the linear span of $\FF^0\cup \{u_0\}$ (see, for example, \cite[Theorem 7.5.4]{CF}).
For  $u\in \mathcal{D}(\mathcal{\EE}^b)=\mathcal{D}(\mathcal{\EE}^0)$, $u\in \mathcal{D}(\mathcal{L}^b)$ if and only if   there is some $f\in L^2(E; m_p)$ such that
\begin{equation}\label{eq:2.2}
\EE^b(u, v)=-\int_E f(x)v(x)m_p(dx), \quad \textrm{for every }v\in \mathcal{D}(\mathcal{\EE}^b).
\end{equation}
If this holds, we set $\mathcal{L}^bu=f$.  \eqref{eq:2.2} is equivalent to the conjunction of 
\begin{align}\label{529528}
\frac{1}{2}\int_E\nabla u(x)\cdot \nabla v(x)m_p(dx) -\int_E b(x)\nabla u(x)v(x)m_p(dx)=-\int_E f(x)v(x)m_p(dx) \nonumber
\\ \textrm{for every }v\in \FF^0;
\end{align}
and
\begin{equation}\label{529530}
\int_E\nabla u(x)\cdot \nabla u_0(x)m_p(dx)-\int_E b(x)\nabla u(x)u_0(x)m_p(dx)=-\int_E f(x)u_0(x)m_p(dx).
\end{equation}
\eqref{529528} is equivalent to that  $\Delta u\in L^2(E)$, and \eqref{529530} is equivalent to $\mathcal{N}_p(u)(a^*)=0$.
\end{proof}

\section{Gradient Estimates and Small Time Heat Kernel Estimates}\label{S:3}

Unless otherwise stated, it is always fixed  in this section that  $T>0$ and $b$ is a measurable function in $L^{p_1, p_2}(E)$ with $p_1\in (1, \infty]$ and $p_2\in (2,\infty]$. Without loss generality, we assume throughout the remaining of this paper  that $0<\eps\le 1/4$. To establish two-sided short-time heat kernel estimates for BMVD with drift $X$ via the method of Duhamel's formula, we first  establish the smallness of the gradients of  the transition density   $p(t,x,y)$ of $X^0$. Recall the canonical form $p^0_{0,\alpha}(t,x,y)$ of the heat kernel of $X^0$ has been defined in \eqref{e:1.3}.
\begin{prop}\label{P:3.1}
 There exist   constants   $C_1>0$ and   $\beta_1>0$ such that
\begin{equation}\label{213114}
\left|\nabla_x p(t,x,y)\right|\le \frac{C_1}{\sqrt{t}}\,p^0_{0,\beta_1}(t,x,y),  \quad t\le T,\, (x,y)\in (E\times E)\setminus ((D_\eps \cup \{a^*\})\times (D_\eps\cup \{a^*\})).
\end{equation}

\end{prop}
\begin{proof}
The condition says at least one of $x$ and $y$ is in $\IR_+$.  We consider the signed radial process of $X^0$ defined as follows: Let
\begin{equation}\label{848}
u(x):=\left\{
  \begin{array}{ll}
    -|x|, & \hbox{$x \in \IR_+$;} \\ \\
    0, &\hbox{$x=a^*$;} \\ \\
   |x|_\rho, & \hbox{$x\in D_\eps $.}
  \end{array}
\right.
\end{equation}
We call $Y_t:=u(X^0_t)$ the signed radial process of $X^0$. It has been shown in \cite[Proposition 4.3]{CL} that $Y$ is characterized by the following SDE:
\begin{equation}\label{YsymmSDE}
dY_t=dB_t+\frac{1}{2(Y_t+\eps)}\mathbf{1}_{\{Y_t>0\}}dt+\frac{2\pi\eps-p}{2\pi\eps+p}d\widehat{L}_t^0(Y),
\end{equation}
where $\widehat{L}^0(Y)$ is the symmetric semimartingale local time of $Y$ at $0$. Let  $\eta:=\frac{2\pi\eps-p}{2\pi\eps+p}$ and $Z$ be the skew Brownian motion
$$ dZ_t = dB_t + \eta \widehat{L}^0_t (Z) ,
$$
where $\widehat{L}^0_t (Z)$ is the symmetric local time of $Z$ at $0$.
The diffusion process $Y$ can be obtained from $Z$ through a drift perturbation
(i.e. Girsanov transform).
The transition density function $p^Z(t, x, y)$ of $Z$ is explicitly known and
enjoys the two-sided Aronson-type Gaussian estimates;
see, e.g., \cite{RY}. One can further verify  that for some constants $c_1, c_2>0$, it holds
$$ | \nabla_x p^Z(t, x, y)| \leq c_1  t^{-1} \exp (-c_2 |x-y|^2/t).
$$
Set $\displaystyle{f(x):=\frac{1}{x+\eps}\mathbf{1}_{\{x>0\}}}$. It is not hard to see that $f\in K_{1}$:
\begin{align*}
 \sup_{x\in \IR}\int_{|y-x|\le 1}|f(y)|dy \le\sup_{x>-1}\int_{0\vee (x-1)}^{x+1}\frac{1}{y+\epsilon}dy<\infty.
\end{align*}
By using the same argument as that for Theorem A(b) in Zhang \cite[\S 4]{Z1}, we can derive that for some $c_3, c_4>0$, 
\begin{equation}\label{e:3.8}
p^Y(t,x,y)\lesssim t^{-1/2}e^{-c_3(x-y)^2/t}, \quad  \left|\frac{\partial}{\partial x}p^Y(t,x,y)\right|\lesssim t^{-1}e^{-c_4(x-y)^2/t},\quad x, y\in \IR,
\end{equation}
where $p^Y(t,x,y)$ denotes the transition density function of $Y$. In the following, we divide our discussion into three cases depending on the positions of $x, y$.
\\
{\it Case 1.} $x, y\in \IR_+$. In this case, since $p(t,x,y)=p^Y(t,-x,-y)$,
we immediately have
\begin{equation*}
\left|\frac{\partial }{\partial x}p(t,x,y)\right|=\left|\frac{\partial}{\partial x}p^Y(t,-x,-y)\right|\lesssim t^{-1}e^{-c_4(x-y)^2/t}. 
\end{equation*}
{\it Case 2. } $x\in D_\eps\cup \{a^*\}$, $y\in \IR_+$. In this case, we have  for any $0< a<b$,
\begin{align*}
\int_a^bp(t,x,y)dy&=\IP_x[X_t\in \IR_+ \text{ with }a\le X_t<b]=\IP_{|x|_\rho}[-b \le Y_t\le -a]
\\
&=\int_{-b}^{-a}p^Y(t, |x|_\rho, u)du \stackrel{y=-u}{=}\int_a^b p^Y(t, |x|_\rho, -y)dy.
\end{align*}
This implies $p(t,x,y)=p^Y(t, |x|_\rho, -y)$, for $x\in D_\eps,\, y\in \IR_+\cup \{a^*\}$, which only depends on $|x|$, for fixed $t$ and $y$.  For a  function $f$ on $\IR^2$ that is rotationally invariant,  i.e., satisfying $f(x)=g(r)$ for some function $g$ on $\IR$ and for all $|x|=r$, it holds that $\nabla_x f(x)=g'(r)\cdot \frac{x}{|x|}$. Thus $|\nabla_x f(x)|=|g'(r)|$. Hence, 
\begin{equation*}
|\nabla_x p(t,x,y)|=\left|\frac{\partial }{\partial r}p^Y(t,r, -y)\Big|_{r=|x|_\rho}\right|\stackrel{\eqref{e:3.8}}{\lesssim} t^{-1}e^{-c_4(|x|_\rho+y)^2/t}, \quad x\in D_\eps\cup \{a^*\} , y\in \IR_+.
\end{equation*}
{\it Case 3.}  $x\in \IR_+$, $y\in D_\eps\cup \{a^*\}$.  Set $\wt {p}(t,x,r):=p(t,x,y)$ for $r=|y|_\rho$.  For any $a>b\ge0$,
\begin{align*}
\int_a^b p^{(Y)}(t,-x,y)dy&=  \IP_{-x} \left[ a\leq Y_t\leq b\right] = \IP_x \left[ X_t \in {D_\eps} \hbox{ with }  a\leq |X_t |_\rho\leq b\right] \\
&=  \int_{y\in {D_\eps}, \; a \leq |y|_\rho \leq b}p (t,x,y)m_p(dy) =\int_{y\in {D_\eps},\;  a+\eps \leq |y| \leq b+\eps}p (t,x,y)m_p(dy)
\\
&=\int_{a}^b 2\pi (r+\eps)\wt {p}(t,x,r)dr.
\end{align*}
This  implies when $x\in \IR_+$, $y\in {D_\eps}\cup \{a^*\}$,
\begin{equation}
p^{(Y)}(t,-x,|y|_\rho) =2\pi ( |y|_\rho +\eps) \wt {p} (t,x, |y|_\rho ) =
2\pi ( |y|_\rho +\eps)p (t,x,y) .
\end{equation}
From here we can derive that for $x\in \IR_+\cup \{a^*\}$, $y\in D_\eps$, 
\begin{align*}
\left|\frac{\partial }{\partial x}p(t,x,y)\right|&=\frac{1}{2\pi \left(|y|_\rho+\eps\right)}\left|\frac{\partial }{\partial x}p^{(Y)}(t, -x, |y|_\rho)\right|=\frac{1}{2\pi |y|}\left|\frac{\partial }{\partial x}p^{(Y)}(t, -x, |y|_\rho)\right|
\\
&\stackrel{\eqref{e:3.8}}{\lesssim }\frac{1}{2\pi |y|} t^{-1}e^{-c_4\left(x^2+|y|^2_\rho\right)/t}\lesssim t^{-1}e^{-c_4\left(x^2+|y|^2_\rho\right)/t}.
\end{align*}
Combining all the three cases above, by \eqref{e:1.3},  we conclude that   there exists some $\beta_1>0$ such that
\begin{equation*}
\left|\nabla_x p(t,x,y)\right|\lesssim \frac{1}{\sqrt{t}}\,p^0_{0,\beta_1}(t,x,y) \quad \quad t\le T,\, (x,y)\in (E\times E)\setminus ((D_\eps \cup \{a^*\})\times (D_\eps\cup \{a^*\})).
\end{equation*} 
\end{proof}
\begin{rmk}\label{R:3.2}
From the proof to Proposition \ref{P:3.1} Case 2, it can be told that the same conclusion holds for  $x\in D_\eps\cup \{a^*\}$, $y=a^*$.
\end{rmk}
\begin{prop}\label{P:3.2}
 There exist constants $\beta_2>0$  and $C_2, C_3>0$ such that 
\begin{description}
\item{\rm(i)} If $x,y\in  D_\eps\cup \{a^*\}$ with $\max\{|x|_\rho, |y|_\rho\}\ge 1$, then
\begin{equation*}
\left|\nabla_x p(t,x,y)\right|\le \frac{C_2}{\sqrt{t}}\,p^0_{0,\beta_2}(t,x,y), \quad t\le T
\end{equation*}
\item{\rm(ii)} If $x,y\in D_\eps\cup \{a^*\}$ with $\max\{|x|_\rho, |y|_\rho\}< 1$, then
\begin{equation*}
|\nabla_x p(t,x,y)|
\le \frac{C_3}{t}e^{-\beta_2(|x|_\rho^2+|y|_\rho^2)/t}+\frac{C_3}{t^{3/2}}\left(1\wedge \frac{|y|_\rho}{\sqrt{t}}\right)e^{-\beta_2 |x-y|^2/t}, \quad t\le T. 
\end{equation*}
\end{description}

\end{prop}
\begin{proof}
First of all, for both cases we have
\begin{equation}\label{115202}
\nabla_xp(t,x,y)=\nabla_xp_{D_\eps}(t,x,y)+\nabla_x\bar{p}_{D_\eps}(t,x,y).
\end{equation}
In this proof, we temporarily denote by $W$    a
  two-dimensional Brownian motion and $p_2(t, x, y)=(2\pi t)^{-1} \exp (-|x-y|^2 /2t)$ its transition density. Started from $x\in D_\eps$, $X^0$ has the same distribution as  $W$  before hitting $a^*$. By \cite[Theorem 1.1]{CKP}, there exists some $c_1>0$ such that
\begin{equation}\label{116209}
\left|\nabla_x p_{D_\eps}(t,x,y)\right|= \left|\nabla_x p_{2, D_\eps}(t,x,y)\right|  \lesssim t^{-3/2}\left(1\wedge \frac{|y|_\rho}{\sqrt{t}}\right)e^{-c_1 |x-y|^2/t}.
\end{equation}
 In the following, by abusing the notation a little bit, we let $\IP_x$ stand for the distribution of both $X^0$ and $W$   started from $x\in D_\eps$   before exiting $D_\eps$.  For the second term on the right hand side of \eqref{115202}, first by the symmetry of $X^0$ as well as its part process, 
 \begin{align}\label{e:3.11}
 \bar{p}_{D_\eps}(t,x,y)&=p(t,x,y)-p_{D_\eps}(t,x,y)=p(t,y,x)-p_{D_\eps}(t,y,x)= \bar{p}_{D_\eps}(t,y,x).
 \end{align}
Owing to Remark \ref{R:3.2}, it   follows   that for some $c_2>0$, it holds 
\begin{align}
|\nabla_x\bar{p}_{D_\eps}(t,x,y)|&\stackrel{\eqref{e:3.11}}{=}|\nabla_x\bar{p}_{D_\eps}(t,y,x)|\le \int_0^t \IP_y\left(\sigma_{a^*}\in ds\right)|\nabla_xp(t-s, a^*,x)| \nonumber
\\
&\lesssim \int_0^t \IP_y\left(\sigma_{a^*}\in ds\right)\frac{1}{t-s}e^{-c_2 |x|_\rho^2/(t-s)}\nonumber
\\
\nonumber &\lesssim e^{-c_2 |x|^2_\rho/t}\int_0^t \frac{1}{t-s} \IP_y(\sigma_{a^*}\in ds)\\
&=e^{-c_2 |x|^2_\rho/t} \cdot \IE_{y} \left[p_2\left(t-\tau_{D_\eps}, W_{\tau_{D_\eps}}, W_{\tau_{D_\eps}}\right), \tau_{D_\eps}<t \right] \nonumber
 \\
 &\lesssim e^{-c_2 |x|^2_\rho/t} \cdot t^{-1}e^{-\delta_{D_\eps}(y)^2/(2t)}\lesssim t^{-1}e^{-(c_2\wedge 1)(|x|_\rho^2+|y|^2_\rho)/(2t)} .\label{108}
\end{align}
Therefore, in view of \eqref{115202}, combining \eqref{116209} and \eqref{108} shows   for both Case (i) and (ii)    that :
\begin{align}\label{e:3.12}
|\nabla_x p(t,x,y)|\lesssim t^{-1}e^{-c_2(|x|_\rho^2+|y|^2_\rho)/t}+t^{-3/2}\left(1\wedge \frac{|y|_\rho}{\sqrt{t}}\right)e^{-c_1 |x-y|^2/t}.
\end{align}
Thus the desired conclusion for  Case (ii) has been proved if we take $\beta_2=c_1\wedge c_2$ in the statement of this proposition. To verify the conclusion for Case (i), we observe that when $\max\{|x|_\rho, |y|_\rho\}>1$,
\begin{equation}\label{e:3.14r}
|x|_\rho+|y|_\rho\ge \max\{|x|_\rho, |y|_\rho\}>1.
\end{equation}
This yields that when $\max\{|x|_\rho, |y|_\rho\}>1$, 
$$
|x|_\rho+|y|_\rho=|x|+|y|-2\eps\ge |x-y|-2\eps \ge \frac{1}{3}|x-y|,\quad \text{if } |x-y|\ge 3\eps;
$$
while when $|x-y|<3\eps\le 3/4$ (since it is assumed $\eps\le 1/4$),
\begin{eqnarray*}
|x|_\rho+|y|_\rho \stackrel{\eqref{e:3.14r}}{>} 1 >\frac{2}{3}|x-y|.
\end{eqnarray*}
Therefore, when $\max\{|x|_\rho, |y|_\rho\}>1$, it always holds $|x-y|^2/18\le |x|^2_\rho+|y|^2_\rho$. Thus
\begin{align*}
|\nabla_x p(t,x,y)|& \lesssim t^{-1}e^{-c_2(|x|_\rho^2+|y|^2_\rho)/t}+t^{-3/2}\left(1\wedge \frac{|y|_\rho}{\sqrt{t}}\right)e^{-c_1 |x-y|^2/t}
\\
&\lesssim t^{-1}e^{-c_2|x-y|^2/(18t)}+t^{-3/2}e^{-c_1 |x-y|^2/t}  \lesssim t^{-3/2}e^{-c_3 |x-y|^2/t}.
\end{align*}
 In view of  Theorem \ref{smalltime},  we have arrived at the conclusion for Case (i)  by selecting $\beta_2=c_3$ in the statement of this proposition. Therefore, the proof to this proposition is complete by choosing $\beta_2=\min_{1\le i\le 3} c_i$.
\end{proof}
\begin{prop}\label{121051}
It holds for every pair of $\alpha, \beta$: $0<\alpha<\beta$ that
\begin{equation}\label{1171130}
\int_0^t \frac{1}{\sqrt{s}}\int_{E}p^0_{0,\alpha}(t-s, x,z)|b(z)|p^0_{0,\beta}(s, z,y)m_p(dz)ds \le C_4(t)p^0_{0,\alpha}(t,x,y), \,  0<s<t\le T,\, x,y \in \IR_+,
\end{equation}
where $C_4(t)$ is a positive non-decreasing function in $t$ \emph{(}possibly depending on $\alpha$ and $\beta$\emph{)}  such that  $C_4(t)\rightarrow 0$ as $t \rightarrow 0$. In particular, there exists some $\alpha_1>0$ such that for every  $\alpha<\alpha_1$, it holds
\begin{align}
\int_0^t \int_{z\in E}p^0_{0,\alpha}(t-s,x,z)|b(z)|\left|\nabla_z p(s, z, y )\right| m_p(dz)ds&\le C_5(t)p^0_{0,\alpha}(t,x,y), \nonumber
\\
& 0<s<t\le T, x,y\in \IR_+,\label{e:3.14}
\end{align}
where $C_5(t)$ is a positive non-decreasing function in $t$ \emph{(}possibly depending on $\alpha$\emph{)} such  that $C_5(t) \rightarrow 0$ as $t\rightarrow 0$.
\end{prop}
\begin{proof}
Without loss of generality, we assume $b\ge 0$. In view of Proposition \ref{P:3.1}, 
We  divide the computation into two parts depending on the position of $z$. Recall $b=b_1+b_2$ where $b_1:=b|_{\IR_+}\in L^{p_1}(\IR_+)$, and $b_2:=b|_{D_\eps}\in L^{p_2}(D_\eps)$. 
\\
{\it Case 1.}  $z\in \IR_+$. We first record the following triangle-inequality for distance function which can be verified using elementary mathematics:
\begin{equation}\label{triangle-v1}
\frac{|x-z|^2}{t-s}+\frac{|z-y|^2}{s}\ge \frac{|x-y|^2}{t},\quad x,y,z\in \IR.
\end{equation}
By H\"{o}lder's inequality, for $p_1,q_1>0$ satisfying $1/p_1+1/q_1=1$, it holds that 
\begin{eqnarray}
&& \int_{0}^t \frac{1}{\sqrt{s}}\int_{\IR_+}p^0_{0,\alpha} (t-s, x,z)b(z)p^0_{0,\beta}(s, z,y)m_p(dz)ds \nonumber
\\
&=& \int_0^t \frac{1}{\sqrt{s}}\int_{\IR_+}\frac{1}{\sqrt{t-s}}e^{-\alpha|x-z|^2/(t-s)}b(z)\frac{1}{\sqrt{s}}e^{-\beta|z-y|^2/s}m_p(dz)ds \nonumber
\\
&\stackrel{\eqref{triangle-v1}}{\le}& e^{-\alpha |x-y|^2/t}\int_0^t s^{-1+1/(2q_1)}(t-s)^{-1/2}\int_{\IR_+}b(z)s^{-1/(2q_1)}e^{-(\beta-\alpha)|z-y|^2/s}m_p(dz)ds \nonumber
\\
&\le& e^{-\alpha |x-y|^2/t}\int_0^t s^{-1+1/(2q_1)}(t-s)^{-1/2} \|b_{
1}\|_{p_1} \left(\int_{\IR_+}\left(s^{-1/(2q_1)}e^{-(\beta-\alpha)|z-y|^2/s}\right)^{q_1} m_p(dz)\right)^{1/q_1}ds \nonumber
\\
&\lesssim& e^{-\alpha |x-y|^2/t}\int_0^t s^{-1+1/(2q_1)}(t-s)^{-1/2} \|b_{
1}\|_{p_1}ds \nonumber
\\
&\lesssim& \|b_1\|_{p_1} e^{-\alpha |x-y|^2/t}\cdot \frac{t^{1/(2q_1)}}{\sqrt{t}}=\|b_1\|_{p_1}t^{1/(2q_1)}p^0_{0, \alpha}(t,x,y),\label{1224}
\end{eqnarray}
where  the second last $``\lesssim"$ in \eqref{124} follows from
\begin{eqnarray*}
&&\int_{\IR_+}\left(s^{-1/(2q_1)}e^{-(\beta-\alpha)|z-y|^2/s}\right)^{q_1} m_p(dz)
\\
&=&\int_{\IR_+}s^{-1/2}e^{-(\beta-\alpha)q_1|z-y|^2/s} \,dz=\int_{\IR_+}s^{-1/2}e^{-(\beta-\alpha)q_1|z|^2/s} \,dz =\frac{1}{2}\sqrt{\frac{\pi }{(\beta -\alpha )q_1}}.
\end{eqnarray*}
{\it Case 2. } $z\in D_\eps \cup \{a^*\}$.    We first note that 
\begin{equation}\label{e:3.16}
xe^{-x^2}\le e^{-x^2/2}, \quad \text{for all }x\in \IR.
\end{equation}
  For $q_2$ satisfying $1/p_2+1/q_2=1$, we first record the following computation which will be used later. 
 \begin{eqnarray}
&& \int_{D_\eps}s^{-1/2}e^{-\beta q_2|z|_\rho^2/s} m_p(dz) \stackrel{r=|z|_\rho}{=} \int_0^\infty 
\frac{r+\eps}{\sqrt{s}}e^{-\beta q_2 r^2/s}p\,dr \nonumber
\\
&\asymp & \int_0^\infty 
\frac{1}{\sqrt{s}}e^{-\beta q_2 r^2/s}dr+\int_0^\infty 
\frac{r}{\sqrt{s}}e^{-\beta q_2 r^2/s}dr \nonumber
 \\
 &=&\frac{1}{2}\sqrt{\frac{\pi }{\beta q_2}}+\frac{1}{\sqrt{\beta q_2}}\int_0^\infty 
\sqrt{\frac{\beta q_2 r^2}{s}}e^{-\beta q_2 r^2/s}dr \nonumber
\\
&\stackrel{\eqref{e:3.16}}{\le}& \frac{1}{2}\sqrt{\frac{\pi }{\beta q_2}}+\frac{1}{\sqrt{\beta q_2}}\int_0^\infty  e^{-\beta q_2 r^2/2s}dr =\frac{1}{2}\sqrt{\frac{\pi }{\beta q_2}}+\frac{1}{\sqrt{\beta q_2}} \sqrt{\frac{\pi s }{2\beta q_2}}. \label{e:3.15}
 \end{eqnarray}
 On account of H\"{o}lder's  inequality, it holds for $0<\alpha<\beta$ that
\begin{eqnarray}
&& \int_{0}^t \frac{1}{\sqrt{s}}\int_{D_\eps}p^0_{0,\alpha} (t-s, x,z)b(z)p^0_{0,\beta}(s, z,y)m_p(dz)ds \nonumber
\\
&=& \int_0^t \frac{1}{\sqrt{s}}\int_{D_\eps}\frac{1}{\sqrt{t-s}}e^{-\alpha(|x|^2+|z|_\rho^2)/(t-s)}b(z)\frac{1}{\sqrt{s}}e^{-\beta(|y|^2+|z|_\rho^2)/s}m_p(dz)ds \nonumber
\\
&\le &  e^{-\alpha (|x|^2+|y|^2)/t}\int_0^t s^{-1+1/(2q_2)}(t-s)^{-1/2} \int_{D_\eps}b(z)s^{-1/(2q_2)}e^{-\beta|z|_\rho^2/s}m_p(dz)ds \nonumber
\\
&\lesssim & e^{-\alpha |x-y|^2/t}\int_0^t s^{-1+1/(2q_2)}(t-s)^{-1/2} \|b_{2}\|_{p_2} \left(\int_{D_\eps}\left(s^{-1/(2q_2)}e^{-\beta|z|_\rho^2/s}\right)^{q_2} m_p(dz)\right)^{1/q_2}ds \nonumber
\\
&\le& e^{-\alpha |x-y|^2/t}\int_0^t s^{-1+1/(2q_2)}(t-s)^{-1/2} \|b_{2}\|_{p_2} \left(\int_{D_\eps}s^{-1/2}e^{-\beta q_2|z|_\rho^2/s} m_p(dz)\right)^{1/q_2}ds \nonumber
\\
&\stackrel{\eqref{e:3.15}}{\lesssim}& e^{-\alpha |x-y|^2/t}\int_0^t s^{-1+1/(2q_2)}(t-s)^{-1/2} \|b_{2}\|_{p_2} \left(\frac{1}{2}\sqrt{\frac{\pi }{\beta q_2}}+\frac{1}{\sqrt{\beta q_2}} \sqrt{\frac{\pi s }{2\beta q_2}}\right)^{1/q_2}ds \nonumber
\\
&\stackrel{s\le T}{\lesssim}& \|b_2\|_{p_2} e^{-\alpha |x-y|^2/t}\cdot \frac{t^{1/(2q_2)}}{\sqrt{t}}\lesssim \|b_2\|_{p_2}t^{1/(2q_2)}p^0_{0, \alpha}(t,x,y).\label{1225}
\end{eqnarray}
Adding up  \eqref{1224} and \eqref{1225} shows that for each pair of $0<\alpha<\beta$, it holds 
\begin{align*}
 \int_{0}^t\int_{E}p^0_{0,\alpha} (t-s, x,z)& b(z)  p^0_{0, \beta}(s, z,y)m_p(dz)ds  \lesssim  \left(t^{1/(2q_1)}+t^{1/(2q_2)}\right)(\|b_1\|_{p_1}+\|b_2\|_{p_2}) p^0_{0, \alpha}(t,x,y),
\end{align*}
which proves \eqref{1171130}. The second conclusion \eqref{e:3.14} readily follows in view of Proposition \ref{P:3.1}.

\end{proof}
The next proposition is analogous to Proposition \ref{121051} above but concerning the case  $x\in \IR_+$ and $y\in D_\eps \cup \{a^*\}$.
\begin{prop}\label{1236}
There exist $\alpha_2>0$ such that for every  $0<\alpha<\alpha_2$, it holds
\begin{eqnarray}
\int_0^t \int_{z\in E}p^0_{0,\alpha}(t-s,x,z)|b(z)|\left|\nabla_z p(s, z, y )\right| m_p(dz)ds\le C_6(t)p^0_{0,\alpha}(t,x,y),\nonumber
 \\
 0<s<t\le T, x\in \IR_+, y\in D_\eps \cup \{a^*\},\label{Duhamel}
\end{eqnarray}
where $C_6(t)$ is a positive non-decreasing function in $t$ \emph{(}possibly depending on $\alpha$\emph{)}  such that  $C_6(t)\rightarrow 0$ as $t \rightarrow 0$.
\end{prop}
\begin{proof}

We again assume $b\ge 0$ and divide the proof into three cases depending on the position of $z$. 
\\
{\it Case 1.} $z\in \IR_+$. In view of Proposition \ref{P:3.1}, for this case, we first  show the following inequality:  For any $0<\alpha<\beta$, it holds
\begin{align}
\int_0^t \frac{1}{\sqrt{s}}\int_{E}p^0_{0,\alpha}(t-s, x,z)b(z)p^0_{0, \beta}(s, z,y)& m_p(dz)ds \le c_1(t)p^0_{0, \alpha}(t,x,y),\nonumber
\\
& 0<s<t\le T, x\in \IR_+, y\in D_\eps \cup \{a^*\}, \label{eqP:3.4}
\end{align}
for some $c_1(t)$ satisfying the same conditions as  those for  $C_6(t)$. The following version of  triangle inequality will be needed in this proof:
\begin{equation}\label{triangle-ineq-1}
\frac{|x-z|^2}{t-s}+\frac{|z|^2}{s}\ge \frac{|x|^2}{t},\quad x,z\in \IR_+.
\end{equation}
Since $b_1\in L^{p_1}(\IR_+)$ with $p_1\in (1, \infty]$, for $q_1>0$ satisfying $1/p_1+1/q_1=1$, it holds for  $0<\alpha<\beta$:
\begin{eqnarray}
&&  \int_{0}^t \frac{1}{\sqrt{s}}\int_{\IR_+}p^0_{0,\alpha} (t-s, x,z)b(z)p^0_{0,\beta}(s, z,y)m_p(dz)ds \nonumber
\\
&=& \int_{0}^t \frac{1}{\sqrt{s}}\int_{z\in \IR_+} \frac{1}{\sqrt{t-s}}e^{-\alpha |x-z|^2/(t-s)}\frac{1}{\sqrt{s}} e^{-\beta (|z|^2+|y|_\rho^2)/s}b(z)m_p(dz)ds  \nonumber
\\
&\stackrel{\eqref{triangle-ineq-1}}{\le}& \int_{0}^t \int_{z\in \IR_+} \frac{1}{(t-s)^{1/2}}\frac{1}{s^{1-1/(2q_1)}}e^{-\alpha(|x|^2+ |y|_\rho^2)/t}\; b(z) \frac{1}{s^{1/(2q_1)}}e^{-(\beta-\alpha)|z|^2/s} m_p(dz)ds  \nonumber
\\
&\le & \int_{0}^t \frac{1}{(t-s)^{1/2}}\frac{\|b_1\|_{p_1}}{s^{1-1/(2q_1)}}e^{-\alpha(|x|^2+ |y|_\rho^2)/t}\left(\int_{z\in \IR_+}\left(\frac{1}{s^{1/(2q_1)}}e^{-(\beta-\alpha)|z|^2/s}\right)^{q_1}m_p(dz)\right)^{1/q_1}ds \nonumber
\\
&\lesssim & \|b_1\|_{p_1}\frac{t^{1/(2q_1)}}{\sqrt{t}}e^{-\alpha (|x|^2+|y|_\rho^2)/t}=\|b_1\|_{p_1}t^{1/(2q_1)}p^0_{0, \alpha}(t,x,y), \label{102}
\end{eqnarray}
where
\begin{align*}
\int_{z\in \IR_+}\left(\frac{1}{s^{1/(2q_1)}}e^{-(\beta-\alpha)|z|^2/s}\right)^{q_1}m_p(dz)&=\int_{z\in \IR_+}\frac{p}{s^{1/2}}e^{-(\beta-\alpha)q_1|z|^2/s}dz=\frac{p}{2}\sqrt{\frac{\pi}{(\beta-\alpha)q_1}}.
\end{align*}
Now with Proposition \ref{P:3.1}, we have for $0<\alpha<\beta_1$ that
\begin{equation}
\int_0^t \int_{\IR_+}p^0_{0,\alpha}(t-s, x,z)b(z)|\nabla _zp(s, z,y)| m_p(dz)ds \lesssim  \|b_1\|_{p_1}t^{1/(2q_1)}p^0_{0, \alpha}(t,x,y).
\end{equation}
{\it Case 2.} $z\in D_\eps,\, \text{and } |z|_\rho\ge 1$.
In view of Proposition \ref{P:3.2}, for this case,  we first  establish the following inequality for  $0<\alpha<\beta$:
\begin{align}
\int_0^t \frac{1}{\sqrt{s}}\int_{z\in D_\eps, |z|_\rho\ge 1}p^0_{0,\alpha}(t-s, x,z)b(z)p^0_{0,\beta}(s, z,y)&m_p(dz)ds \le c_2(t)p^0_{0,\alpha}(t,x,y), \nonumber
\\
& 0<s<t\le T, x\in \IR_+, y\in D_\eps \cup \{a^*\},\label{604}
\end{align}
where $c_2(t)\downarrow 0$ as $t\rightarrow 0$. Similar to \eqref{triangle-ineq-1}, we record the following version of triangle inequality:
\begin{equation}\label{triangle-inequality-2}
 \frac{|z|_\rho^2}{t-s}+\frac{|z-y|^2}{s}\ge \frac{|y|_\rho^2}{t},
\end{equation}
which follows from
$$
(t-s)s|y|^2_\rho\le (t-s)s(|z|_\rho+|z-y|)^2\le st|z|_\rho^2+(t-s)t|z-y|^2.
$$
Again by H\"{o}lder's   inequality,  since $p_2>2$, we have for $q_2: 1/p_2+1/	q_2=1$ that
\begin{eqnarray}
&& \int_{0}^t \frac{1}{\sqrt{s}}\int_{z\in D_\eps, |z|_\rho\ge 1}p^0_{0, \alpha} (t-s, x,z)b(z)p^0_{0,\beta}(s, z,y)m_p(dz)ds \nonumber
\\
&=& \int_{0}^t \frac{1}{(t-s)^{1/2}}e^{-\alpha (|x|^2+|z|_\rho^2)/(t-s)}\int_{z\in D_\eps, |z|_\rho\ge 1}\frac{1}{s^{3/2}}e^{-\beta |z-y|^2/s}b(z)m_p(dz)ds \nonumber
\\
&\le &  \int_{0}^t \frac{1}{(t-s)^{1/2}}e^{-\alpha (|x|^2+|z|_\rho^2)/(t-s)}\int_{z\in D_\eps}\frac{1}{s^{3/2}}e^{-\beta |z-y|^2/s}b(z)m_p(dz)ds \nonumber
\\
&= & \int_{0}^t \frac{1}{(t-s)^{1/2}}e^{-\alpha (|x|^2+|z|_\rho^2)/(t-s)-\alpha |z-y|^2/s}\int_{D_\eps}\frac{1}{s^{3/2}}e^{-(\beta-\alpha) |z-y|^2/s}b(z)m_p(dz)ds \nonumber
\\
&\stackrel{\eqref{triangle-inequality-2}}{\le} & e^{-\alpha (|x|^2+|y|_\rho^2)/t}\int_{0}^t \frac{1}{s^{3/2-1/q_2}} \frac{1}{(t-s)^{1/2}}  \int_{z\in D_\eps}\frac{b(z)}{s^{1/q_2}}e^{-(\beta-\alpha)|z-y|^2/s} m_p(dz)ds \nonumber
\\
&\lesssim & e^{-\alpha (|x|^2+|y|_\rho^2)/t} \int_{0}^t \frac{1}{s^{3/2-1/q_2}} \frac{\|b_2\|_{p_2}}{(t-s)^{1/2}} \left(\int_{z\in D_\eps}\left(\frac{1}{s^{1/q_2}}e^{-(\beta-\alpha)|z-y|^2/s}\right)^{q_2} m_p(dz)\right)^{1/q_2} ds  \nonumber
\\
&\lesssim & \frac{t^{1/{q_2}}}{t}e^{-\alpha(|x|^2+|y|^2_\rho)/t}\|b_2\|_{p_2}=\|b_2\|_{p_2}t^{1/2-1/p_2}p^0_{0, \alpha}(t,x,y), \label{146}
\end{eqnarray}
where the last $``\lesssim"$ in \eqref{146}  follows from
\begin{align*}
\int_{D_\eps}\left(\frac{1}{s^{1/q_2}}e^{-(\beta-\alpha)|z-y|^2/s}\right)^{q_2} m_p(dz)&=\int_{D_\eps}\frac{1}{s}e^{-(\beta-\alpha)q_2|z-y|^2/s} dz
\\
&\le \int_{\IR^2}\frac{1}{s}e^{-(\beta-\alpha)q_2|z|^2/s} dz=\frac{\pi }{( \beta -\alpha)q_2}.
\end{align*}
Now by Proposition \ref{P:3.2}, we have shown for $0<\alpha<\beta_2$:
\begin{equation}
\int_0^t \int_{z\in D_\eps, |z|_\rho \ge 1}p^0_{0,\alpha}(t-s, x,z)b(z)|\nabla _zp(s, z,y)| m_p(dz)ds \lesssim \|b_2\|_{p_2}t^{1/2-1/p_2}p^0_{0, \alpha}(t,x,y).
\end{equation}
{\it Case 3.} $z\in D_\eps\cup \{a^*\}, \, \text{and } |z|_\rho\le 1$.  Again we let $q_2$  satisfy $1/p_2+1/	q_2=1$.  In this case, if $|y|_\rho\ge 1$, following an almost exact same computation as that for Case 2, we have for any $0<\alpha<\beta$ that
\begin{eqnarray}
&&\int_{0}^t \frac{1}{\sqrt{s}}\int_{z\in D_\eps, |z|_\rho< 1}p^0_{0, \alpha} (t-s, x,z)b(z)p^0_{0, \beta}(s, z,y)m_p(dz)ds \nonumber
\\
&\le &\int_{0}^t \frac{1}{(t-s)^{1/2}}e^{-\alpha (|x|^2+|z|_\rho^2)/(t-s)}\int_{D_\eps}\frac{1}{s^{3/2}}e^{-\beta |z-y|^2/s}b(z)m_p(dz)ds \nonumber
\\
&\lesssim &e^{-\alpha (|x|^2+|y|_\rho^2)/t} \int_{0}^t \frac{1}{s^{3/2-1/q_2}} \frac{\|b_2\|_{p_2}}{(t-s)^{1/2}}  \left(\int_{z\in D_\eps}\left(\frac{1}{s^{1/q_2}}e^{-(\beta-\alpha)|z-y|^2/s}\right)^{q_2} m_p(dz)\right)^{1/q_2} ds  \nonumber
\\
&\lesssim &\frac{t^{1/{q_2}}}{t}e^{-\alpha(|x|^2+|y|^2_\rho)/t}\|b_2\|_{p_2}=\|b_2\|_{p_2}t^{1/2-1/p_2}p^0_{0, \alpha}(t,x,y). \label{e:3.28}
\end{eqnarray}
Owing to Proposition \ref{P:3.2}, \eqref{e:3.28} shows for this case with $|y|_\rho\ge 1$ that,  for $0<\alpha<\beta_2$,
\begin{equation*}
\int_0^t \int_{z\in D_\eps, |z|_\rho<1}p^0_{0,\alpha}(t-s,x,z)b(z)\left|\nabla_z p(s, z, y )\right| m_p(dz)ds \lesssim \|b_2\|_{p_2}t^{1/2-1/p_2}p^0_{0, \alpha}(t,x,y).
\end{equation*}
If $|y|_\rho<1$, in view of Proposition \ref{P:3.2}, for $\beta_2>0$ it holds
\begin{align*}
|\nabla_z p(s,z,y)|&\lesssim \frac{1}{s}e^{-\beta_2(|z|_\rho^2+|y|_\rho^2)/s}+\frac{1}{s^{3/2}}\left(1\wedge \frac{|y|_\rho}{\sqrt{s}}\right)e^{-\beta_2 |z-y|^2/s} 
\\
&\le  \frac{1}{s}e^{-\beta_2(|z|_\rho^2+|y|_\rho^2)/s}+\frac{1}{s^{3/2}}e^{-\beta_2 |z-y|^2/s}.
\end{align*}
Therefore, for  $\alpha, \beta$ satisfying $0<\alpha<\beta<\beta_2$, where $\beta_2$ is  chosen  in Proposition \ref{P:3.2}, it holds
\begin{eqnarray}
&& \int_{0}^t \int_{z\in D_\eps, |z|_\rho< 1}p^0_{0, \alpha} (t-s, x,z)b(z)|\nabla_zp(s, z,y)|m_p(dz)ds  \nonumber
\\
&\lesssim & \int_0^t \frac{1}{\sqrt{t-s}}e^{-\alpha (|x|^2+|z|_\rho^2)/(t-s)}\int_{z\in D_\eps, |z|_\rho\le 1}\frac{b(z)}{s}e^{-\beta (|z|_\rho^2+|y|_\rho^2)/s}m_p(dz)ds \nonumber
\\
&+& \int_{0}^t \frac{1}{\sqrt{t-s}}e^{-\alpha (|x|^2+|z|_\rho^2)/(t-s)}\int_{z\in  D_\eps, |z|_\rho\le 1}\frac{1}{s^{3/2}}e^{-\beta |z-y|^2/s}b(z)m_p(dz)ds \nonumber
\\
& =&{\rm(I)}+{\rm (II)}. \label{e:3.29}
\end{eqnarray}
We first observe that using the exact same computation as that for Case 2, one can show 
\begin{equation}\label{eq:452}
{\rm (II)}\le \int_{0}^t \frac{1}{\sqrt{t-s}}e^{-\alpha (|x|^2+|z|_\rho^2)/(t-s)}\int_{D_\eps}\frac{1}{s^{3/2}}e^{-\beta |z-y|^2/s}b(z)m_p(dz)ds \lesssim \|b_2\|_{p_2}t^{1/2-1/p_2}p^0_{0, \alpha}(t,x,y).
\end{equation}
On the other hand, to handle (I) in \eqref{e:3.29}, we have:
\begin{align}
{\rm (I)}&= \int_0^t \frac{1}{\sqrt{t-s}}e^{-\alpha (|x|^2+|z|_\rho^2)/(t-s)}\int_{z\in D_\eps, |z|_\rho\le 1}\frac{b(z)}{s}e^{-\beta (|z|_\rho^2+|y|_\rho^2)/s}m_p(dz)ds \nonumber 
\\
&\le  e^{-\alpha(|x|^2+|y|_\rho^2)/t}\int_0^t \frac{1}{(t-s)^{1/2}}\frac{1}{s^{1-1/(2q_2)}}\int_{z\in D_\eps, |z|_\rho\le 1}\frac{b(z)}{s^{1/(2q_2)}}e^{-\beta |z|^2_\rho/s}m_p(dz)ds \nonumber
\\
&\lesssim e^{-\alpha(|x|^2+|y|_\rho^2)/t}\int_0^t \frac{1}{(t-s)^{1/2}}\frac{\|b_2\|_{p_2} }{s^{1-1/(2q_2)}}\left(\int_{z\in D_\eps, |z|_\rho\le 1}\frac{1}{\sqrt{s}}e^{-\beta q_2 |z|^2_\rho/s}m_p(dz)\right)^{1/q_2}ds\nonumber
\\
&\stackrel{r=|z|_\rho}{=}  e^{-\alpha(|x|^2+|y|_\rho^2)/t}\int_0^t \frac{1}{(t-s)^{1/2}}\frac{\|b_2\|_{p_2}}{s^{1-1/(2q_2)}}\left(\int_{0}^1\frac{r+\eps}{\sqrt{s}}e^{-\beta q_2 r^2/s}dr\right)^{1/q_2} ds \nonumber
\\
&\lesssim  \|b_2\|_{p_2}\frac{t^{1/(2q_2)}}{\sqrt{t}}e^{-\alpha(|x|^2+|y|_\rho^2)/t}=\|b_2\|_{p_2}t^{1/(2q_2)}p^0_{0, \alpha}(t,x,y),\label{e:3.30}
\end{align}
where the last inequality in \eqref{e:3.30} follow from 
\begin{align}\label{integral-1-d-BM-density}
\int_{0}^1\frac{r+\eps}{\sqrt{s}}e^{-\beta q_2 r^2/s}dr &\lesssim \int_{0}^1\frac{1}{\sqrt{s}}e^{-\beta q_2 r^2/s}dr \le  \int_{0}^\infty \frac{1}{\sqrt{s}}e^{-\beta q_2 r^2/s}dr =\frac{1}{2}\sqrt{\frac{\pi }{\beta q_2}}.
\end{align}
In view of \eqref{e:3.29}, adding up \eqref{eq:452} and \eqref{e:3.30}   shows for Case 3 that when $|y|_\rho<1$, for $0<\alpha<\beta_2$, 
\begin{align*}
\int_0^t \int_{z\in D_\eps, |z|_\rho<1}p^0_{0,\alpha}(t-s,x,z)b(z)&\left|\nabla_z p(s, z, y )\right| m_p(dz)ds \lesssim \|b_2\|_{p_2}\left(t^{1/(2q_2)}+t^{1/2-1/p_2}\right)p^0_{0, \alpha}(t,x,y).
\end{align*}
Now combining the  two parts for Case 3: $|y|_\rho\ge 1$ and $|y|_\rho<1$, we establish for Case 3 that,  for  any $0<\alpha<\beta_2$ where  $\beta_2$  is chosen in Proposition \ref{P:3.2}, it holds 
\begin{align*}
&\quad \int_0^t \int_{z\in D_\eps}p^0_{0, \alpha}(t-s,x,z)b(z)\left|\nabla_z p(s, z, y )\right| m_p(dz)ds  \lesssim \|b_2\|_{p_2}\left(t^{1/(2q_2)}+t^{1/2-1/p_2}\right)p^0_{0, \alpha}(t,x,y).
\end{align*}
Combining the discussion for Cases 1-3,  we have shown that  for  any $0<\alpha<(\beta_1\wedge \beta_2)$, it holds 
\begin{align*}
 \int_0^t \int_{z\in E}p^0_{0,\alpha}(t-s,x,z)b(z) & \left|\nabla_z p(s, z, y )\right| m_p(dz)ds
\\ 
 & \lesssim \left(\|b_1\|_{p_1}+\|b_2\|_{p_2}\right)\left(t^{1/(2q_1)}+t^{1/(2q_2)}+t^{1/2-1/p_2}\right)p^0_{0, \alpha}(t,x,y).
\end{align*}
\end{proof}
For the case when $x\in D_\eps \cup \{a^*\}$ and $y\in \IR_+$, to get the result analogous to Proposition \ref{1236}, we need to split it further into two propositions depending on whether $|x|_\rho\le 1$ or $|x|_\rho> 1$. When $x\in D_\eps \cup \{a^*\}$ with $|x|_\rho\le 1$, for computation convenience, we rewrite the canonical form of the heat kernel of $X^0$ and name it as $p^0_{1, \alpha}$ as follows:
\begin{equation}\label{canonical-form1}
p^0_{1,\alpha}(t,x,y):=\left\{
  \begin{aligned}
   & \frac{1}{\sqrt{t}}e^{-\alpha |x-y|^2/t} ,\quad
 x\in \IR_+,y\in \IR_+;&
\\ 
& \frac{1}{\sqrt{t}}e^{-\alpha (|x|^2+|y|_\rho^2)/t} ,\quad 
x\in \IR_+,y\in D_\eps \cup \{a^*\} ;& 
\\ 
& \frac{1}{\sqrt{t}}e^{-\alpha (|x|_\rho^2+|y|^2)/t} ,\quad
x\in D_\eps  \cup \{a^*\}, y\in \IR_+;& 
\\ 
 & \frac{1}{t}e^{-\alpha |x-y|^2/t}, \quad  x,y\in D_\eps\cup \{a^*\} \text{ with } \max\{|x|_\rho,|y|_\rho\}> 1;& 
\\ 
&   \frac{1}{\sqrt{t}}e^{-\alpha(|x|^2_\rho+|y|_\rho^2)/t}+\frac{1}{t}\left(1\wedge \frac{|x|_\rho}{\sqrt{t}}\right)\left(1\wedge \frac{|y|_\rho}{\sqrt{t}}\right)e^{-2\alpha |x-y|^2/t}, \text{ otherwise}. & 
  \end{aligned}
\right.
\end{equation}
The only difference between $p^0_{0, \alpha}(t,x,y)$ and $p^0_{1,\alpha}(t,x,y)$ is that in the last display, the coefficient $\alpha$ on the second exponential term is replaced with $2\alpha$. The following lemma is immediate.
\begin{lem}
There exist constants $C_7, C_8>0$ and $\alpha_3, \alpha_4>0$ such that 
\begin{equation*}
C_7\;p^0_{1,\alpha_3}(t,x,y) \le p(t,x,y)\le C_8\;p^0_{1,\alpha_4}(t,x,y), \qquad (t,x,y)\in (0,T]\times E\times E.
\end{equation*}
\end{lem}
\begin{prop}\label{P:3.6}
 It holds for any pair of $\alpha, \beta$ satisfying $0<\alpha<\beta/2$ that
\begin{align}
\int_0^t \frac{1}{\sqrt{s}}\int_{E}p^0_{1,\alpha}(t-s, x,z)|b(z)|& p^0_{1,\beta}(s, z,y)m_p(dz)ds \le C_{9}(t)p^0_{1,\alpha}(t,x,y), \nonumber
\\
&\quad  0<s<t\le T,\, x\in D_\eps \cup \{a^*\}, |x|_\rho\le 1, y \in \IR_+,\label{1015}
\end{align}
where $C_{9}(t)$ is a positive non-decreasing function   in $t$ \emph{(}possibly depending on $\alpha$ and $\beta$\emph{)} with  $C_{9}(t)\rightarrow 0$ as $t \rightarrow 0$. In  particular, there exists some $\alpha_5>0$ such that for each $0<\alpha<\alpha_5$, it holds
\begin{eqnarray}
\int_0^t \int_{z\in E}p^0_{1,\alpha}(t-s,x,z)|b(z)|\left|\nabla_z p(s, z, y )\right| m_p(dz)ds\le C_{10}(t)p^0_{1,\alpha}(t,x,y)\nonumber
\\
 0<s<t\le T, x\in D_\eps \cup \{a^*\}, |x|_\rho\le 1,  y\in \IR_+,\label{e:3.35}
\end{eqnarray}
where $C_{10}(t)$ is   a positive  non-decreasing  function   in $t$   \emph{(}possibly depending on $\alpha$\emph{)} with $C_{10}(t)\rightarrow 0$ as $t \rightarrow 0$.

\end{prop}
\begin{proof}
 We observe that \eqref{1015} implies \eqref{e:3.35} because Proposition \ref{P:3.1} and Proposition \ref{P:3.2} both hold with $p^0_{\beta_i}$ being replaced with $p^0_{1, \beta_i/2}$, $i=1, 2$.
Again without loss of generality we assume $b\ge 0$. \eqref{1015} will be established for $z\in \IR_+$ and $z\in D_\eps\cup \{a^*\}$ separately.\\ 
{\it Case 1.} $z\in \IR_+$. 
For $q_1>0$ satisfying $1/p_1+1/q_1=1$, it holds
 \begin{eqnarray}
&& \int_{0}^t \frac{1}{\sqrt{s}}\int_{\IR_+}p^0_{1,\alpha} (t-s, x,z)b(z)p^0_{1,\beta}(s, z,y)m_p(dz)ds \nonumber
\\
&= &  \int_{0}^t \frac{1}{(t-s)^{1/2}}e^{-\alpha (|x|_\rho^2+|z|^2)/(t-s)}\int_{\IR_+}\frac{1}{s}e^{-\beta |z-y|^2/s}b(z)m_p(dz)ds \nonumber
\\
&=  & \int_{0}^t \frac{1}{(t-s)^{1/2}}e^{-\alpha (|x|_\rho^2+|z|^2)/(t-s)-\alpha |z-y|^2/s}\int_{\IR_+}\frac{1}{s}e^{-(\beta-\alpha) |z-y|^2/s}b(z)m_p(dz)ds \nonumber
\\
&\le &  \|b_1\|_{p_1} e^{-\alpha (|x|_\rho^2+|y|^2)/t} \int_{0}^t \frac{1}{s^{1-1/2q_1}} \frac{1}{(t-s)^{1/2}} \left(\int_{z\in \IR_+}\left(\frac{1}{s^{1/2q_1}}e^{-(\beta-\alpha)|z-y|^2/s}\right)^{q_1} m_p(dz)\right)^{1/q_1} ds  \nonumber
\\
&\lesssim & \|b_1\|_{p_1}  \frac{t^{1/{2q_1}}}{\sqrt{t}}e^{-\alpha(|x|_\rho^2+|y|^2)/t}=\|b_1\|_{p_1}t^{1/(2q_1)}p^0_{1, \alpha}(t,x,y), \label{1210}
 \end{eqnarray}
where  the  $``\le "$ in \eqref{1210}  follows from H\"{o}lder's inequality and  the following triangle inequality  which is similar to \eqref{triangle-ineq-1} and \eqref{triangle-inequality-2}:
$$
|z|^2/(t-s)+|z-y|^2/s\ge |y|^2/t.
$$
{\it Case 2.} $z\in D_\eps\cup \{a^*\}$, $|z|_\rho \ge 1$. We first observe that
\begin{equation}\label{252}
\beta |z|_\rho\ge \alpha (|z|_\rho+\eps)=\alpha |z|, \quad \text{when } |z|_\rho \ge \frac{\alpha \eps}{\beta-\alpha}.
\end{equation}
This in particular implies that  when $\alpha<\beta/2$ and $|z|_\rho\ge \eps$, it holds $\beta|z|_\rho\ge \alpha|z|$. Furthermore, since it is assumed throughout this paper that $0<\eps\le 1/4$, it holds for $0<\alpha<\beta/2 $ that
\begin{equation*}
\beta |z|_\rho\ge \alpha|z|, \quad \text{when }|z|_\rho \ge 1. 
\end{equation*}
We first record the following computation which will be used later in this proof. We temporarily in the following computation denote  by $\wh{p}_2(t,x,y)=(2\pi t)^{-1}\exp(-|x-y|^2/(2t))$ the transition density of a standard $2$-dimensional Brownian motion. We let  $q_2$ satisfy $1/p_2+1/q_2=1$.
\begin{eqnarray}
&& \int_{z\in D_\eps, |z|_\rho\ge 1} \frac{1}{t-s}e^{-\alpha q_2|x-z|^2/(t-s)}\frac{1}{s}e^{-\alpha q_2|z|^2/s}m_p(dz)  \nonumber
\\
&\le & \int_{\IR^2} \frac{1}{t-s}e^{-\alpha q_2|x-z|^2/(t-s)}\frac{1}{s}e^{-\alpha q_2|z|^2/s}dz \nonumber
\\
&= &\int_{\IR^2} \frac{\pi ^2}{\alpha ^2 q_2^2}\, \wh{p}_2\left(\frac{t-s}{2\alpha q_2}, x, z\right)\wh{p}_2\left(\frac{s}{2\alpha q_2}, z, 0\right)\,dz  \nonumber
\\
&= &\frac{ \pi ^2}{\alpha^2 q_2^2} \, \wh{p}_2\left(\frac{t}{2\alpha q_2}, x, 0\right)\asymp \frac{1}{t} e^{-\alpha q_2|x|^2/t}. \label{semigroup-2-d-BM}
\end{eqnarray}
Therefore,
\begin{eqnarray}
&& \int_{0} ^t\frac{1}{\sqrt{s}}\int_{z\in D_\eps, |z|_\rho\ge 1}p^0_{1,\alpha} (t-s, x,z)b(z)p^0_{1,\beta}(s, z,y)m_p(dz)ds \nonumber
\\
&= & \int_{0}^t \frac{1}{t-s}e^{-\alpha |x-z|^2/(t-s)}\int_{z\in D_\eps, |z|_\rho\ge 1} b(z)\frac{1}{s}e^{-\beta(|y|^2+|z|_\rho^2)/s}m_p(dz)ds \nonumber
\\
&\stackrel{\eqref{252}}{\le} & e^{-\beta |y|^2/t}\int_{0}^t \int_{z\in D_\eps, |z|_\rho\ge 1} \frac{1}{t-s}e^{-\alpha |x-z|^2/(t-s)}\cdot b(z)\frac{1}{s}e^{-\alpha |z|^2/s}m_p(dz)ds \nonumber
\\
&\lesssim & \|b_2\|_{p_2}e^{-\beta |y|^2/t}\int_{0}^t \frac{1}{(t-s)^{1-1/q_2}}\frac{1}{s^{1-1/q_2}}\nonumber
\\
&\times &\left(\int_{z\in D_\eps, |z|_\rho\ge 1} \frac{1}{t-s}e^{-\alpha q_2|x-z|^2/(t-s)}\frac{1}{s}e^{-\alpha q_2|z|^2/s}m_p(dz)\right)^{1/q_2 } ds  \nonumber
\\
&\stackrel{\eqref{semigroup-2-d-BM}}{\lesssim} & e^{-(\beta |y|^2+\alpha |x|^2)/t}\frac{\|b_2\|_{p_2}}{t^{1/q_2}}\int_{0}^t \frac{1}{(t-s)^{1-1/q_2}}\frac{1}{s^{1-1/q_2}}ds\nonumber
\\
&\lesssim &\frac{\|b_2\|_{p_2}}{t^{1-1/q_2}}e^{-\alpha(|x|^2+|y|^2)/t } \le  \frac{\|b_2\|_{p_2}}{t^{1/p_2}}e^{-\alpha(|x|_\rho^2+|y|^2)/t }=\|b_2\|_{p_2}t^{1/2-1/p_2}p^0_{1, \alpha}(t,x,y). \label{eq:3.32}
\end{eqnarray}
 where $1/2-1/p_2>0$, given that $p_2\in (2,\infty]$. 
 \\
{\it Case 3.} $z\in D_\eps\cup \{a^*\}$, $|z|_\rho<1$.  We recall that when $\max\{|x|_\rho, |z|_\rho\}\le 1$,
\begin{align}
p^0_{1,\alpha}(t-s,x,z)&= \frac{1}{\sqrt{t-s}}e^{-\alpha(|x|_\rho^2+|z|_\rho^2)/(t-s)}+\frac{1}{t-s}\left(1\wedge \frac{|x|_\rho}{\sqrt{t-s}}\right)\left(1\wedge \frac{|z|_\rho}{\sqrt{t-s}}\right)e^{-2\alpha|x-z|^2/(t-s)}\nonumber 
\\
&\le \frac{1}{\sqrt{t-s}}e^{-\alpha(|x|_\rho^2+|z|_\rho^2)/(t-s)}+\frac{1}{t-s}e^{-2\alpha|x-z|^2/(t-s)}.\label{eq:3.33}
\end{align}
Therefore, 
\begin{eqnarray}
&& \int_{0}^t\frac{1}{\sqrt{s}}\int_{z\in D_\eps\cup \{a^*\}, |z|_\rho<1}p^0_{1,\alpha} (t-s, x,z)b(z)p^0_{1,\beta}(s, z,y)m_p(dz)ds \nonumber
\\
&\le & \int_{0}^t \frac{1}{\sqrt{s}} \int_{z\in D_\eps \cup \{a^*\}, |z|_\rho<1} \frac{1}{\sqrt{t-s}}e^{-\alpha (|x|_\rho^2+|z|^2_\rho)/(t-s)}b(z)p^0_{1,\beta}(s, z,y)m_p(dz)ds \nonumber
\\
&+& \int_{0}^t\frac{1}{\sqrt{s}}\int_{z\in D_\eps\cup \{a^*\}, |z|_\rho< 1}\frac{1}{t-s}e^{-2\alpha|x-z|^2/(t-s)}b(z)p^0_{1,\beta}(s,z,y)m_p(dz)ds 
\\
&= & {\rm(I})+{\rm (II)}. \label{eq:3.40}
\end{eqnarray}
For (I) on the right hand side of \eqref{eq:3.40}, it holds for $q_2:\, 1/p_2+1/q_2=1$ that
\begin{align}
{\rm(I)}&= \int_{0}^t \frac{1}{\sqrt{s}} \int_{|z|_\rho<1} \frac{1}{\sqrt{t-s}}e^{-\alpha (|x|_\rho^2+|z|^2_\rho)/(t-s)}b(z)p^0_{1,\beta}(s, z,y)m_p(dz)ds \nonumber
\\
&= \int_{0}^t \frac{1}{(t-s)^{1/2}}e^{-\alpha (|x|_\rho^2+|z|_\rho^2)/(t-s)}\int_{|z|_\rho<1}\frac{1}{s}e^{-\beta (|z|_\rho^2+|y|^2)/s}b(z)m_p(dz)ds \nonumber
\\
&\lesssim e^{-\alpha (|x|_\rho^2+|y|^2)/t} \int_{0}^t \frac{1}{s^{1-1/(2q_2)}} \frac{ \|b_2\|_{p_2}}{(t-s)^{1/2}} \left(\int_{|z|_\rho<1}\left(\frac{1}{s^{1/(2q_2)}}e^{-(\beta+\alpha)|z|_\rho^2/s}\right)^{q_2} m_p(dz)\right)^{1/q_2} ds \nonumber
\\
&\lesssim \|b_2\|_{p_2} \frac{t^{1/{(2q_2)}}}{\sqrt{t}}e^{-\alpha(|x|_\rho^2+|y|^2)/t}=\|b_2\|_{p_2}t^{1/(2q_2)}p^0_{1, \alpha}(t,x,y), \label{510}
\end{align}
where the last $``\lesssim"$ holds because the integral in $z$ is comparable to a constant, by adopting polar coordinates and  computation similar  to   that for \eqref{integral-1-d-BM-density}. To handle (II) on the right hand side of \eqref{eq:3.40}, we  first record the following computation over  $\{z\in D_\eps \cup \{a^*\}:\; |z|_\rho<1\}$, which will be used in this proof. Temporarily in the following computation, we denote by $p_1(t,x,y)=(2\pi t)^{-1/2}\exp (-|x-y|^2/(2t))$ the transition density for a $1$-dimensional standard Brownian motion.
\begin{eqnarray}
&& \int_{z\in D_\eps \cup \{a^*\}, |z|_\rho<1}\frac{1}{\sqrt{t-s}}e^{-\alpha|x-z|^2/(t-s)}\frac{1}{\sqrt{s}}e^{-\alpha(|z|_\rho^2+|y|^2)/s}m_p(dz) \nonumber
\\
&\le & e^{-\alpha |y|^2/t}\int_{z\in D_\eps \cup \{a^*\}, |z|_\rho<1} \frac{1}{\sqrt{t-s}}e^{-\alpha(|x|-|z|)^2/(t-s)}\frac{1}{\sqrt{s}}e^{-\alpha |z|^2_\rho/s}m_p(dz) \nonumber
\\
&\stackrel{r=|z|}{=}&e^{-\alpha |y|^2/t}\int_{\eps}^{1+\eps} \frac{r}{\sqrt{t-s}}e^{-\alpha (|x|-r)^2/(t-s)}\frac{1}{\sqrt{s}}e^{-\alpha (r-\eps)^2/s} \,dr \nonumber
\\
&\lesssim &e^{-\alpha |y|^2/t}\int_{\eps}^{1+\eps} \frac{1}{\sqrt{t-s}}e^{-\alpha (|x|-r)^2/(t-s)}\frac{1}{\sqrt{s}}e^{-\alpha (r-\eps)^2/s} dr \nonumber
\\
&\le  & e^{-\alpha |y|^2/t}\int_0^\infty  \frac{\pi }{\alpha } \,p_1\left(\frac{t-s}{2\alpha},|x|, r\right) p_1\left(\frac{s}{2\alpha}, r, \eps \right) dr \nonumber
\\
&\lesssim & e^{-\alpha |y|^2/t}\, p_1\left(\frac{t}{2\alpha}, |x|, \eps\right)   \lesssim \frac{1}{\sqrt{t}}e^{-\alpha [|y|^2+(|x|-\eps)^2]/t} = \frac{1}{\sqrt{t}}e^{-\alpha (|x|_\rho^2+|y|^2)/t}. \label{608}
\end{eqnarray}
With   \eqref{608}, we choose $q'_2,r_2>0$ satisfying $1/p_2+1/q'_2+1/r_2=1$. Since  $p_2\in (2, \infty]$, we may pick $q'_2$ sufficiently large so that $1/(2q'_2)+1/p_2<1/2$.  It then follows that 
\begin{align}
{\rm (II)}&= \int_{0}^t\frac{1}{\sqrt{s}}\int_{z\in D_\eps\cup \{a^*\}, |z|_\rho< 1}\frac{1}{t-s}e^{-2\alpha|x-z|^2/(t-s)}b(z)p^0_{1,\beta}(s,z,y)m_p(dz)ds \nonumber
\\
&= \int_0^t \frac{1}{\sqrt{s}}\int_{ |z|_\rho<1}\frac{1}{t-s}b(z)e^{-2\alpha|x-z|^2/(t-s)}\frac{1}{\sqrt{s}}e^{-\beta(|y|^2+|z|_\rho^2)/s}m_p(dz)ds \nonumber
\\
\nonumber &\le \int_0^t \frac{1}{s^{1-1/(2q'_2)}}\frac{1}{(t-s)^{1-1/(2q'_2)-1/r_2}}\int_{ |z|_\rho<1}\frac{1}{s^{1/(2q'_2)}}e^{-\alpha(|y|^2+|z|_\rho^2)/s}\frac{1}{(t-s)^{1/(2q'_2)}}e^{-\alpha |x-z|^2/(t-s)}
\\
&\qquad \cdot \frac{b(z)}{(t-s)^{1/r_2}}e^{-\alpha |x-z|^2/(t-s)}m_p(dz)ds \nonumber
\\
\nonumber &\lesssim \int_0^t \frac{1}{s^{1-1/(2q'_2)}}\frac{\|b_2\|_{p_2}}{(t-s)^{1-1/(2q'_2)-1/r_2}} \left(\int_{ |z|_\rho<1}\frac{1}{\sqrt{s}}e^{-\frac{q'_2\alpha(|y|^2+|z|_\rho^2)}{s}}\frac{1}{\sqrt{t-s}}e^{-\frac{q'_2\alpha |x-z|^2}{t-s}}m_p(dz)\right)^{1/q'_2}\nonumber
\\
&\quad \cdot \left(\int_{ |z|_\rho<1} \frac{1}{t-s}e^{-\alpha r_2|x-z|^2/(t-s)} m_p(dz)\right)^{1/r_2} ds \nonumber \stackrel{\eqref{608}}{\lesssim}\frac{\|b_2\|_{p_2} }{t^{1-1/(q'_2)-1/r_2}}\frac{1}{t^{1/(2q'_2)}}e^{-\alpha (|x|_\rho^2+|y|^2)/t}
\\
&=\frac{\|b_2\|_{p_2}}{t^{1/p_2+1/(2q'_2)}}e^{-\alpha(|x|_\rho^2+|y|^2)/t} =\|b_2\|_{p_2}t^{1/2-1/p_2-1/(2q'_2)}p^0_{1, \alpha}(t,x,y),\label{eq:3.36}
\end{align}
where it has also been used in the last $``\lesssim"$ that 
\begin{align*}
\int_{ |z|_\rho<1} \frac{1}{t-s}e^{-\alpha r_2|x-z|^2/(t-s)} m_p(dz)&\lesssim  \int_{ \IR^2} \frac{1}{t-s}e^{-\alpha r_2|x-z|^2/(t-s)} dz=\frac{\pi}{\alpha r_2}.
\end{align*}
In view of \eqref{eq:3.40}, adding up  \eqref{510} and \eqref{eq:3.36} yields for $0<\alpha<\beta/2$ that
\begin{align}\label{conclusion-(II)}
\int_{0}^t\frac{1}{\sqrt{s}}\int_{z\in D_\eps\cup \{a^*\}, |z|_\rho<1}p^0_{1,\alpha} (t-s, x,z)&b(z)p^0_{1,\beta}(s, z,y)m_p(dz)ds  \nonumber
\\
&\le \|b_2\|_{p_2}(t^{1/(2q_2)}+t^{1/2-1/(2q_2')-1/p_2})p^0_{1, \alpha}(t,x,y).
\end{align}
Finally, combining Cases 1-3  shows for $0<\alpha<\beta/2$ that
 \begin{align*}
\int_{0}^t\frac{1}{\sqrt{s}}&\int_{E}p^0_{1,\alpha} (t-s, x,z)b(z)p^0_{1,\beta}(s, z,y)m_p(dz)ds 
\\
&\lesssim \left(\|b_1\|_{p_1}+\|b_2\|_{p_2} \right)\left(t^{1/(2q_1)}+t^{1/2-1/p_2}+ t^{1/(2q_2)}+t^{1/2-1/(2q_2')-1/p_2}  \right)p^0_{1, \alpha}(t,x,y).
\end{align*}
 The second conclusion of this proposition follows readily from Proposition \ref{P:3.1}. 
 \end{proof}

For the next case which is $x\in D_\eps$ with $|x|_\rho>1$ and $y\in \IR_+$, we rewrite the canonical form of the heat kernel of $X^0$ for computation convenience as follows : 
\begin{equation}\label{canonical-form2}
p^0_{2,\alpha}(t,x,y):=\left\{
  \begin{aligned}
   & \frac{1}{\sqrt{t}}e^{-\alpha |x-y|^2/t},\quad 
 x\in \IR_+,y\in \IR_+;&
\\ 
& \frac{1}{\sqrt{t}}e^{-\alpha (|x|^2+|y|_\rho^2)/t},\quad
x\in \IR_+,y\in D_\eps\cup \{a^*\};& 
\\ 
& \frac{1}{\sqrt{t}}e^{-\alpha (|x|_\rho^2+|y|^2)/t},\quad
x\in D_\eps \cup \{a^*\}, y\in \IR_+;& 
\\
 & \frac{1}{t}e^{-2\alpha |x-y|^2/t},\quad   x,y\in D_\eps\cup \{a^*\} \text{ with } \max\{|x|_\rho,|y|_\rho\}> 1;& 
\\ 
&   \frac{1}{\sqrt{t}}e^{-\alpha(|x|^2_\rho+|y|_\rho^2)/t}+\frac{1}{t}\left(1\wedge \frac{|x|_\rho}{\sqrt{t}}\right)\left(1\wedge \frac{|y|_\rho}{\sqrt{t}}\right)e^{-\alpha |x-y|^2/t}, \text{ otherwise}. & 
  \end{aligned}
\right.
\end{equation}
The difference between $p^0_{0, \alpha}(t,x,y)$ and $p^0_{2,\alpha}(t,x,y)$ is that for the case $x,y\in D_\eps \cup \{a^*\}$ with $\max\{|x|_\rho, |y|_\rho\}>1$, the coefficient $\alpha$ on the  exponential term is replaced with $2\alpha$. As before, the following lemma is immediate.
\begin{lem}
 There exist constants $C_{11}, C_{12}>0$ and $\alpha_6, \alpha_7>0$ such that 
\begin{equation*}
C_{11}p^0_{2,\alpha_6}(t,x,y)\le p(t,x,y)\le C_{12}\;p^0_{2,\alpha_7}(t,x,y), \qquad (t,x,y)\in (0,T]\times E\times E.
\end{equation*}
\end{lem}

\begin{prop}\label{P:3.8}
It holds for any $\alpha, \beta$ satisfying $0<\alpha<\beta/2$ that
\begin{align}
\int_0^t \frac{1}{\sqrt{s}}\int_{E}p^0_{2,\alpha}(t-s, x,z)b(z)p^0_{2,\beta}(s, z,y)& m_p(dz)ds \le C_{13}(t)p^0_{2,\alpha}(t,x,y), \nonumber
\\
&\quad  0<s<t\le T,\, x\in D_\eps, |x|_\rho> 1, y \in \IR_+,\label{eq:3.38}
\end{align}
where $C_{13}(t)$ is a positive non-decreasing function in $t$  \emph{(}possibly depending on $\alpha$ and $\beta$\emph{)} and $C_{13}(t)\rightarrow 0$ as $t\rightarrow 0$.  In particular, there exists some $\alpha_8>0$ such that for each $0<\alpha<\alpha_8$, 
\begin{eqnarray}
\int_0^t \int_{z\in E}p^0_{2,\alpha}(t-s,x,z)|b(z)|\left|\nabla_z p(s, z, y )\right| m_p(dz)ds\le C_{14}(t)p^0_{2,\alpha}(t,x,y)\nonumber
\\
 0<s<t\le T, x\in D_\eps, |x|_\rho> 1, y\in \IR_+,\label{eq:3.39}
\end{eqnarray}
where $C_{14}(t)$ is a positive non-decreasing function in $t$  \emph{(}possibly depending on $\alpha$\emph{)} and $C_{14}(t)\rightarrow 0$ as $t\rightarrow 0$.
\end{prop}
\begin{proof}  We observe that \eqref{eq:3.38} implies \eqref{eq:3.39} because Proposition \ref{P:3.1} and Proposition \ref{P:3.2} both hold with $p^0_{\beta_i}$ being replaced with $p^0_{2, \beta_i/2}$, $i=1, 2$. Again \eqref{eq:3.38} will be established by dividing our discussion into several cases depending on the position of $z$. Without loss of generality, we assume $b\ge 0$.\\
{\it Case 1. }$z\in \IR_+$. Following  the exact same computation as that  for Case 1 of Proposition \ref{P:3.6}, we can show for $q_1$ satisfying $1/p_1+1/q_1=1$ that
\begin{align*}
\int_{0}^t \frac{1}{\sqrt{s}}\int_{\IR_+}p^0_{1,\alpha} (t-s, x,z)b(z)p^0_{1,\beta}(s, z,y)m_p(dz)ds \lesssim \|b_1\|_{p_1}t^{1/(2q_1)}p^0_{1, \alpha}(t,x,y).
\end{align*}
{\it Case 2. }$z\in D_\eps\cup \{a^*\}$. For this case, we choose $q_2,r_2>0$ satisfying $1/p_2+1/q_2+1/r_2=1$. Since $p_2\in (2, \infty]$, we may pick $q_2$ sufficiently large such that $1/(2q_2)+1/p_2<1/2$.  Dividing the integral into two parts: $\{|z|_\rho>1\}$ and $\{|z|_\rho\le 1\}$, then  following the same computation as that for  \eqref{eq:3.32} and \eqref{eq:3.36} in  the proof to Proposition \ref{P:3.6}, one can show that
\begin{eqnarray}
&& \int_{0}^t\frac{1}{\sqrt{s}}\int_{D_\eps\cup \{a^*\}}p^0_{2,\alpha} (t-s, x,z)b(z) p^0_{2,\beta}(s,z,y)m_p(dz)ds\nonumber
\\
&=&  \int_0^t \frac{1}{\sqrt{s}}\int_{D_\eps\cup \{a^*\}}\frac{1}{t-s}b(z)e^{-2\alpha|x-z|^2/(t-s)}\frac{1}{\sqrt{s}}e^{-\beta(|y|^2+|z|_\rho^2)/s}m_p(dz)ds \nonumber
\\
&\lesssim &  \|b_2\|_{p_2} t^{1/2-1/p_2-1/(2q_2)}\left(1+t^{1/(2q_2)}\right)p^0_{2, \alpha}(t,x,y) \nonumber
\\
&\stackrel{t\le T}{\le }& \|b_2\|_{p_2} t^{1/2-1/p_2-1/(2q_2)}p^0_{2, \alpha}(t,x,y).\label{eq:3.41}
\end{eqnarray}
The details of  the computation  to the $``\lesssim"$ of  \eqref{eq:3.41}  are omitted here. Combining  Cases  1-2  yields that 
\begin{align*}
\int_{0}^t\frac{1}{\sqrt{s}}\int_{E}p^0_{2,\alpha} (t-s, x,z)b(z) &p^0_{2,\beta}(s,z,y)m_p(dz)ds 
\\
&\lesssim  \left(\|b_1\|_{p_1}+\|b_2\|_{p_2}\right)\left(t^{1/(2q_1)}+t^{1/2-1/p_2-1/(2q_2)}\right)p^0_{2, \alpha}(t,x,y).
\end{align*}
Again the second conclusion of this proposition follows from Proposition \ref{P:3.1}.
\end{proof}
For  $\beta_1, \beta_2>0$ be chosen as in Proposition \ref{P:3.1} and Proposition \ref{P:3.2}, we let
\begin{equation}\label{def-beta3}
\beta_3=(\beta_1\wedge \beta_2)/16 
\end{equation}
be fixed for the rest of this section.  For the rest of this section, we set a  constant $M=48$.  When $|y|_\rho\ge M=48$, noting that $\eps\in (0, 1/4)$, by elementary algebra one can show 
\begin{equation}\label{e:3.52r}
\frac{1}{4}\left(|y|_\rho^2-8|y|_\rho\right)\ge \frac{1}{4}\left(16|y|_\rho+12\right)\ge \frac{1}{4}\left(16|y|_\rho+48\eps\right)=\frac{1}{4}\left(16|y|+32\eps\right)=4|y|+8\eps.
\end{equation}
In particular, when $z\in D_\eps$, $|z|_\rho<4$, and $|y|_\rho\ge M=48$,
\begin{equation}\label{choice-of-M}
\frac{1}{4}|z-y|^2\ge \frac{1}{4}\left(|y|_\rho^2-2|z|_\rho |y|_\rho +|z|_\rho^2\right)\ge \frac{1}{4}\left(|y|_\rho^2-8|y|_\rho\right)\stackrel{\eqref{e:3.52r}}{\ge } 4|y|+\eps.
\end{equation}

For the next proposition, we let the canonical form be as follows:
\begin{equation}\label{canonical-form2}
p^0_{3,\alpha}(t,x,y):=\left\{
  \begin{aligned}
   & \frac{1}{\sqrt{t}}e^{-\alpha |x-y|^2/t} ,\quad
 x\in \IR_+,y\in \IR_+;&
\\ 
& \frac{1}{\sqrt{t}}e^{-4\alpha (|x|^2+|y|_\rho^2)/t} ,\quad
x\in \IR_+,y\in D_\eps \cup \{a^*\} ;& 
\\ 
& \frac{1}{\sqrt{t}}e^{-4\alpha (|x|_\rho^2+|y|^2)/t},\quad
x\in D_\eps \cup \{a^*\}, y\in \IR_+;& 
\\ 
 & \frac{1}{t}e^{-\alpha |x-y|^2/t}, \quad  x,y\in D_\eps \cup \{a^*\} \text{ with } \max\{|x|_\rho,|y|_\rho\}> 1;& 
\\ 
&   \frac{1}{\sqrt{t}}e^{-\alpha(|x|^2_\rho+|y|_\rho^2)/t}+\frac{1}{t}\left(1\wedge \frac{|x|_\rho}{\sqrt{t}}\right)\left(1\wedge \frac{|y|_\rho}{\sqrt{t}}\right)e^{-\alpha |x-y|^2/t}, \text{ otherwise}. & 
  \end{aligned}
\right.
\end{equation}
Note that the difference between $p^0_{3, \alpha}(t,x,y)$ and $p^0_{0,\alpha}(t,x,y)$ is that for the case when one of $x$ and $y$ is in $\IR_+$ and the other variable is in $D_\eps \cup \{a^*\}$,  the constant involved in the exponential term is $4\alpha$ instead of $\alpha$. The next lemma follows immediately. 
\begin{lem}
 There exists constants $C_{15}, C_{16}>0$, $\alpha_9, \alpha_{10}>0$ such that it holds  that
\begin{equation*}
C_{15}\,p^0_{3,\alpha_9}(t,x,y)\le p(t,x,y)\le C_{16}\,p^0_{3,\alpha_{10}}(t,x,y), \quad t\in (0, T],\,(x,y)\in  E\times E.
\end{equation*}
\end{lem}

\begin{prop}\label{1100}
 Let $M=48$  so that \eqref{choice-of-M} holds.
It holds for  any $\alpha, \beta$ satisfying $0<\alpha<\beta/4$ that  there exists some function  $C_{17}(t)>0$ \emph{(}possibly depending on $\alpha$ and $\beta$\emph{)}  which  is non-decreasing in $t$ with  $C_{17}(t)\rightarrow 0$ as $t \rightarrow 0$, such that
\begin{align}
\int_0^t \frac{1}{\sqrt{s}}\int_{E}p^0_{3, \alpha}(t-s, x,z)b(z)&p^0_{3,\beta}(s, z,y)m_p(dz)ds \le C_{17}(t)p^0_{3, \alpha}(t,x,y), \nonumber
\\
&\quad 0<s<t\le T,\, x,y\in D_\eps \cup \{a^*\}, \text{ and } |y|_\rho>M,\label{611}
\end{align}
In particular,
 there exists some $\alpha_{11}>0$ such that for each $0<\alpha<\alpha_{11}$,  there exists some $C_{18}(t)>0$ \emph{(}possibly depending on $\alpha$\emph{)} which  is a positive   non-decreasing  function in $t$, with $C_{18}(t)\rightarrow 0$ when $t\rightarrow 0$, such that
\begin{align}
\int_0^t \int_{z\in E}p^0_{3, \alpha}(t-s,x,z)&|b(z)|\left|\nabla_z p(s, z, y )\right| dzds\le C_{18}(t)p^0_{3, \alpha}(t,x,y), \nonumber
\\
 & 0<s<t\le T,\, x,y\in D_\eps \cup \{a^*\},\, \text{and }|y|_\rho>M.\label{858}
\end{align}
\end{prop}

\begin{proof}
We first note that \eqref{611} implies \eqref{858} because Proposition \ref{P:3.2} holds when the $p_{\beta _2}^0$ there is replaced with $p^0_{3, \beta_2/4}$. Again we divide our proof into three cases depending on the position of $z$.  Without loss of generality, we assume that $b\ge 0$. \\
{\it Case 1.} $z\in \IR_+$. Observe that for this case, since $|y|_\rho>M=48$ and $\eps <1/4$,
\begin{equation}\label{443}
4(|x|_\rho^2+|y|_\rho^2)\ge |x|_\rho^2+|y|_\rho^2+2|x|_\rho|y|_\rho+4\eps (|x|_\rho+|y|_\rho+\eps)=(|x|_\rho+|y|_\rho+2\eps)^2\geq |x-y|^2.
\end{equation}
With \eqref{443}, we have for $0<\alpha<\beta/4$ and $q_1: 1/p_1+1/q_1=1$ that
\begin{eqnarray}
&&\int_{0}^t \frac{1}{\sqrt{s}}\int_{\IR_+}p^0_{3, \alpha} (t-s, x,z)b(z)p^0_{3, \beta}(s, z,y)m_p(dz)ds \nonumber
\\
&= & \int_0^t \frac{1}{\sqrt{s}}\int_{\IR_+}\frac{1}{\sqrt{t-s}}e^{-4\alpha (|x|_\rho^2+|z|^2)/(t-s)}b(z)\frac{1}{\sqrt{s}}e^{-4\beta(|z|^2+|y|_\rho^2)/s}m_p(dz)ds \nonumber
\\
&\le & e^{-4\alpha (|x|_\rho^2+|y|^2_\rho)/t} \int_0^t s^{-1+1/(2q_1)}(t-s)^{-1/2}\int_{\IR_+}b(z)s^{-1/(2q_1)}e^{-4\beta|z|^2/s}m_p(dz)ds \nonumber
\\
&\le & e^{-4\alpha (|x|_\rho^2+|y|^2_\rho)/t}\int_0^t s^{-1+1/(2q_1)}(t-s)^{-1/2} \|b_{1}\|_{p_1} \left(\int_{\IR_+}\left(s^{-1/(2q_1)}e^{-4\beta|z|^2/s}\right)^{q_1} m_p(dz)\right)^{1/q_1}ds \nonumber
\\
&\lesssim & \|b_1\|_{p_1} e^{-4\alpha (|x|_\rho^2+|y|^2_\rho)/t}\cdot t^{1/(2q_1)-1/2} \nonumber
\\
& \stackrel{\eqref{443}}{\le } & \|b_1\|_{p_1} e^{-\alpha |x-y|^2/t}\cdot t^{1/(2q_1)-1/2}  \le \|b_1\|_{p_1}t^{1/(2q_1)+1/2}\;p^0_{3, \alpha}(t,x,y),\label{e:3.48}
\end{eqnarray}
where the last $``\lesssim"$ in \eqref{e:3.48} holds because
\begin{equation*}
\int_{\IR_+}\left(s^{-1/(2q_1)}e^{-4\beta|z|^2/s}\right)^{q_1} m_p(dz)=\int_{\IR_+}s^{-1/2}e^{-4q_1\beta|z|^2/s}dz=\frac{1}{4}\sqrt{\frac{\pi}{q_1\beta}}.
\end{equation*}
{\it Case 2.} $z\in D_\eps, \, |z|_\rho\ge 4$. Let $0<\alpha<\beta/4$.  For $q_2$ satisfying  $1/p_2+1/q_2=1$, by H\"{o}lder's inequality it holds
\begin{eqnarray}
&& \int_{0}^t \frac{1}{\sqrt{s}}\int_{z\in D_\eps, |z|_\rho\ge 4}p^0_{3,\alpha} (t-s, x,z)b(z)p^0_{3,\beta}(s, z,y)m_p(dz)ds \nonumber
\\
&= & \int_{0}^t\frac{1}{\sqrt{s}} \int_{|z|_\rho \ge 4}\frac{1}{t-s}e^{-\alpha |x-z|^2/(t-s)}\frac{1}{s}e^{-\beta |y-z|^2/s}b(z)m_p(dz)ds \nonumber
\\
&\lesssim & \int_{0}^t \frac{1}{(t-s)^{1-1/q_2}}\frac{\|b_2\|_{p_2}}{s^{3/2-1/q_2}} \nonumber
\\
&\times &  \left(\int_{|z|_\rho\ge 4} \left(\frac{1}{(t-s)^{1/q_2}}e^{-\alpha q_2|x-z|^2/(t-s)}\frac{1}{s^{1/q_2}}e^{-\alpha q_2|y-z|^2/s}\right)^{q_2}m_p(dz)\right)^{1/q_2} ds \nonumber
\\
&\le & \int_{0}^t \frac{1}{(t-s)^{1-1/q_2}}\frac{\|b_2\|_{p_2}}{s^{3/2-1/q_2}}\left(\int_{|z|_\rho\ge 4}\frac{1}{t-s}e^{-\alpha q_2|x-z|^2/(t-s)}\frac{1}{s}e^{-\alpha q_2|y-z|^2/s}m_p(dz)\right)^{1/q_2} ds  \nonumber
\\
&\lesssim & \frac{\|b_2\|_{q_2}}{t^{3/2-2/q_2}}\frac{1}{t^{1/q_2}}e^{-\alpha|x-y|^2/t}=\|b_2\|_{p_2}t^{1/q_2-1/2}\;p^0_{3, \alpha}(t,x,y), \label{257}
\end{eqnarray}
where the last $``\lesssim"$ follows from
\begin{align*}
\int_{|z|_\rho\ge 4}\frac{1}{t-s}e^{-\alpha q_2|x-z|^2/(t-s)}\frac{1}{s}e^{-\alpha q_2|y-z|^2/s}m_p(dz) \lesssim \frac{1}{t}e^{-\alpha q_2 |x-y|^2/t},
\end{align*}
by  the   computation similar to  that for \eqref{semigroup-2-d-BM}.

For the case $z\in D_\eps \cup \{a^*\}$ with $|z|_\rho<4$, we further divide it into two subcases depending on whether $|x|_\rho\ge 1$ or $|x|_\rho<1$.
\\
{\it Case 3.} $z\in D_\eps\cup \{a^*\},\,|z|_\rho<4, \,|x|_\rho\ge 1 $. For this case, provided $0<\alpha<\beta/4$,  we can adopt the same computation as that for Case 2 and obtain for $q_2: 1/p_2+1/q_2=1$ that
\begin{align*}
\int_{0}^t \frac{1}{\sqrt{s}}\int_{z\in D_\eps\cup \{a^*\}, |z|_\rho< 4}p^0_{3,\alpha} (t-s, x,z)b(z)p^0_{3,\beta}(s, z,y)m_p(dz)ds 
&\lesssim \|b_2\|_{p_2}t^{1/q_2-1/2}\;p^0_{3, \alpha}(t,x,y), \label{606}
\end{align*}
where again $1/q_2>1/2$ because $p_2>2$. 
\\
{\it Case 4. }$z\in D_\eps \cup \{a^*\},\, |z|_\rho<4, \,|x|_\rho<1$.  For this case, whether $|z|_\rho>1$ or not,
\begin{equation*}
p^0_{3,\alpha} (t-s,x,z)\le \frac{1}{\sqrt{t-s}}e^{-\alpha(|x|_\rho^2+|z|^2_\rho)/(t-s)}+\frac{1}{t-s}e^{-\alpha |x-z|^2/(t-s)}.
\end{equation*}
Therefore, we write
\begin{eqnarray}
&&\int_{0}^t \int_{|z|_\rho<4}p^0_{3, \alpha} (t-s, x,z)b(z)p^0_{3, \beta}(s, z,y)m_p(dz)ds \nonumber
\\
&\le & \int_{0}^t\frac{1}{\sqrt{s}} \int_{|z|_\rho < 4}\frac{1}{\sqrt{t-s}}e^{-\alpha(|x|_\rho^2+|z|^2_\rho)/(t-s)}\frac{1}{s}e^{-\beta |y-z|^2/s}b(z)m_p(dz)ds \nonumber
\\
&+ & \int_{0}^t\frac{1}{\sqrt{s}} \int_{|z|_\rho < 4}\frac{1}{t-s}e^{-\alpha |x-z|^2/(t-s)}\frac{1}{s}e^{-\beta |y-z|^2/s}b(z)m_p(dz)ds ={\rm (I)}+{\rm (II)} \label{e:3.58}
\end{eqnarray}
We observe that (II) in  \eqref{e:3.58} can be handled in exact same way as that for   Case 3, i.e., 
\begin{align}
{\rm (II)} \lesssim e^{-\alpha|x-y|^2/t}\frac{\|b_2\|_{p_2}}{t^{3/2-1/q_2}}=\|b_2\|_{p_2}t^{1/q_2-1/2}p^0_{3, \alpha}(t,x,y),\label{927}
\end{align}
where $1/q_2=1/p_2$. To handle (I) in  \eqref{e:3.58}, we recall that with $M=48$, \eqref{choice-of-M} holds. We also observe that when  $|x|_\rho<1$, $|y|_\rho\ge M=48$, and $\eps\in (0, 1/4)$,  it holds 
\begin{align}
(|x|_\rho+|y|_\rho+2\eps)^2 &= |x|_\rho^2+|y|_\rho^2+4\eps^2+4\eps (|x|_\rho+|y|_\rho)+2|x|_\rho |y|_\rho \nonumber
\\
&\le |x|_\rho^2+|y|_\rho^2+3|y|_\rho+\eps +4\eps \le |x|_\rho^2+|y|_\rho^2+8\eps+4|y|_\rho. \label{104}
\end{align}
Also  when $|y|_\rho\ge 48$ and $|z|_\rho<4$,   it must hold 
\begin{equation}\label{eq:3.61}
|z-y|^2 \ge |y|^2/2
\end{equation}
as well as
\begin{equation}\label{eq:3.62}
|z-y|^2/4>8\eps+4|y|. 
\end{equation}
Thus, for  $0<\alpha<\beta/4$   and $q_2: 1/p_2+1/q_2=1$   we have 
\begin{align}
{\rm (I)}&= \int_{0}^t \frac{1}{(t-s)^{1/2}}e^{-\alpha (|x|_\rho^2+|z|_\rho^2)/(t-s)}\int_{|z|_\rho<4}\frac{1}{s^{3/2}}e^{-\beta |z-y|^2/s}b(z)m_p(dz)ds \nonumber
\\
&\le \int_{0}^t \frac{1}{(t-s)^{1/2}}e^{-\alpha |x|_\rho^2/t}\int_{|z|_\rho<4}\frac{1}{s^{3/2}}e^{-\beta |z-y|^2/4s}e^{-\beta |z-y|^2/2s}e^{-\beta |z-y|^2/4s}b(z)m_p(dz)ds \nonumber
\\
&\stackrel{\eqref{eq:3.61}}{\le} \int_{0}^t \frac{1}{(t-s)^{1/2}}e^{-\alpha |x|_\rho^2/t}\int_{|z|_\rho<4}\frac{1}{s^{3/2}}\cdot e^{-\beta|z-y|^2/4t}e^{-\beta |y|^2/4s}e^{-\beta |z-y|^2/4s}b(z)m_p(dz)ds \nonumber
\\
&\stackrel{\eqref{eq:3.62}}{\le} \int_{0}^t \frac{1}{(t-s)^{1/2}}e^{-\alpha |x|_\rho^2/t}\cdot e^{-\alpha(8\eps+4|y|)/t}e^{-\alpha |y|^2/t}\int_{|z|_\rho<4}\frac{1}{s^{3/2}}e^{-\beta |z-y|^2/4s}b(z)m_p(dz)ds \nonumber
\\
&\stackrel{\eqref{104}}{\lesssim} e^{-\alpha (|x|_\rho+|y|_\rho+2\eps)^2/t} \int_{0}^t \frac{\|b_2\|_{p_2} }{s^{3/2-1/q_2}} \frac{1}{(t-s)^{1/2}} \left(\int_{|z|_\rho<4}\left(\frac{1}{s^{1/q_2}}e^{-\beta |z-y|^2/4s}\right)^{q_2} m_p(dz)\right)^{\frac{1}{q_2}} ds\nonumber
\\
&\lesssim \|b_2\|_{p_2}\frac{t^{1/q_2}}{t}e^{-\alpha(|x|+|y|)^2/t}\le \|b_2\|_{p_2}\frac{t^{1/q_2}}{t}e^{-\alpha |x-y|^2/t}=\|b_2\|_{p_2}t^{1/q_2}p^0_{3, \alpha }(t,x,y), \label{931}
\end{align}
where the second last $``\lesssim "$ follows from 
\begin{equation*}
\int_{|z|_\rho<4}\left(\frac{1}{s^{1/q_2}}e^{-\beta |z-y|^2/4s}\right)^{q_2} m_p(dz)\le \int_{\IR^2}\frac{1}{s}e^{-\beta q_2 |z-y|^2/4s} dz =\int_{\IR^2}\frac{1}{s}e^{-\beta q_2 |z|^2/4s} dz=\frac{4 \pi }{\beta q_2}.
\end{equation*}
In view of \eqref{e:3.58}, adding up  \eqref{927} and \eqref{931}  yields for Case 4 that
\begin{align}
&\quad \int_{0}^t \int_{|z|_\rho<4}p^0_{3,\alpha} (t-s, x,z)b(z)p^0_{3, \beta}(s, z,y)m_p(dz)ds \lesssim \|b_2\|_{p_2}\left(t^{1/q_2-1/2}+t^{1/q_2}\right)p^0_{3,\alpha}(t,x,y),\label{952}
\end{align}
where $1/q_2>1/2$ because $p_2>2$. Finally, combining Cases 1-4 shows  for $0<\alpha<\beta/4$ that
\begin{align*}
\int_{0}^t \int_{E}p^0_{3, \alpha} (t-s, x,z)b(z)&p^0_{3,\beta}(s, z,y)m_p(dz)ds 
\\
&\lesssim \left(\|b_1\|_{p_1}+\|b_2\|_{p_2}\right)\left(t^{1/(2q_1)+1/2}+t^{1/q_2-1/2}\right)p^0_{3,\alpha}(t,x,y).
\end{align*}
The second conclusion of this proposition now follows from Proposition \ref{P:3.1}
 and Proposition \ref{P:3.2}.
 \end{proof}
For the next proposition, we use another version of the  canonical forms for $X^0$ as follows:
\begin{equation}\label{e:3.61}
p^0_{4,\alpha}(t,x,y):=\left\{
  \begin{aligned}  
   & \frac{1}{\sqrt{t}}e^{-\alpha |x-y|^2/t},\quad
x,y\in \IR_+;& 
\\
& \frac{1}{\sqrt{t}}e^{-4\alpha (|x|^2+|y|_\rho^2)/t}, \quad
x\in \IR_+,y\in D_\eps \cup \{a^*\};& 
\\ 
& \frac{1}{\sqrt{t}}e^{-4\alpha (|x|_\rho^2+|y|^2)/t}, \quad
x\in D_\eps \cup \{a^*\}, y\in \IR_+;& 
\\ 
 & \frac{1}{t}e^{-2\alpha |x-y|^2/t}, \quad  x,y\in D_\eps\cup \{a^*\}, \quad |y|_\rho\ge  M;& 
\\ 
 & \frac{1}{t}e^{-\alpha (|x|-|y|)^2/t},\quad   x,y\in D_\eps\cup \{a^*\}, \quad  |x|_\rho\ge 2M,|y|_\rho< M;& 
 \\ 
&   \frac{1}{\sqrt{t}}e^{-\alpha(|x|^2_\rho+|y|_\rho^2)/t}+\frac{1}{t}\left(1\wedge \frac{|x|_\rho}{\sqrt{t}}\right)\left(1\wedge \frac{|y|_\rho}{\sqrt{t}}\right)e^{-\alpha |x-y|^2/t}, \quad  \text{otherwise}. & 
  \end{aligned}
\right.
\end{equation}
Note that when $\eps<1/4$,  $|x|_\rho\ge 2M$ and $|y|_\rho<M$, it holds $|x|-|y|\le |x-y|\le 4( |x|-|y|)$, i.e., $|x-y|\asymp |x|-|y|$. The following lemma follows immediately. 
\begin{lem}
There exists a  constant $C_{19}, C_{20}>0$, $\alpha_{12}, \alpha_{13}>0$ such that it holds
\begin{equation*}
C_{19}\,p^0_{4,\alpha_{12}}(t,x,y)\le p(t,x,y)\le C_{20}\,p^0_{4,\alpha_{13}}(t,x,y), \quad t\in (0, T],\,(x,y)\in  E\times E.
\end{equation*}
\end{lem}
\begin{prop}\label{152}
 Let $M\ge 16$ be fixed  as in \eqref{choice-of-M}.   There exists  $\alpha_{14}>0$ such that for every $
0<\alpha<\alpha_{14}$, there exists some function $C_{21}(t)>0$ \emph{(}possibly depending on $\alpha$\emph{)} such that
\begin{align}
\int_0^t \int_{z\in E}p^0_{4,\alpha}(t-s,x,z)|b(z)| \left|\nabla_z p(s, z, y )\right| & m_p(dz)ds\le C_{21}(t)p^0_{4,\alpha}(t,x,y), \, 0<s<t\le T, \nonumber
\\
&x,y\in D_\eps \cup \{a^*\},\, |x|_\rho\ge2M,|y|_\rho<M, \label{eq:3.56}
\end{align}
where $C_{21}(t)$ is non-decreasing in $t$, $C_{21}(t)\rightarrow 0$ as $t \rightarrow 0$.
\end{prop}
\begin{proof}  We establish \eqref{eq:3.56} via dividing the computation into three cases based on the position of $z$. Without loss of generality, we assume $b\ge 0$.     Recall in \eqref{def-beta3}, we have defined $\beta_3=(\beta_1\wedge \beta_2)/16 $, where $\beta_1, \beta_2$ are chosen in Proposition \ref{P:3.1}
 and  Proposition \ref{P:3.2} respectively.
 \\
{\it Case 1.} $z\in \IR_+$.   Since $|x|_\rho>2M$, similar to \eqref{443}, again it holds
\begin{equation}\label{115}
4(|x|_\rho^2+|y|_\rho^2)\ge |x|_\rho^2+|y|_\rho^2+2|x|_\rho|y|_\rho+4\eps (|x|_\rho+|y|_\rho)=(|x|_\rho+|y|_\rho+2\eps)^2.
\end{equation}
Therefore, following the same lines as proving \eqref{e:3.48} in Proposition \ref{1100}, one can show that  for any $0<\alpha<\beta/4$, where $0<\beta<\beta_3$.
\begin{align*}
\quad \int_0^t \frac{1}{\sqrt{s}}\int_{\IR_+}  p^0_{4, \alpha}(t-s, x,z)b(z)p^0_{4,\beta}(s, z,y)m_p(dz)ds & \lesssim \|b_1\|_{p_1} e^{-\alpha |x-y|^2/t}\cdot t^{1/(2q_1)-1/2}
\\
\le \|b_1\|_{p_1} e^{-\alpha (|x|-|y|)^2/t}\cdot t^{1/(2q_1)-1/2}&=\|b_1\|_{p_1} t^{1/(2q_1)+1/2}p^0_{4, \alpha}(t,x,y).
\end{align*}
It has been claimed in  Proposition \ref{P:3.1} that for any $0<\beta<\beta_3$,
\begin{equation*}
|\nabla _z p(s,z,y)|\lesssim \frac{1}{\sqrt{s}}p^0_{4, \beta}(s,z,y), \quad z\in \IR_+, y\in D_\eps\cup \{a^*\}.
\end{equation*}
Therefore, for every $0<\alpha<\beta_3/4$,  it holds
\begin{align*}
\int_0^t \int_{\IR_+}  p^0_{4, \alpha}(t-s, x,z)b(z)|\nabla_z p(s,z,y)|m_p(dz)ds &\lesssim c_{1}(t)p^0_{4, \alpha}(t,x,y), 
\end{align*}
where $c_{1}(t)>0$ is non-decreasing in $t$ with $c_{1}(t)\rightarrow 0$ as $t \rightarrow 0$.
\\
{\it Case 2.} $z\in D_\eps, \,  |z|_\rho\ge 2M$, where $M=48$.  We first note that given $0<\eps<1/4$, it holds 
\begin{equation*}
\Big||y|-r\Big|>\frac{r}{4}, \quad \text{for all }r>2M+\eps,\, |y|<M+\eps.
\end{equation*}
 Thus for every $\alpha>0$,  there exists some constant $c_1>0$  such that 
\begin{equation}\label{eq:3.71}
\frac{r}{\sqrt{s}}e^{-2\alpha q_2(|y|-r)^2/s}\le \frac{r}{\sqrt{s}}e^{-2\alpha q_2r^2/(16s)}\le c_1, \quad \text{for all }r>2M+\eps, |y|<M+\eps, s<T.
\end{equation}
For $1/p_2+1/q_2=1$,  by H\"{o}lder's inequality,   it holds  for every $0<\alpha<\beta/4$, where $0<\beta<\beta_3$ that
\begin{eqnarray}
&&\int_0^t \frac{1}{\sqrt{s}}\int_{z\in D_\eps, |z|_\rho \ge 2M}  p^0_{4, \alpha}(t-s, x,z)b(z)p^0_{4,\beta}(s, z,y)m_p(dz)ds\nonumber
\\
&= &\int_{0}^t\frac{1}{\sqrt{s}} \int_{|z|_\rho \ge  2M}\frac{1}{t-s}e^{-2\alpha |x-z|^2/(t-s)}\frac{1}{s}e^{-\beta (|y|-|z|)^2/s}b(z)m_p(dz)ds \nonumber
\\
&\le &\int_{0}^t\frac{1}{\sqrt{s}} \int_{|z|_\rho \ge 2M}\frac{1}{t-s}e^{-2\alpha |x-z|^2/(t-s)}\frac{1}{s}e^{-4\alpha(|y|-|z|)^2/s}b(z)m_p(dz)ds \nonumber
\\
&\lesssim &\int_{0}^t \frac{1}{(t-s)^{1-1/(2q_2)}}\frac{\|b_2\|_{p_2} }{s^{3/2-1/q_2}}\nonumber
\\
&\times & \left(\int_{|z|_\rho\ge 2M}\frac{1}{\sqrt{t-s}}e^{-2\alpha q_2(|x|-|z|)^2/(t-s)}\frac{1}{s}e^{-4\alpha q_2(|y|-|z|)^2/s}m_p(dz)\right)^{1/q_2}ds \nonumber
\\
&\stackrel{r=|z|}{=} &\int_{0}^t \frac{1}{(t-s)^{1-1/(2q_2)}}\frac{\|b_2\|_{p_2} }{s^{3/2-1/q_2}}\left(\int_{r\ge 2M+\epsilon}\frac{r}{\sqrt{t-s}}e^{-2\alpha q_2(|x|-r)^2/(t-s)}\frac{1}{s}e^{-4\alpha q_2(|y|-r)^2/s}dr\right)^{\frac{1}{q_2}} ds \nonumber
\\
&\stackrel{\eqref{eq:3.71}}{\lesssim} &\int_{0}^t \frac{1}{(t-s)^{1-1/(2q_2)}}\frac{\|b_2\|_{p_2} }{s^{3/2-1/q_2}} \left(\int_{r\ge 2M+\epsilon}\frac{1}{\sqrt{t-s}}e^{-2\alpha q_2(|x|-r)^2/(t-s)}\frac{1}{\sqrt{s}}e^{-2\alpha q_2(|y|-r)^2/s}dr\right)^{\frac{1}{q_2}} ds \nonumber
\\
&\lesssim  & \frac{\|b_2\|_{q_2}}{t^{3/2-3/(2q_2)}}\cdot \frac{1}{t^{1/(2q_2)}} e^{-\alpha(|x|-|y|)^2/t} \nonumber = \frac{\|b_2\|_{q_2}}{t^{3/2-1/q_2}}e^{-\alpha(|x|-|y|)^2/t}
\\
&=& \|b_2\|_{p_2}t^{1/q_2-1/2}p^0_{3, \alpha}(t,x,y). \label{e:3.69}
\end{eqnarray}
where the  second  last  $``\lesssim "$ in \eqref{e:3.69}  follows from
\begin{align*}
&\quad \int_{r\ge 2M+\epsilon}\frac{1}{\sqrt{t-s}}e^{-2\alpha q_2(|x|-r)^2/(t-s)}\frac{1}{\sqrt{s}}e^{-2\alpha q_2(|y|-r)^2/s}dr \lesssim \frac{1}{\sqrt{t}}e^{-2\alpha q_2 (|x|-|y|)^2/t}, 
\end{align*}
by the computation used for \eqref{608}.
This establishes the following inequality for $0<\alpha<\beta/4$, where $0<\beta<\beta_3$:
\begin{equation*}
\int_0^t \frac{1}{\sqrt{s}}\int_{z\in D_\eps, |z|_\rho\ge 2M}  p^0_{4, \alpha}(t-s, x,z)b(z)p^0_{4,\beta}(s, z,y)m_p(dz)ds \le \|b_2\|_{q_2} t^{1/q_2-1/2} p^0_{4, \alpha}(t,x,y).
\end{equation*}
 It has been claimed in  Proposition \ref{P:3.2} that for $0<\beta<\beta_3$:
\begin{equation*}
|\nabla _z p(s,z,y)|\lesssim \frac{1}{\sqrt{s}}p^0_{4, \beta}(s,z,y), \quad \text{when  }z, y\in D_\eps\cup \{a^*\}, \, \max\{|z|_\rho, |y|_\rho\}\ge 1. 
\end{equation*}
Therefore, for every $0<\alpha<\beta_3/4$, it holds
\begin{align*}
\int_0^t \int_{z\in D_\eps, |z|_\rho\ge 2M}  p^0_{4, \alpha}(t-s, x,z)b(z)|\nabla_z p(s,z,y)|m_p(dz)ds &\lesssim c_{2}(t)p^0_{4, \alpha}(t,x,y), 
\end{align*}
where $c_{2}(t)>0$ is non-decreasing in $t$ with $c_{1}(t)\rightarrow 0$ as $t \rightarrow 0$.
\\
{\it Case 3. }$z\in D_\eps\cup \{a^*\}, \, |z|_\rho<2M$. It has been shown  that for $0<\beta< \beta_3$ in Proposition \ref{P:3.2}, whether $|y|_\rho, |z|_\rho$ are greater than $1$ or not, it always holds
\begin{equation}\label{1244}
|\nabla_z p(s,z,y)|\lesssim \frac{1}{s}e^{-\beta(|z|_\rho^2+|y|_\rho^2)/s}+\frac{1}{s^{3/2}}e^{-\beta |z-y|^2/s}.
\end{equation}
In addition, we observe that regardless of whether $|z|_\rho\ge M$ or $|z|_\rho <M$, it always holds for this case that 
\begin{equation*}
p^0_{4,\alpha}(t-s, x,z)\lesssim \frac{1}{t-s}e^{-\alpha (|x|-|z|)^2/(t-s)}.
\end{equation*}
Therefore, for $0<\alpha<\beta/4$ where $\beta<\beta_3$,
\begin{eqnarray}
&& \int_0^t \int_{z\in D_\eps \cup \{a^*\}, |z|_\rho<2M}  p^0_{4, \alpha}(t-s, x,z)b(z)|\nabla_z p(s, z,y)|m_p(dz)ds  \nonumber
\\
&\lesssim & \int_{0}^t\int_{|z|_\rho<2M}\frac{1}{t-s}b(z)e^{-\alpha(|x|-|z|)^2/(t-s)}\left(\frac{1}{s}e^{-\beta(|z|_\rho^2+|y|_\rho^2)/s}+\frac{1}{s^{3/2}}e^{-\beta |z-y|^2/s}\right)m_p(dz)ds  \nonumber
\\
&\le & \int_{0}^t\int_{|z|_\rho<2M}\frac{1}{t-s}b(z)e^{-\alpha(|x|-|z|)^2/(t-s)}\frac{1}{s}e^{-\beta(|z|_\rho^2+|y|_\rho^2)/s}m_p(dz)ds \nonumber
\\
&+& \int_{0}^t\int_{|z|_\rho<2M}\frac{1}{t-s}b(z)e^{-\alpha(|x|-|z|)^2/(t-s)}\frac{1}{s^{3/2}}e^{-\beta |z-y|^2/s}m_p(dz)ds={\rm (I)}+{\rm(II)}. \label{e:3.64}
\end{eqnarray}
It holds for (I) in  \eqref{e:3.64}  that, for  $1/p_2+1/q_2=1$, when $0<\alpha<\beta/4<\beta_3/4$,
\begin{align}
{\rm(I)}&= \int_{0}^t\int_{|z|_\rho<2M}\frac{1}{t-s}b(z)e^{-\alpha(|x|-|z|)^2/(t-s)}\frac{1}{s}e^{-\beta(|z|_\rho^2+|y|_\rho^2)/s}m_p(dz)ds \nonumber
\\
&\lesssim e^{-\beta |y|^2_\rho/t}\int_{0}^t \frac{\|b_2\|_{p_2}}{(t-s)^{1-1/(2q_2)}}\frac{1}{s^{1-1/(2q_2)}} \left(\int_{|z|_\rho<2M}\frac{1}{\sqrt{t-s}}e^{-\frac{\alpha q_2(|x|-|z|)^2}{t-s}}\frac{1}{\sqrt{s}}e^{-\frac{\beta q_2|z|^2_\rho}{s}}m_p(dz)\right)^{\frac{1}{q_2}}ds \nonumber
\\ 
&\stackrel{r=|z|_\rho}{\lesssim} e^{-\beta |y|^2_\rho/t}\int_{0}^t \frac{\|b_2\|_{p_2}}{(t-s)^{1-1/(2q_2)}}\frac{1}{s^{1-1/(2q_2)}}\left(\int_{r<2M}\frac{r+\eps}{\sqrt{t-s}}e^{-\frac{\alpha q_2(|x|-r-\eps)^2}{t-s}}\frac{1}{\sqrt{s}}e^{-\frac{\beta q_2 r^2}{s}}dr\right)^{\frac{1}{q_2}}ds \nonumber
\\
&\stackrel{r<2M}{\lesssim} e^{-\beta |y|^2_\rho/t}\int_{0}^t \frac{\|b_2\|_{p_2}}{(t-s)^{1-1/(2q_2)}}\frac{1}{s^{1-1/(2q_2)}}\left(\int_{r<2M}\frac{1}{\sqrt{t-s}}e^{-\frac{\alpha q_2(|x|-r-\eps)^2}{t-s}}\frac{1}{\sqrt{s}}e^{-\frac{\beta q_2 r^2}{s}}dr\right)^{\frac{1}{q_2}} ds\nonumber
\\
&\lesssim  e^{-\beta |y|^2_\rho/t} \frac{1}{t^{1-1/q_2}}\cdot \frac{1}{t^{1/(2q_2)}}e^{-\alpha (|x|-\eps)^2/t}\|b_2\|_{q_2}\lesssim \frac{\|b_2\|_{p_2}}{t^{1-1/(2q_2)}}e^{-\alpha (|x|_\rho^2+|y|_\rho^2)/t}  \nonumber
\\
&\lesssim \frac{\|b_2\|_{p_2}}{t^{1-1/(2q_2)}}e^{-\alpha (|x|_\rho-|y|_\rho)^2/t}=\frac{\|b_2\|_{p_2}}{t^{1-1/(2q_2)}}e^{-\alpha (|x|-|y|)^2/t}=\|b_2\|_{p_2}t^{1/(2q_2)}p^0_{4, \alpha}(t,x,y),
\label{1243}
\end{align}
For (II) in  \eqref{e:3.64}, we select $q'_2, r_2$: $1/p_2+1/q'_2+1/r_2=1$ with  $q'_2$ being sufficiently large such  that $1/(2q'_2)+1/r_2>1/2$. This is possible because $p_2>2$. Therefore, for $\alpha, \beta$ such that $0<\alpha<\beta/4<\beta_3/4$, it holds 
\begin{align}
{\rm (II)}&= \int_{0}^t\frac{1}{\sqrt{s}} \int_{|z|_\rho <2 M}\frac{1}{t-s}e^{-\alpha (|x|-|z|)^2/(t-s)}\frac{1}{s}e^{-\beta |y-z|^2/s}b(z)m_p(dz)ds \nonumber
\\
&\le \int_0^t \int_{|z|_\rho<2M} \frac{1}{t-s}e^{-\alpha (|x|-|y|)^2/(t-s)}\frac{b(z)}{s^{3/2}}e^{-\alpha (|y|-|z|)^2/s}e^{-(\beta-\alpha )|y-z|^2/s}\,m_p(dz)ds\nonumber
\\
&\lesssim \int_0^t \frac{1}{s^{3/2-1/(2q'_2)-1/r_2}}\frac{\|b_2\|_{p_2} }{(t-s)^{1-1/(2q'_2)}} \nonumber
\\
&\qquad \times \left(\int_{|z|_\rho<2M}\frac{1}{\sqrt{t-s}}e^{-\alpha q'_2(|x|-|z|)^2/(t-s)}\frac{1}{\sqrt{s}}e^{-\alpha q'_2(|y|-|z|)^2/s}m_p(dz)\right)^{1/q'_2}\nonumber
\\
&\qquad \times \left(\int_{|z|_\rho<2M}\frac{1}{s}e^{-(\beta-\alpha)r_2|y-z|^2/s}m_p(dz)\right)^{1/r_2}ds \nonumber
\\
& \lesssim \int_0^t \frac{\|b_2\|_{p_2}}{s^{3/2-1/(2q'_2)-1/r_2}}\frac{1}{(t-s)^{1-1/(2q'_2)}}\nonumber
\\
&\qquad \times\left(\int_{r<2M+\eps}\frac{r+\eps}{\sqrt{t-s}}e^{-\alpha q'_2(|x|-r)^2/(t-s)}\frac{1}{\sqrt{s}}e^{-\alpha q'_2(|y|-r)^2/s}\,dr\right)^{1/q'_2}ds \nonumber
\\
& \lesssim \int_0^t \frac{\|b_2\|_{p_2}}{s^{3/2-1/(2q'_2)-1/r_2}}\frac{1}{(t-s)^{1-1/(2q'_2)}} \nonumber
\\
&\qquad \times \left(\int_{r<2M+\eps}\frac{1}{\sqrt{t-s}}e^{-\alpha q'_2(|x|-r)^2/(t-s)}\frac{1}{\sqrt{s}}e^{-\alpha q'_2(|y|-r)^2/s}\,dr\right)^{1/q'_2}ds \nonumber
\\
&\lesssim \frac{1}{t^{3/2-1/q'_2-1/r_2}}\cdot \frac{1}{t^{1/(2q'_2)}}e^{-\alpha (|x|-|y|)^2/t}\cdot \|b_2\|_{p_2} \nonumber
\\
&=\frac{1}{t^{3/2-1/(2q'_2)-1/r_2}}e^{-\alpha (|x|-|y|)^2/t}\|b_2\|_{p_2}=\|b_2\|_{p_2}t^{1/(2q_2')+1/r_2-1/2}\cdot p^0_{4, \alpha}(t,x,y), \label{1251}
\end{align}
given our choice of $q'_2, r_2$: $1/(2q'_2)+1/r_2>1/2$. In view of  \eqref{e:3.64}, adding up  \eqref{1251}  and  \eqref{1243} proves  for Case 3 that for  $0<\alpha<\beta_3/4$:
\begin{align*}
 \int_0^t \int_{|z|_\rho<2M}  p^0_{4, \alpha}(t-s, x,z)b(z)& |\nabla _zp(s, z,y)|m_p(dz)ds 
\\ 
& \lesssim \|b_2\|_{p_2} \left(t^{1/(2q_2)}+t^{1/(2q'_2)+1/r_2-1/2}\right)p^0_{4, \alpha}(t,x,y).
\end{align*}
Now combining Cases 1-3 yields when $0<\alpha<\beta_3/4$, 
\begin{align*}
 \int_0^t \int_{E} &  p^0_{4, \alpha}(t-s, x,z)b(z)|\nabla_z p(s, z,y)|m_p(dz)ds
 \\
&\lesssim \left(\|b_1\|_{p_1}+\|b_2\|_{p_2}\right) \left(t^{1/(2q_1)+1/2}+t^{1/q_2-1/2}+t^{1/(2q_2)}+t^{1/(2q'_2)+1/r_2-1/2}\right)p^0_{4, \alpha}(t,x,y).
\end{align*}
\end{proof}
For the next  and the last case which is $x,y\in D_\eps\cup \{a^*\}$ with $|x|_\rho<2M$ and $|y|_\rho<M$, we set the canonical form of the heat kernel of $X^0$ as follows: 
\begin{equation}\label{canonical-form2}
p^0_{5,\alpha}(t,x,y):=\left\{
  \begin{aligned}
   & \frac{1}{\sqrt{t}}e^{-\alpha |x-y|^2/t},\quad 
 x\in \IR_+,y\in \IR_+;&
\\ 
& \frac{1}{\sqrt{t}}e^{-\alpha (|x|^2+|y|_\rho^2)/t},\quad
x\in \IR_+,y\in D_\eps \cup \{a^*\} ;& 
\\ 
& \frac{1}{\sqrt{t}}e^{-\alpha (|x|_\rho^2+|y|^2)/t},\quad 
 x\in D_\eps \cup \{a^*\}, y\in \IR_+;& 
\\
 & \frac{1}{t}e^{-2\alpha |x-y|^2/t}, \quad  x,y\in D_\eps\cup \{a^*\} \text{ with } \max\{|x|_\rho,|y|_\rho\}> 4M;& 
\\ 
&   \frac{1}{\sqrt{t}}e^{-\alpha(|x|^2_\rho+|y|_\rho^2)/t}+\frac{1}{t}\left(1\wedge \frac{|x|_\rho}{\sqrt{t}}\right)\left(1\wedge \frac{|y|_\rho}{\sqrt{t}}\right)e^{-2\alpha |x-y|^2/t}, \text{ otherwise}. & 
  \end{aligned}
\right.
\end{equation}
The following theorem is immediate.
\begin{lem}
There exist  constants $C_{22}, C_{23}>0$ and $\alpha_{15}, \alpha_{16}>0$ such that it holds
\begin{equation*}
C_{22}\;p^0_{5,\alpha_{15}}(t,x,y)\le p(t,x,y)\le C_{23}\;p^0_{5,\alpha_{16}}(t,x,y), \qquad (t,x,y)\in (0,T]\times E\times E.
\end{equation*}
\end{lem}

\begin{prop}\label{153}
 Let $M\ge 16$ be  fixed  as in \eqref{choice-of-M}. There exists some $\alpha_{17}>0$, such that for every  $0<\alpha<\alpha_{17}$, there exists a function $C_{24}(t)>0$ \emph{(}possibly depending on $\alpha$\emph{)} which is non-decreasing  in $t$, satisfying $C_{24}(t)\rightarrow 0$ as $t \rightarrow 0$, such that
\begin{align}
\int_0^t \int_{z\in E}p^0_{5,\alpha}(t-s,x,z)|b(z)|\left|\nabla_z p(s, z, y )\right| & m_p(dz)ds\le C_{24}(t)p^0_{5,\alpha}(t,x,y), \, 0<s<t\le T, \nonumber
\\
&x,y\in D_\eps \cup \{a^*\},\,|x|_\rho<2M,|y|_\rho<M. \label{1214}
\end{align}
\end{prop}
\begin{proof} In this proof, we let  $\alpha, \beta$ satisfy $0<\alpha<\beta/2$   with $0<\beta<\beta_3$.  We also assume without loss of generality that $b\ge 0$. 
\\
{\it Case 1.} $z\in \IR_+$.  Since $0<\beta<\beta_3$,  we have for $0<\alpha<\beta/2$ that
\begin{eqnarray}
&& \int_{0}^t \frac{1}{\sqrt{s}}\int_{\IR_+}p^0_{5, \alpha} (t-s, x,z)b(z)p^0_{5, \beta}(s, z,y)m_p(dz)ds \nonumber
\\
&=& \int_0^t \frac{1}{\sqrt{s}}\int_{\IR_+}\frac{1}{\sqrt{t-s}}e^{-\alpha (|x|_\rho^2+|z|^2)/(t-s)}b(z)\frac{1}{\sqrt{s}}e^{-\beta(|z|^2+|y|_\rho^2)/s}m_p(dz)ds \nonumber
\\
&\le & e^{-\alpha (|x|_\rho^2+|y|^2_\rho)/t} \int_0^t s^{-1+1/(2q_1)}(t-s)^{-1/2}\int_{\IR_+}b(z)s^{-1/(2q_1)}e^{-\beta|z|^2/s}m_p(dz)ds \nonumber
\\
&\le &  e^{-\alpha (|x|_\rho^2+|y|^2_\rho)/t}\int_0^t s^{-1+1/(2q_1)}(t-s)^{-1/2} \|b_{1}\|_{p_1} \left(\int_{\IR_+}\left(s^{-1/(2q_1)}e^{-\beta|z|^2/s}\right)^{q_1} m_p(dz)\right)^{1/q_1}ds \nonumber
\\
&\lesssim & \|b_1\|_{p_1} e^{-\alpha (|x|_\rho^2+|y|^2_\rho)/t}\cdot \frac{t^{1/(2q_1)}}{\sqrt{t}}\le \|b_1\|_{p_1}t^{1/(2q_1)}p^0_{5, \alpha}(t,x,y).\label{210}
\end{eqnarray}
In view of Proposition \ref{P:3.1}, \eqref{210}  implies that for $0<\alpha<\beta_3/2$, it holds
\begin{equation*}
\int_0^t \int_{z\in \IR_+}p^0_{5,\alpha}(t-s,x,z)|b(z)|\left|\nabla_z p(s, z, y )\right| m_p(dz)ds \lesssim \|b_1\|_{p_1} t^{1/2q_1}\cdot p^0_{5, \alpha}(t,x,y).
\end{equation*}
\\
{\it Case 2.} $z\in D_\eps$, $|z|_\rho\ge 4M$.   Note that $|x|_\rho<2M$ and $|y|_\rho<M$, it thus holds  
\begin{equation}\label{e:3.79}
|x-z|^2\ge |x|_\rho^2/2+|x-z|^2/2, \text{ and } |y-z|^2\ge |y|^2_\rho/2+|y-z|^2/2.
\end{equation}
By Proposition \ref{P:3.2},   it holds for $0<\alpha<\beta/2<\beta_3/2$ that
\begin{eqnarray}
&&\int_0^t \int_{z\in D_\eps, |z|_\rho\ge 4M}p^0_{5,\alpha}(t-s,x,z)|b(z)|\left|\nabla_z p(s, z, y )\right| m_p(dz)ds  \nonumber
\\
&\lesssim &\int_{0}^t \int_{|z|_\rho\ge 4M}\frac{1}{t-s}e^{-2\alpha |x-z|^2/(t-s)}\frac{1}{s^{3/2}}e^{-2\beta |y-z|^2/s}b(z)m_p(dz)ds \nonumber
\\
&\stackrel{\eqref{e:3.79}}{\le} & e^{-\alpha(|x|^2_\rho+|y|^2_\rho)/t}\int_{0}^t \int_{|z|_\rho\ge 4M} \frac{1}{t-s}e^{-\alpha|x-z|^2/(t-s)}\frac{1}{s^{3/2}}e^{-\beta|y-z|^2/s}b(z)m_p(dz)ds \nonumber
\\
&\stackrel{|y-z|\ge 3M}{\lesssim} &e^{-\alpha(|x|^2_\rho+|y|^2_\rho)/t}\int_{0}^t  \int_{|z|_\rho\ge 4M} \frac{1}{t-s}e^{-\alpha|x-z|^2/(t-s)}\frac{1}{s}e^{-\alpha |y-z|^2/s}b(z)m_p(dz)ds \nonumber
\\
&\lesssim &\|b_2\|_{p_2}e^{-\alpha(|x|^2_\rho+|y|^2_\rho)/t} \int_0^t \frac{1}{s^{1-1/q_2}}\frac{1}{(t-s)^{1-1/q_2}} \nonumber
\\
&&  \left(\int_{|z|_\rho\ge 4M}\left(\frac{1}{(t-s)^{1/q_2}}e^{-\alpha|x-z|^2/(t-s)}\frac{1}{s^{1/q_2}}e^{-\alpha|y-z|^2/s}\right)^{q_2}m_p(dz)\right)^{1/q_2}ds \nonumber
\\
&\lesssim &\|b_2\|_{p_2}e^{-\alpha(|x|^2_\rho+|y|^2_\rho)/t} \int_0^t \frac{1}{s^{1-1/q_2}}\frac{1}{(t-s)^{1-1/q_2}} \nonumber
\\
&& \left(\int_{|z|_\rho\ge 4M}\frac{1}{t-s}e^{-\alpha q_2|x-z|^2/(t-s)}\frac{1}{s}e^{-\alpha q_2|y-z|^2/s}m_p(dz)\right)^{1/q_2}ds \nonumber
\\
&\lesssim  &\|b_2\|_{p_2}e^{-\alpha(|x|^2_\rho+|y|^2_\rho)/t}  \frac{t^{1/q_2}}{t^{1-1/q_2}}\cdot \frac{1}{t^{1/q_2}} \le \|b_2\|_{p_2}t^{1/2-1/p_2}p^0_{5, \alpha}(t,x,y),\label{338}
\end{eqnarray}
where the last $``\lesssim "$ follows from the computation similar to that for  \eqref{semigroup-2-d-BM}.
\\
{\it Case 3.} $z\in D_\eps\cup \{a^*\}$, $|z|_\rho<4M$. From the proof to  Proposition \ref{P:3.2},  we know for  $0<\beta<\beta_3$, whether $|y|_\rho, |z|_\rho$ are greater than $1$ or not, it always holds
\begin{equation*}
|\nabla_z p(s,z,y)|\lesssim e^{-\beta(|z|_\rho^2+|y|_\rho^2)/s}+\frac{1}{s^{3/2}}\left(1\wedge \frac{|y|_\rho}{\sqrt{s}}\right)e^{-\beta |z-y|^2/s}.
\end{equation*}
It therefore holds
\begin{eqnarray}
&&\int_0^t\int_{z\in D_\eps \cup \{a^*\}, |z|_\rho<4M}p^0_{5,\alpha}
(t-s, x,z)b(z)|\nabla_z p(s,z,y)|m_p(dz)ds\nonumber
\\
&\lesssim &\int_0^t\int_{|z|_\rho<4M}\left[\frac{1}{\sqrt{t-s}}e^{-\alpha(|x|_\rho^2+|z|^2_\rho)/(t-s)}+\frac{1}{t-s}\left(1\wedge \frac{|x|_\rho}{\sqrt{t-s}}\right)\left(1\wedge \frac{|z|_\rho}{\sqrt{t-s}}\right)e^{-2\alpha|x-z|^2/(t-s)}\right]\nonumber
\\
&\times &  b(z)\left[\frac{1}{s}e^{-\beta(|y|_\rho^2+|z|_\rho^2)/s}+\frac{1}{s^{3/2}}\left(1\wedge \frac{|y|_\rho}{\sqrt{s}}\right)e^{-\beta |y-z|^2/s}\right]m_p(dz)ds \nonumber
\\
 &= & {\rm (I)+(II)+(III)+(IV)}.\label{eq:3.66}
\end{eqnarray}
The right hand side of \eqref{eq:3.66} can be expanded into the sum of four terms. We now bound these four terms on right hand side of \eqref{eq:3.66} from above term by term. In the following it is always assumed $0<\alpha<\beta/2$ and $0<\beta<\beta_3$.  First of all,  for $1/p_2+1/q_2=1$,  we have by H\"{o}lder's inequality,
\begin{align}
{\rm (I)}&= \int_0^t \frac{1}{\sqrt{t-s}}e^{-\alpha (|x|_\rho^2+|z|_\rho^2)/(t-s)}\int_{z\in D_\eps\cup\{a^*\}, |z|_\rho<4M}\frac{b(z)}{s}e^{-\beta (|z|_\rho^2+|y|_\rho^2)/s}m_p(dz)ds \nonumber 
\\
&\le e^{-\alpha(|x|_\rho^2+|y|_\rho^2)/t}\int_0^t \frac{1}{(t-s)^{1/2}}\frac{1}{s^{1-1/(2q_2)}}\int_{ |z|_\rho<4M}\frac{b(z)}{s^{1/(2q_2)}}e^{-\beta |z|^2_\rho/s}m_p(dz)ds \nonumber
\\
&\lesssim e^{-\alpha(|x|_\rho^2+|y|_\rho^2)/t}\int_0^t \frac{1}{(t-s)^{1/2}}\frac{ \|b_2\|_{p_2} }{s^{1-1/(2q_2)}}\left(\int_{|z|_\rho<4M}\frac{1}{\sqrt{s}}e^{-\beta q_2 |z|^2_\rho/s}m_p(dz)\right)^{1/q_2}ds\nonumber
\\
&\stackrel{r=|z|_\rho}{=} e^{-\alpha(|x|_\rho^2+|y|_\rho^2)/t}\int_0^t \frac{1}{(t-s)^{1/2}}\frac{\|b_2\|_{p_2}}{s^{1-1/(2q_2)}}\left(\int_{r<4M}\frac{r+\eps}{\sqrt{s}}e^{-\beta q_2 r^2/s}\;  dr \right)^{1/q_2}ds  \nonumber
\\
&\lesssim \|b_2\|_{p_2}\frac{t^{1/(2q_2)}}{\sqrt{t}}e^{-\alpha(|x|_\rho^2+|y|_\rho^2)/t}\le \|b_2\|_{p_2} t^{1/(2q_2)}p^0_{5, \alpha}(t,x,y),\label{124}
\end{align}
where the last $``\lesssim"$ follows from 
\begin{equation*}
\int_{r<4M}\frac{r+\eps}{\sqrt{s}}e^{-\beta q_2 r^2_\rho/s}\;  dr \le (4M+\eps) \int_{r<4M}\frac{1}{\sqrt{s}}e^{-\beta q_2 r^2/s}dr\lesssim \int_0^\infty \frac{1}{\sqrt{s}}e^{-\beta q_2 r^2/s}dr=\frac{1}{2}\sqrt{\frac{\pi }{\beta q_2}}.
\end{equation*}
Secondly, with the following version of  the triangle inequality 
\begin{equation}\label{triangle-ineq-v3}
|z|_\rho^2/(t-s)+|z-y|^2/s\ge |y|^2_\rho/t,
\end{equation}
we have for $1/p_2+1/q_2=1$ that 
\begin{align}
{\rm (II)}&= \int_{0}^t \int_{z\in D_\eps\cup \{a^*\}, |z|_\rho<4M}\frac{1}{(t-s)^{1/2}}e^{-\alpha (|x|_\rho^2+|z|_\rho^2)/(t-s)} b(z)\frac{1}{s^{3/2}}\left(1\wedge \frac{|y|_\rho}{\sqrt{s}}\right) e^{-\beta |z-y|^2/s} m_p(dz)ds \nonumber
\\
&\le \int_{0}^t \frac{1}{(t-s)^{1/2}}e^{-\alpha (|x|_\rho^2+|z|_\rho^2)/(t-s)}\int_{D_\eps}\frac{1}{s^{3/2}}e^{-\beta |z-y|^2/s}b(z)m_p(dz)ds \nonumber
\\
&= \int_{0}^t \frac{1}{s^{3/2-1/q_2}}\frac{1}{(t-s)^{1/2}}e^{-\alpha (|x|_\rho^2+|z|_\rho^2)/(t-s)}e^{-\alpha |z-y|^2/s}\int_{D_\eps}\frac{1}{s^{1/q_2}}e^{-(\beta-\alpha) |z-y|^2/s}b(z)m_p(dz)ds \nonumber
\\
&\stackrel{\eqref{triangle-ineq-v3}}{\lesssim} e^{-\alpha (|x|_\rho^2+|y|_\rho^2)/t} \int_{0}^t \frac{1}{s^{3/2-1/q_2}} \frac{\|b_2\|_{p_2}}{(t-s)^{1/2}} \left(\int_{D_\eps}\left(\frac{1}{s^{1/q_2}}e^{-(\beta-\alpha)|z-y|^2/s}\right)^{q_2} m_p(dz)\right)^{1/q_2}ds  \nonumber
\\
&\lesssim \frac{1}{t^{1-1/q_2}}e^{-\alpha(|x|_\rho^2+|y|^2_\rho)/t}\|b_2\|_{p_2}= \frac{\|b_2\|_{p_2}}{t^{1/p_2}}e^{-\alpha(|x|_\rho^2+|y|^2_\rho)/t}\le \|b_2\|_{p_2}t^{1/2-1/p_2}p^0_{5, \alpha}(t,x,y), \label{139}
\end{align}
where the last $``\lesssim "$ follows from
\begin{equation*}
\int_{D_\eps}\left(\frac{1}{s^{1/q_2}}e^{-(\beta-\alpha)|z-y|^2/s}\right)^{q_2} m_p(dz)\lesssim \int_{\IR^2}\frac{1}{s}e^{-(\beta-\alpha)q_2|z-y|^2/s} dz=\frac{\pi }{(\beta -\alpha)q_2}.
\end{equation*}
Next,  we select  $q'_2, r_2>0$ satisfying $1/p_2+1/q'_2+1/r_2=1$ and  $1/(2r_2)+1/q'_2>1/2$, which is possible since $p_2>2$.  Again we need the following version of triangle inequality:
\begin{equation}\label{triangle-ineq-v4}
|x-z|^2/(t-s)+z|_\rho^2/s\ge |x|^2_\rho/t.
\end{equation}
It follows  that 
\begin{align}
\nonumber {\rm (III)}&= \int_0^t \int_{ |z|_\rho<4M}\frac{1}{t-s}\left(1\wedge \frac{|x|_\rho}{\sqrt{t-s}}\right)\left(1\wedge \frac{|z|_\rho}{\sqrt{t-s}}\right)e^{-2\alpha|x-z|^2/(t-s)}\frac{b(z)}{s}e^{-\beta(|z|_\rho^2+|y|_\rho^2)/s}m_p(dz)ds
\\
\nonumber &\le \int_0^t \int_{|z|_\rho<4M}\frac{1}{t-s}e^{-2\alpha|x-z|^2/(t-s)}\frac{b(z)}{s}\;e^{-\beta (|z|_\rho^2+|y|^2_\rho)/s}b(z)m_p(dz)ds
\\
\nonumber &= \int_0^t e^{-\alpha|x-z|^2/(t-s)}e^{-\alpha (|z|_\rho^2+|y|_\rho^2)/s} \int_{|z|_\rho<4M}\frac{1}{t-s}e^{-\alpha|x-z|^2/(t-s)}\frac{b(z)}{s}e^{-(\beta-\alpha)(|z|_\rho^2+|y|^2_\rho)/s}m_p(dz)ds
\\
\nonumber &\stackrel{\eqref{triangle-ineq-v4}}{\lesssim}  e^{-\alpha (|x|_\rho^2+|y|_\rho^2)/t} \int_0^t \frac{1}{s^{1-1/(2r_2)}}\frac{\|b_2\|_{p_2}}{(t-s)^{1-1/q'_2}}\left(\int_{|z|_\rho<4M}\frac{1}{t-s}e^{-\alpha q'_2|x-z|^2/(t-s)}m_p(dz)\right)^{1/q'_2}
\\
&\quad \left(\int_{|z|_\rho<4M}\frac{1}{\sqrt{s}}e^{-(\beta-\alpha ) r_2|z|_\rho^2/s}m_p(dz)\right)^{1/r_2}ds  \nonumber
\\
&\lesssim \frac{1}{t^{1-1/q'_2-1/(2r_2)}}e^{-\alpha (|x|_\rho^2+|y|_\rho^2)/t}\cdot \|b_2\|_{p_2}\le \|b_2\|_{p_2}t^{1/(2r_2)+1/q_2'-1/2}p^0_{5, \alpha}(t,x,y), \label{e:3.85}
\end{align}
where for the last $``\lesssim "$ in \eqref{e:3.85}, it is not hard to see that the first integral in $z$  on the left hand side is comparable to a constant, and so is the second integral since
\begin{align*}
\int_{|z|_\rho<4M}\frac{1}{\sqrt{s}}e^{-(\beta-\alpha ) r_2|z|_\rho^2/s}m_p(dz)&\stackrel{r=|z|_\rho}{=}\int_{r\le 4M}\frac{r+\eps}{\sqrt{s}}e^{-(\beta -\alpha)r_2 r^2/s} dr  
\\
&\lesssim  \int_0^\infty \frac{1}{\sqrt{s}}e^{-(\beta -\alpha)r_2 r^2/s}dr
=\frac{1}{2}\sqrt{\frac{\pi}{(\beta -\alpha)r_2}}.
\end{align*}
Note in \eqref{e:3.85}
$1/(2r_2)+1/q'_2>1/2$ by our selection of $q'_2$ and $r_2$. 
Finally, in order to check the last term of \eqref{eq:3.66}, we need to verify the following:
\begin{align}
(IV) &:= \int_0^t\frac{1}{\sqrt{s}}\int_{ |z|_\rho<4M}\frac{1}{t-s}\left(1\wedge \frac{|x|_\rho}{\sqrt{t-s}}\right)\left(1\wedge \frac{|z|_\rho}{\sqrt{t-s}}\right)e^{-2\alpha|x-z|^2/(t-s)} \nonumber
\\
&\quad \cdot   \frac{b(z)}{s}\left(1\wedge \frac{|y|_\rho}{\sqrt{s}}\right)e^{-\beta|y-z|^2/s}m_p(dz)ds\lesssim \frac{1}{t}\left(1\wedge \frac{|x|_\rho}{\sqrt{t}}\right)\left(1\wedge \frac{|y|_\rho}{\sqrt{t}}\right)e^{-2\alpha|x-y|^2/s}.\label{1111437}
\end{align}
Indeed, since $\alpha<\beta /2$, we may take $D=D_\eps\cup \{a^*\}$ and $\nu(dzds)= b_2(z)dzds$  which belongs to parabolic Kato class $\mathbf{K}_2$ on $\IR^2$. Hence  \cite[Lemma 3.6]{CKP}  immediately yields that 
\begin{eqnarray*}
&&\int_0^t\frac{1}{\sqrt{s}}\int_{z\in D_\eps\cup \{a^*\}}\frac{1}{t-s}\left(1\wedge \frac{|x|_\rho}{\sqrt{t-s}}\right)\left(1\wedge \frac{|z|_\rho}{\sqrt{t-s}}\right)e^{-2\alpha|x-z|^2/(t-s)}
\\
&\times & \frac{b(z)}{s}\left(1\wedge \frac{|y|_\rho}{\sqrt{s}}\right)e^{-\beta|y-z|^2/s}m_p(dz)ds\lesssim \frac{1}{t}\left(1\wedge \frac{|x|_\rho}{\sqrt{t}}\right)\left(1\wedge \frac{|y|_\rho}{\sqrt{t}}\right)e^{-2\alpha|x-y|^2/t}\cdot N(t),
\end{eqnarray*}
where, as defined in \cite{CKP},
\begin{align*}
N(t)&=\sup_{\tau >0, x\in \IR^2} \int_{(\tau -t, \tau )\times \IR^2}\frac{ b(y)}{(\tau -s)^{3/2}}e^{-c_1 |x-y|^2/(\tau -s)}dsdy  
\\
&+\sup_{s>0, x\in \IR^2} \int_{(s, s+t)\times \IR^2}\frac{ b(y)}{(\tau -s)^{3/2}}e^{-c_1 |x-y|^2/(\tau -s)} d\tau dy <\infty, \quad \text{for all }t\le T,
\end{align*}
where $c_1>0$ is a small positive constant can be taken as, for instance, $c_1=\alpha /32$ (for more details, interested readers may refer to \cite[p.1114]{CKP}). It also holds that $N(t)\downarrow 0$ as $t\downarrow 0$. This verifies \eqref{1111437}. Adding up (I)-(IV) shows that  for each pair of $\alpha, \beta$ satisfying $0<\alpha<\beta/2<\beta_3/2$, it holds
\begin{eqnarray}
&& \int_0^t\int_{z\in D_\eps \cup \{a^*\}, |z|_\rho<4M}p^0_{5,\alpha}
(t-s, x,z)|b(z)|\nabla_z p(s,z,y)|m_p(dz)ds  \nonumber
\\
&\le & \left[\|b_2\|_{p_2}\left(t^{1/(2q_2)}+t^{1/2-1/p_2}+t^{1/q'_2+1/(2r_2)-1/2}\right)+N(t)\right] p^0_{5,\alpha}(t,x,y). \label{e:3.78}
\end{eqnarray}
Now combining Cases 1-3 yields that for any pair of $\alpha, \beta$ satisfying $0<\alpha<\beta/2$ and $0<\beta<\beta_3$, it holds
\begin{align*}
&\quad \int_0^t\int_Ep^0_{5,\alpha}
(t-s, x,z)|b(z)|\nabla_z p(s,z,y)|m_p(dz)ds \lesssim C_{20}(t)p^0_{5, \alpha}(t,x,y),
\end{align*}
where 
\begin{equation*}
C_{24}(t):=\left(\|b_1\|_{p_1}+\|b_2\|_{p_2}\right)\left(t^{1/(2q_1)}+t^{1/(2q_2)}+t^{1/2-1/p_2}+t^{1/2-1/q'_2-1/2r_2}\right)+N(t).
\end{equation*}
\end{proof}

Now our goal is to establish two-sided heat kernel bounds for  BMVD with drift  $X$. Toward  this purpose, we set $k_0(t,x,y)=p(t,x,y)$ and then inductively define
\begin{equation}\label{def-k_n}
k_n(t,x,y):=\int_0^t \int_{E} k_{n-1}(t-s, x,z)b(z)\cdot \nabla_zp(s,z,y)dzds, \quad \text{for } n\ge 1.
\end{equation}

\begin{lem}\label{L:3.16}
There exists $T_1>0$  such that for all $x,y\in E$ and $n\ge 1$,  $k_n(t,x,y)$ is well-defined on $(0, T_1]$. Furthermore, there exist $C_i>0,25\le i\le 31$, such that $\sum_{n=0}^\infty k_n(t,x,y)$ absolutely converges  on $(0, T_1]\times E\times E$  with the  following upper bound:
\begin{description}
\item{\rm (i)} For $x \in \IR_+$ and $y\in E$ or $y\in \IR_+$ and $x\in E$,
\begin{equation*}
\sum_{n=0}^\infty k_n(t,x,y)\le\frac{C_{25}}{\sqrt{t}}e^{-C_{26}\rho(x,y)^2/t};
\end{equation*}
\item{\rm (ii)} For $x,y\in D_\eps\cup \{a^*\}$ with $\max\{|x|_\rho,|y|_\rho\}<1$,
\begin{align*}\
\sum_{n=0}^\infty k_n(t,x,y)\le\frac{C_{27}}{\sqrt{t}}e^{-C_{28}\rho(x,y)^2/t}+\frac{C_{27}}{t}\left(1\wedge
\frac{|x|_\rho}{\sqrt{t}}\right)\left(1\wedge
\frac{|y|_\rho}{\sqrt{t}}\right)e^{-C_{29}|x-y|^2/t};
\end{align*}
and when $\max\{|x|_\rho,|y|_\rho\}\ge 1$,
\begin{equation*}
\sum_{n=0}^\infty k_n(t,x,y) \le \frac{C_{30}}{t}e^{-C_{31}\rho(x,y)^2/t}.
\end{equation*}
\end{description}
\end{lem}
\begin{proof}
In view of Proposition \ref{121051} to Proposition \ref{153},  we use induction to show the following upper bounds for  $k_n(t,x,y)$, by dividing our argument into several cases depending on the  regions of $x$ and $y$. For example, when  $x\in \IR_+,\, y\in E$. By Propositions \ref{121051} and  \ref{1236}, we can select $t_1>0$ sufficiently small so that  for some $c_1, c_2>0$ and   $0<c_3<1/2$, it holds 
\begin{equation*}
k_0(t,x,y)=p(t,x,y)\le c_1p^0_{0, c_2}(t,x,y), \quad \text{for all }x\in \IR_+, y\in E, t\in (0, t_1]
\end{equation*}
and
\begin{align*}
|k_n(t,x,y)|\le c_1\cdot c_3^{n-1} \int_0^t\int_{E} p^0_{0,c_2}(t-s, x,z)& |b(z)||\nabla_zp(s,z,y)|  m_p(dz)ds \le c_1\cdot c_3^np^0_{0,c_2}(t,x,y), 
\\
&\quad \text{for all }x\in \IR_+,y\in E, \,t\in (0, t_1], \text{ and } n=1,2,\cdots 
\end{align*}
Therefore, 
\begin{align*}
 \sum_{n=0}^\infty k_n(t,x,y)\le \sum_{n=0}^\infty |k_n(t,x,y)|\le \frac{c_1}{1-c_3}p^0_{0,c_2}(t,x,y), \, \text{for }x\in \IR_+, y\in E, t\in (0, t_1].
\end{align*}
Note that when $x$ and $y$ are both in $\IR_+$, $|x-y|=\rho(x,y)$, and when $x\in \IR_+$, $y\in D_\eps \cup \{a^*\}$, $|x|+|y|_\rho= \rho(x,y)$. It thus holds for some $c_4>0$ that
\begin{equation*}
\sum_{n=0}^\infty k_n(t,x,y)\lesssim \frac{1}{\sqrt{t}}e^{-c_4\rho(x,y)^2/t} \quad \text{for }x\in \IR_+, y\in E, t\in (0, t_1].
\end{equation*}

The remaining cases can all be taken care of in the exact manner. Each case can be established on an interval $t\in (0, t_i]$, for some $t_i>0$.  The proof  is thus  complete by taking $T_1=\min_i t_i$, for $i=1,2,\cdots$ (finitely many).
\end{proof}
We now claim the following lemma which combined with Lemma \ref{L:3.16} gives the upper bounds in Theorem \ref{T:smalltime}.
\begin{lem}\label{L:3.17}
Let
\begin{equation}\label{def-density-X}
p^b(t,x,y):=\sum_{n=0}^\infty k_n(t,x,y), \quad \text{for }t>0, x,y\in E.
\end{equation}
Then the family $\{p^b(t,x,y)\}_{t>0, x,y\in E}$  is indeed the transition density function  of $X$.   In particular, for any $T_2>0$, the series
$\displaystyle{\sum_{n=0}^\infty k_n(t,x,y)}$
 converges  on $(0,T_2]\times E\times E$ with the same type of upper bound as that  in Lemma \ref{L:3.16}.
\end{lem}
\begin{proof}
Using an argument   similar to \cite[Lemmas 15 \& 16]{JS}, we can show that $\sum_{n=0}^\infty k_n(t,x,y)$ satisfies Chapman-Kolmogorov equation as follows:
\begin{equation}\label{C-K-equation3.94}
\int_{z\in E} \sum_{i\ge 0}k_i(s, x, z)\sum_{j\ge 0}k_j(t-s, z, y)dz =\sum_{n\ge 0}k_n(t,x,y), \quad \text{for }0<s, t-s\le T_1,
\end{equation}
where $T_1$ is the same as in Lemma \ref{L:3.16}.  The details are omitted here. 
Furthermore, taking $s=T_1$ and $T_1<t\le 2T_1$ in \eqref{C-K-equation3.94} implies that  for some $c_1, c_2>0$, $\sum_{n\ge 0}k_n(t,x,y)$ converges  on $0<t\le 2T_1$ with  the following upper bound:
\begin{equation*}
\left|\sum_{n\ge 0}k_n(t,x,y)\right| \le c_1^2 p (c_2t,x,y), \quad \text{on }T_1<t\le 2T_1.
\end{equation*}
Repeating the same procedure, we know there exist $c_3, c_4, c_5>0$ such that
\begin{equation}\label{UP-k_n}
\left|\sum_{n\ge 0}k_n(t,x,y)\right|\le c_3e^{c_4t}p(c_2t,x,y)\le c_3 e^{c_4t}p^0_{0, c_5}(t,x,y), \quad \text{for all }t>0.  
\end{equation}
With the above upper bound  for $\sum_n k_n(t,x,y)$, we now claim 
\begin{equation*}
p^b(t,x,y):=\sum_{n=0}^\infty k_n(t,x,y), \quad \text{for }t>0, x,y\in E,
\end{equation*}
 is indeed the transition density of $X$. We claim this  by showing that for sufficiently large $\alpha>0$,  the Laplace transform of  $p^b(t,x,y)$ is well-defined and indeed equals  the kernel of $G_\alpha^b$.

 For $\alpha>2c_4$ in \eqref{UP-k_n}, $n=0,1,...,$ and $x,y\in E$,  we define $u^{(n)}_\alpha (x,y)$ as the Laplace transform of $k_n(t,x,y)$ as follows:
\begin{equation*}
u_\alpha^{(n)}(x,y):=\int_0^\infty e^{-\alpha t}k_n(t,x,y) dt.
\end{equation*}
We can  therefore define for $\alpha>2c_4$ that 
\begin{equation}\label{resolventkernel}
u^b_\alpha (x,y):=\sum_{n=0}^\infty u_\alpha^{(n)}(x,y)=\sum_{n=0}^\infty \int_0^\infty e^{-\alpha t}k_n(t,x,y)dt =\int_0^\infty e^{-\alpha t}\left(\sum_{n=0}^\infty k_n(t,x,y)\right)dt,
\end{equation}
where the infinite sum converges in view of \eqref{UP-k_n}.
Using the same argument as that for \cite[Lemma 18]{BJ}, one can show for  $\alpha>2c_4$ that $G_\alpha^0(b\cdot\nabla  G_\alpha^0)^n$ has kernel $u_\alpha^{(n)}(x,y)$.
In view of Theorem \ref{T:2.1},  it holds for any $\alpha>\max\{2c_4, \alpha_0\}$ ($\alpha_0$ is chosen in Theorem \ref{T:2.1}) that 
\begin{align*}
G_\alpha^b f(x)&\stackrel{\text{Theorem }\ref{T:2.1}}{=}\sum_{n=0}^\infty G_\alpha^0(b\cdot \nabla G_\alpha^0)^nf=\sum_{n=0}^\infty\int_{E} u_\alpha^{(n)}(x,y)f(y)dy
\\
&=\int_E\int_0^\infty e^{-\alpha t}p^b(t,x,y)f(y)dt dy, \quad \forall f\in L^2(E)\cap L^\infty (E), \forall x\in E.
\end{align*}
Thus the same relationship as above holds for all  $f\in L^2(E)$.   This implies that when $\alpha>\max\{2c_4, \alpha_0\}$, for all $x,y\in E$, $G_\alpha^b$ has   kernel  $\int_0^\infty e^{-\alpha t}p^b(t,x, y )dt$.
 i.e.,  $\{p^b(t,x,y)\}_{(t,x,y)\in \IR_+\times E\times  E}$ is indeed the transition density of $X$.

Finally, to complete the proof, we need to justify the last statement of the lemma. For this, we only need to show $p^b(t,x,y)$ has the same type of upper bounds as those in Lemma \ref{L:3.16} on $(0, 2T_1]$.  Since $p^b(t,x,y)$ satisfies Chapman-Kolmogorov equation \eqref{C-K-equation3.94},  there exist some constants $c_6, c_7, c_8>0$ such that 
\begin{align*}
p^b(t,x,y)&=\int_{ E} p^b(t/2, x,z)p^b(t/2, z,y)m_p(dz) \lesssim \int_E p^0_{0, c_6}(t/2, x,z)p^0_{0, c_6}(t/2, z,y)m_p(dz)
\\
&\lesssim \int_E p(c_7t, x,z)p(c_7t, z,y)m_p(dz) =p(2c_7t, x, y) \lesssim p^0_{0, c_8}(t,x,y), \qquad t\in (0, 2T_1].
\end{align*}
Therefore  for any $T_2>0$  in the    statement of the   lemma, the same type of upper bound holds for $p^b(t,x,y)$ on $t\in (0, T_2]$  by repeating the same process finitely  many  times. 
\end{proof}

We are now in the position to  present the proof to Theorem \ref{T:smalltime}.

\medskip 
 
\noindent{\bf Proof of Theorem \ref{T:smalltime}.}
For all  the three cases, the upper bounds follow  immediately from the conjunction of Lemmas \ref{L:3.16} and \ref{L:3.17}. Therefore, we only need to establish the lower bounds. Furthermore, the low bounds only need to be established  for some small $T=t_1>0$, because from there, the lower bounds  for an arbitrary $T>0$ follow immediately from Chapman-Kolmogorov equation.  Note that the constant $c_3$   in the proof to Lemma \ref{L:3.16}  is monotonically decreasing in $t$, and similar property holds for all other cases in that proof. 
Thus in view of the inequality
\begin{equation*}
p^b(t,x,y)=\sum_{n=0}^\infty k_n(t,x,y)\ge p(t,x,y)-\sum_{n=1}^\infty |k_n(t,x,y)|.
\end{equation*}
 and  Lemma \ref{L:3.17},  we may pick $t_1\in (0, 1]$ sufficiently small such that for some $c_1>0$, it holds  for all $t\in (0, t_1]$ and all $x, y:\,\rho(x,y)<2\sqrt{t}$:
\begin{description}
\item{(i)}  $p^b(t,x,y)\ge c_1/\sqrt{t}$ if $x\in \IR_+$ or $y\in \IR_+$.
\item{(ii)} $p^b(t,x,y)\ge c_1/t$ if $x,y\in D_\eps$ with  $\min\{|x|_\rho, |y|\}\ge 2$.
\item{(iii)} $p^b(t,x,y)\ge \frac{c_1}{\sqrt{t}}+\frac{c_1}{t} (1\wedge \frac{|x|_\rho}{\sqrt{t}})(1\wedge \frac{|y|_\rho}{\sqrt{t}})$ otherwise.
\end{description}
With the above ``near-diagonal'' estimates, in order to get the ``off-diagonal'' estimates, we again divide the computation into several parts depending on the positions of $x$ and $y$. In the following computation, it is always assumed $|x-y|\ge \rho(x,y)\ge 2\sqrt{t}$, as otherwise it is covered by ``near-diagonal" estimate  (i)-(iii) above. Throughout this proof, we fix a sufficiently large constant $\lambda>64$ which will be used in the rest of the proof. In the  following,  the values of the positive constants $c_i>0,  i=1,2,\dots$ are  not important and thus may change from line to line. 
\\
{\it Case 1.} $x,y\in \IR_+$. Let $m$ be the smallest integer such that $m\ge \lambda |x-y|^2/t\ge 4\lambda$.  Let $y_0=x$, $y_m=y$. It follows that
\begin{align}
p^b(t,x,y)&\ge \int_{\IR_+\cap \{|y_k-y_{k-1}|\le \frac{1}{4}\sqrt{\frac{t}{m}}, k=1,\cdots , m\}} p^b(t/m, x,y_1)\cdots p^b(t/m, y_{m-1}, y)dy_1\cdots dy_{m-1}\nonumber
\\
&\ge \left(\frac{c_1}{\sqrt{t/m}}\right)^m \left(\frac{1}{2}\sqrt{\frac{t}{m}}\right)^{m-1}\gtrsim \frac{\sqrt{m}c_1^m}{\sqrt{t}}\gtrsim \frac{1}{\sqrt{t}}e^{-c_2|x-y|^2/t}, \label{327}
\end{align}
where the last inequality is due to the fact that $m\asymp |x-y|^2/t$.
\\
{\it Case 2.} $x\in \IR_+$, $y\in D_\eps \cup \{a^*\}$ with $|y|_\rho<\sqrt{t}$. Note that in this case, $|x| \le |x|+|y|_\rho= \rho(x,y)$. On the set $\{y_1: \;y_1\in \IR_+, |y_1|<\sqrt{t}/4\}$, it holds $p^b(3t/4, y_1, y)\ge c_1/\sqrt{t}$ by the ``near-diagonal" estimate. Also, the result for Case 1 yields, for $|y_1|<\sqrt{t}/4$,
\begin{equation}\label{e:3.95}
p^b(t/4, x, y_1)\gtrsim \frac{1}{\sqrt{t/4}}e^{-4c_2|x-y_1|^2/t}\gtrsim \frac{1}{\sqrt{t}}e^{-8c_2(|x|^2+|y_1|^2)/t}\gtrsim \frac{1}{\sqrt{t}}e^{-8c_2|x|^2/t}.
\end{equation}
Hence,
\begin{align}\label{e:3.110}
p^b(t,x,y)&\ge \int_{y_1\in \IR_+, |y_1|<\sqrt{t}/4}p^b(t/4, x,y_1)p^b(3t/4, y_1, y) dy_1 \nonumber
\\
&\gtrsim 
\frac{1}{\sqrt{t}}e^{-8c_2|x|^2/t}\cdot \frac{c_1}{\sqrt{t}} \cdot \frac{\sqrt{t}}{4} \gtrsim \frac{1}{\sqrt{t}}e^{-8c_2|x|^2/t}\ge \frac{1}{\sqrt{t}}e^{-8c_2(|x|+|y|_\rho)^2/t}\gtrsim \frac{1}{\sqrt{t}}e^{-c_3\rho(x,y)^2/t}.
\end{align}
{\it Case 3.} $x\in \IR_+, y\in D_\eps$ with  $|y|_\rho>\sqrt{t}$. Again, for the fixed $\lambda$, we let $m$ be the smallest integer such that $m\ge \lambda |y|_\rho^2/t>\lambda$. Let $y_0=x$ and $y_{m+1}=y$.  Set a region $A_1\subset (D_\eps)^m$ (the product space of $D_\eps$) as follows:
$A_1:=\{(y_1, \cdots, y_m)\in (D_\eps)^m: |y_1|_\rho\le \sqrt{t}/4,\,  |y_k-y_{k-1}|<\frac{1}{4}\sqrt{\frac{t}{m}},\, k=2, \cdots, m+1; \, |y_k|_\rho>\sqrt{t}/8 \text{ for all }k=1,2,\dots, m\}$.
 We note that $m_p(A_1)\asymp (t/m)^{m}$ (here we abuse the notation a little bit  by letting $m_p$ be the product measure on $D_\eps^m$, which is equivalent to Lebesgue measure on $\IR^m$).    Denote by $W^{b_2}$ a  two-dimensional Brownian motion with drift $b_2$, and denote by $W^{b_2}_{D_\eps}$ the killed process of $W^{b_2}$ upon exiting $D_\eps$. Let $p^{W^{b_2}}_{D_\eps}$  be the transition density of $W^{b_2}_{D_\eps}$.  Since $X$ and $W^{b_2}$ have the same distribution before exiting $D_\eps$,  in view of  the Dirichlet heat kernel estimate for $p^{W^{b_2}}_{D_\eps}$ (see \cite[Theorem 1.1]{CKP}),   it holds
\begin{equation}\label{e:3.97}
p^b(t/2m, y_{k-1}, y_{k})\ge p^b_{D_\eps}(t/2m, y_{k-1}, y_{k}) = p^{W^{b_2}}_{D_\eps}(t/2m, y_{k-1}, y_{k})\gtrsim \frac{m}{t}, \, k=2,\cdots, m+1.
\end{equation}
Applying \eqref{e:3.110} to $p^b(t/2, x,y_1)$ and noting $m_p(A_1)\asymp (t/m)^m$, we have
\begin{align}
p^b(t,x,y)&\ge \int_{A_1} p^b(t/2, x,y_1)p^b(t/2m, y_1, y_2)\cdots p^b(t/2m, y_{m}, y)m_p(dy_1)\cdots m_p( dy_{m})\nonumber
\\
&\gtrsim \frac{1}{\sqrt{t}}e^{-8c_2|x|^2/t}\left(\frac{c_4m}{t}\right)^{m}\cdot \left(\frac{t}{m}\right)^{m}  \nonumber
\\
&\gtrsim \frac{c_4^m}{\sqrt{t}}e^{-8c_2|x|^2/t}\gtrsim \frac{1}{\sqrt{t}}e^{-c_5(|x|^2+|y|_\rho^2)/t}\gtrsim \frac{1}{\sqrt{t}}e^{-c_6\rho(x,y)^2/t}, \label{211}
\end{align}
where the  second last ``$\gtrsim$'' in \eqref{211} is due to the fact that $m\asymp |y|_\rho^2/t$.
\\
{\it Case 4. }$x\in D_\eps \cup \{a^*\}$, $y\in \IR_+$. It is easy to see that this case can be handled following the exact same argument as Cases 2-3 by switching the roles of $x$ and $y$. 
\\
{\it Case 5. } $x,y\in D_\eps $ with $\max\{|x|_\rho, |y|_\rho\}>2$ and $\min\{|x|_\rho, |y|_\rho\}>\sqrt{t}/2$. As in Case 3, we denote by $W^{b_2}$ a  two-dimensional Brownian motion with drift $b_2$, and denote by $W^{b_2}_{D_\eps}$ the killed process of $W^{b_2}$ upon exiting $D_\eps$. Let $p^{W^{b_2}}_{D_\eps}$  be the transition density of $W^{b_2}_{D_\eps}$.  Since $X$ and $W^{b_2}$ have the same distribution before exiting $D_\eps$, on account of  the Dirichlet heat kernel estimate in \cite[Theorem 1.1]{CKP},   it holds
\begin{align}
p^b(t,x,y)&\ge p^b_{D_\eps}(t,x,y)=p^{W^{b_2}}_{D_\eps}(t,x,y)\gtrsim \frac{1}{t}\left(1\wedge \frac{|x|_\rho}{\sqrt{t}}\right)\left(1\wedge \frac{|y|_\rho}{\sqrt{t}}\right)e^{-c_7|x-y|^2/t} \nonumber
\\
&\ge  \frac{1}{4t} e^{-c_7|x-y|^2/t}\ge \frac{1}{4t} e^{- c_8\rho(x,y)^2/t}, \label{e:3.99}
\end{align}
where the  second last inequality follows from the fact that  $\min\{|x|_\rho, |y|_\rho\}>\sqrt{t}/2$, and the last inequality is due to the fact that $\rho(x,y)\ge |x-y|/2$ when $\max\{|x|_\rho, |y|_\rho\}> 2$.
\\
{\it Case 6. } $x,y\in D_\eps \cup\{a^*\}$ with $\max\{|x|_\rho, |y|_\rho\}>2$ and $\min\{|x|_\rho, |y|_\rho\}\le \sqrt{t}/2$. Similar to Case 5, for this case it also holds $\rho(x,y)\ge |x-y|/2$. Without loss of generality, we assume $|x|_\rho<\sqrt{t}/2$ and $|y|_\rho>2$. Let $A_2:=\{y_1\in D_\eps: \sqrt{t}/2<|y_1|_\rho<\sqrt{t}\}$. Thus
\begin{equation}\label{e:3.100}
m_p(A_2)=\pi\left((\sqrt{t}+\eps)^2-(\frac{\sqrt{t}}{2}+\eps)^2\right)=\pi \left(\frac{3t}{4}+\eps\sqrt{t}\right) \asymp \sqrt{t}.
\end{equation}
Also for $y_1\in A_2$, since $t\le t_1\le 1$, we have
\begin{equation}\label{e:3.101}
|y-y_1|\le |y|+|y_1|\le |y|+\sqrt{t}+\eps \le 2|y|.
\end{equation}
Applying \eqref{e:3.99} to $p^b(t/4, y_1, y)$ and in view of the ``near-diagonal" estimate, we have 
\begin{align*}
p^b(t,x,y)&\ge \int_{A_2}p^b(3t/4, x,y_1)p^b(t/4, y_1, y)m_p(dy_1) \nonumber
\\
(\rho(x,y_1)\le 3\sqrt{t}/2 )&\gtrsim  \int_{A_2}  \frac{1}{\sqrt{t}}\cdot \frac{1}{t}e^{-c_7|y-y_1|^2/t} m_p(dy_1) \nonumber
\\
(\eqref{e:3.100}, \eqref{e:3.101} )&\gtrsim \frac{1}{t}e^{-4c_7|y|^2/t}\gtrsim \frac{1}{t}e^{-c_9\rho(x,y)^2/t},
\end{align*}
where the last ``$\gtrsim$" follows from the fact that $|x|_\rho<\sqrt{t}/2<1/2$ and $|y|_\rho>2$, so $|y|\asymp \rho(x,y)$ by elementary geometry.
\\
{\it Case 7. } $x,y\in D_\eps \cup \{a^*\}$ with $\max\{|x|_\rho, |y|_\rho\}<2$.   It again follows from Dirichlet heat kernel estimate in \cite[Theorem 1.1]{CKP} that
\begin{align}
p^b(t,x,y)&=p^b_{D_\eps}(t,x,y)+\bar{p}^b_{D_\eps}(t,x,y)\nonumber
\\
&\gtrsim \frac{1}{t}\left(1\wedge \frac{|x|_\rho}{\sqrt{t}}\right)\left(1\wedge \frac{|y|_\rho}{\sqrt{t}}\right)e^{-c_7|x-y|^2/t}+\bar{p}^b_{D_\eps}(t,x,y). \label{e:3.117}
\end{align}
Therefore, it suffices to show the following lower bound  for some $c>0$:
\begin{equation}\label{329}
\bar{p}^b_{D_\eps}(t,x,y)\gtrsim  \frac{1}{\sqrt{t}}e^{-c(|x|^2_\rho+|y|^2_\rho)/t}.
\end{equation}
Indeed, using the results for Cases 2-4, we have
\begin{align*}
\bar{p}^b_{D_\eps}(t,x,y)&\ge \int_{\{y_1\in \IR_+, |y_1|<\sqrt{t}/16\}} p^b(t/2,x, y_1)p^b(t/2, y_1, y)dy_1 \nonumber
\\
&\gtrsim\int_{\{y_1\in \IR_+, |y_1|<\sqrt{t}/16\}} \frac{1}{\sqrt{t}} e^{-c_{10}(|x|_\rho^2+|y_1|^2)/t}\cdot \frac{1}{\sqrt{t}}e^{-c_{10}(|y_1|^2+|y|^2_\rho)/t} dy_1\nonumber
\\
&\gtrsim\int_{\{y_1\in \IR_+, |y_1|<\sqrt{t}/16\}} \frac{1}{\sqrt{t}} e^{-c_{10}|x|_\rho^2/t}\cdot \frac{1}{\sqrt{t}}e^{-c_{10}|y|^2_\rho/t} dy_1\nonumber
\\
&\gtrsim \frac{1}{\sqrt{t}} e^{-c_{10}|x|_\rho^2/t}\cdot \frac{1}{\sqrt{t}}e^{-c_{11}|y|^2_\rho/t} \cdot \sqrt{t} \gtrsim \frac{1}{\sqrt{t}}e^{-c_{10}(|x|^2_\rho+|y|^2_\rho)/t}.
\label{328}
\end{align*}
This establishes \eqref{329},  which on account of   \eqref{e:3.117} proves the desired lower bound for the current case. The proof to the theorem is therefore complete. \qed

\section{Green Function Estimate for Drifted BMVD}\label{S:4}

\indent In this section, we establish two-sided bounds for the Green function of drifted  BMVD 
$X$ killed upon exiting a bounded connected $C^{1,1}$ open set $D\subset E$. 
Recall that the Green function $G_D(x,y)$ is defined as follows:
\begin{equation*}
G^b_D(x,y)=\int_0^\infty p^b_D(t,x,y)dt,
\end{equation*}
where $p^b_D(t, x, y)$ is the transition density function of the part process $X^{D}$  killed upon exiting $D$.
 We assume $a^*\in D$ throughout this section, as otherwise, due to the connectedness of $D$, either $D\subset \IR_+$ or $D\subset D_\eps$. Therefore $G^b_D(x,y)$ is just the  Green function of a bounded $C^{1,1}$ domain  for Brownian motion 
 with drift in one-dimensional or two-dimensional  spaces, whose two-sided estimates are known, see \cite{CKP}.  
It is not  hard to see from 
$$ p^b_D(t, x, y)=p^b(t, x, y)-\IE_x [ p^b(t-\tau_D, X_{\tau_D}, y); \tau_D <t],
$$
that $p^b_D(t, x, y)$ is jointly continuous in $(t, x, y)$.

 
 \medskip



\medskip 
 
\noindent{\bf Proof of Theorem \ref{T:1.10}.}
The proof to this theorem is almost identical to the proof to \cite[Theorem 1.5]{CL}, except for some minor changes. The details are omitted here. \qed

\section*{Acknowledgment}
I am very grateful to Professor  Zhen-Qing Chen for his constant    support, including the great amount of  helpful   advice and  informative discussions. I also  thank the referee for the  very careful corrections and  detailed comments, which has significantly improved the presentation of this paper.

\vskip 0.3truein

\noindent {\bf Shuwen Lou}

\smallskip \noindent
Department of Statistical Sciences

\noindent
University of Toronto

\noindent
Toronto, Ontario, M5S 3G3, Canada

\noindent
E-mail:  \texttt{shuwen.lou@utoronto.ca}

\end{document}